\documentclass[reqno,11pt,letterpaper]{amsart}
\UseRawInputEncoding
\usepackage{mathrsfs}
\usepackage{amssymb}
\usepackage[usenames,dvipsnames]{xcolor}
\usepackage{amsthm}
\usepackage{amsmath}
\usepackage{amsfonts}
\usepackage{enumerate}
\usepackage{txfonts}
\usepackage{bm}

\usepackage{graphicx}

\numberwithin{equation}{section}

\newtheorem{theorem}{Theorem}[section]
\newtheorem{lemma}[theorem]{Lemma}
\newtheorem{corollary}[theorem]{Corollary}
\newtheorem{proposition}[theorem]{Proposition}

\theoremstyle{definition}
\newtheorem{definition}[theorem]{Definition}
\newtheorem{assumption}[theorem]{Assumption}
\newtheorem{example}[theorem]{Example}

\theoremstyle{remark}
\newtheorem{remark}[theorem]{Remark}

\allowdisplaybreaks[1]

\begin{document}

\title[Spherical transonic shock]{Three dimensional spherical transonic shock in a hemispherical shell}

\author{Shangkun WENG}
\address{School of Mathematics and Statistics, Wuhan University, Wuhan, 430072, Hubei Province, China.}
\email{skweng@whu.edu.cn}


\date{}
\maketitle

\def\be{\begin{eqnarray}}
\def\ee{\end{eqnarray}}
\def\ba{\begin{aligned}}
\def\ea{\end{aligned}}
\def\bay{\begin{array}}
\def\eay{\end{array}}
\def\bca{\begin{cases}}
\def\eca{\end{cases}}
\def\p{\partial}
\def\no{\nonumber}
\def\e{\epsilon}
\def\de{\delta}
\def\De{\Delta}
\def\om{\omega}
\def\Om{\Omega}
\def\f{\frac}
\def\th{\theta}
\def\la{\lambda}
\def\lab{\label}
\def\b{\bigg}
\def\var{\varphi}
\def\na{\nabla}
\def\ka{\kappa}
\def\al{\alpha}
\def\La{\Lambda}
\def\ga{\gamma}
\def\Ga{\Gamma}
\def\ti{\tilde}
\def\wti{\widetilde}
\def\wh{\widehat}
\def\ol{\overline}
\def\ul{\underline}
\def\Th{\Theta}
\def\si{\sigma}
\def\Si{\Sigma}
\def\oo{\infty}
\def\q{\quad}
\def\z{\zeta}
\def\co{\coloneqq}
\def\eqq{\eqqcolon}
\def\di{\displaystyle}
\def\bt{\begin{theorem}}
\def\et{\end{theorem}}
\def\bc{\begin{corollary}}
\def\ec{\end{corollary}}
\def\bl{\begin{lemma}}
\def\el{\end{lemma}}
\def\bp{\begin{proposition}}
\def\ep{\end{proposition}}
\def\br{\begin{remark}}
\def\er{\end{remark}}
\def\bd{\begin{definition}}
\def\ed{\end{definition}}
\def\bpf{\begin{proof}}
\def\epf{\end{proof}}
\def\bex{\begin{example}}
\def\eex{\end{example}}
\def\bq{\begin{question}}
\def\eq{\end{question}}
\def\bas{\begin{assumption}}
\def\eas{\end{assumption}}
\def\ber{\begin{exercise}}
\def\eer{\end{exercise}}
\def\mb{\mathbb}
\def\mbR{\mb{R}}
\def\mbZ{\mb{Z}}
\def\mc{\mathcal}
\def\mcS{\mc{S}}
\def\ms{\mathscr}
\def\lan{\langle}
\def\ran{\rangle}
\def\lb{\llbracket}
\def\rb{\rrbracket}
\begin{abstract}
  The existence and stability of a spherical transonic shock in a hemispherical shell under the three dimensional perturbations of the incoming flows and the exit pressure is established without any further restrictions on the background transonic shock solutions. The perturbed transonic shock are completely free and its strength and position are uniquely determined by the incoming flows and the exit pressure. A key issue in the analysis is the ``spherical projection coordinates" (i.e. the composition of the spherical coordinates and the stereographic projection), which provides an appropriate setting for the spherical transonic shock problem in the sense that the transformed equations have a similar structure as the steady Euler equations and do not contain any coordinates singularities. Then we decompose the hyperbolic and elliptic modes in the steady Euler equations in terms of the deformation and vorticity. An elaborate reformulation of the Rankine-Hugoniot conditions yields an unusual second order differential boundary condition on the shock front to the first order nonlocal deformation-curl system, from which an oblique boundary condition can be derived after homogenizing the curl system and introducing the potential function. The analysis of the compatibility conditions at the intersection of the shock front and the shell boundary is crucial for the optimal regularity of all physical quantities.
\end{abstract}

\section{Introduction and main results}

In this paper, we concern the existence and stability of the spherical transonic shock in the divergent part of De Laval nozzles whose cross section decreases first and then increases. Courant and Friedrichs \cite{cf48} had described that if the
upcoming flow becomes supersonic after passing through the throat of a De Laval nozzle, to match the prescribed appropriately large exit pressure, a shock front intervenes at some place in the diverging part of the nozzle and the gas is compressed and slowed down to subsonic speed. This transonic shock phenomena can be described by the steady compressible Euler equations
\be\label{com-euler}
\begin{cases}
\text{div }(\rho {\bf u})=0,\\
\text{div }(\rho {\bf u}\otimes {\bf u}+ p I_3) =0,\\
\text{div }(\rho (\f12|{\bf u}|^2 +e) {\bf u}+ P{\bf u}) =0.
\end{cases}
\ee
where ${\bf u}=(u_1, u_2,u_3)$, $\rho$, $P$, and $e$ stand for the velocity, density, pressure, and internal energy, respectively. For polytropic gases, the equation of state and the internal energy take the form
\begin{equation}\no
P=K(S)\rho^{\ga}=A  e^{\frac{S}{c_v}}\rho^{\gamma}\quad \text{and}\quad
e=\f{P}{(\ga-1)\rho},\ \ \ \gamma>1,
\end{equation}
respectively, where $\gamma> 1$, $A$, and $c_v$ are positive constants, and $S$ is called the specific entropy. It is well-known that the system \eqref{com-euler} is  hyperbolic for supersonic flows ($M_a>1$),  hyperbolic-elliptic  coupled   for subsonic flows ($M_a<1$), and degenerate at sonic point (i.e. $M_a=1$), where $M_a=\frac{|{\bf u}|}{c(\rho,K)}$ is the Mach number of the flow with $c(\rho,K)=\sqrt{\partial_\rho P(\rho, K)}$ being the local sound speed. The quantity $B=\frac{|{\bf u}|^2}{2}+e+\frac{P}{\rho}$ is called the Bernoulli's quantity.

To study the spherical transonic shock flow, people usually use the spherical coordinates
\be\no
x_1 =r\cos\theta,\q x_2=r\sin\theta\cos\varphi,\quad x_3=r\sin\theta\sin\varphi.
\ee
and decompose the velocity ${\bf u}=U_r {\bf e}_r+ U_{\theta}{\bf e}_{\th} + U_{\varphi}{\bf e}_{\var}$, where
\be\no
&&{\bf e}_r =(\cos\th,\sin\theta \cos\var,\sin\theta \sin\var)^t,\q {\bf e}_{\th}=(-\sin\th,\cos\theta \cos\var,\cos\th \sin\var)^t,\\\no
&&{\bf e}_{\var}=(0,-\sin\var,\cos\var)^t.
\ee
Then the steady compressible Euler equations can be rewritten as
\be\lab{euler-sph}\begin{cases}
\p_r U_r +\frac{1}{r}\p_{\theta} U_{\theta} + \frac{1}{r\sin\theta}\p_{\varphi} U_{\varphi} +\frac{2}{r} U_r \\
\q\q\q+\frac1{r} U_{\theta}\cot\theta+ \frac 1{\rho}(U_r \p_r +\frac{U_{\theta}}{r} \p_{\theta} +\frac{U_{\varphi}}{r\sin\theta} \p_{\varphi})\rho=0,\\
(U_r \p_r +\frac{U_{\theta}}{r} \p_{\theta} +\frac{U_{\varphi}}{r\sin\theta} \p_{\varphi})U_r + \frac 1{\rho} \p_r P- \frac{(U_{\theta}^2+U_{\varphi}^2)}{r}=0,\\
(U_r \p_r +\frac{U_{\theta}}{r} \p_{\theta} +\frac{U_{\varphi}}{r\sin\theta} \p_{\varphi})U_{\theta}+\frac{1}{\rho}\frac{1}{r}\p_{\theta} P +\frac{U_r U_{\theta}}{r}-\frac{U_{\varphi}^2}{r}\cot\theta=0,\\
(U_r \p_r +\frac{U_{\theta}}{r} \p_{\theta} +\frac{U_{\varphi}}{r\sin\theta} \p_{\varphi})U_{\varphi}+\frac{1}{\rho}\frac{1}{r\sin\theta}\p_{\varphi} P+\frac{U_r U_{\varphi}}{r}+\frac{U_{\theta} U_{\varphi}}{r}\cot\theta=0,\\
(U_r \p_r +\frac{U_{\theta}}{r} \p_{\theta} +\frac{U_{\varphi}}{r\sin\theta} \p_{\varphi}) B=0.
\end{cases}\ee
It was shown that there exists a class of spherical symmetric transonic shock solution in $\Omega$, which can be stated as follows.
\begin{proposition}\label{background}
{\it Suppose that a spherically symmetric supersonic flow ${\bf u}^-(x)=\bar{U}^-(r_1){\bf e}_r, \bar{\rho}^-(x)=\bar{\rho}^-(r_1)>0, \bar{K}^-(x)=\bar{K}^->0$ is prescribed at $r=r_1$, where $(\bar{U}^-(r_1))^2>c^2(\bar{\rho}^-(r_1),\bar{K}^-)$. Then there exists two positive constants $P_1$ and $P_2$ depending only on the incoming supersonic flows and $r_1,r_2$, such that when the exit pressure $P_e\in (P_1, P_2)$, there exists a unique spherical symmetric transonic shock solution
\be\no\begin{cases}
({\bar {\bf u}}^-,\bar{P}^-,\bar{K}^-)(r,\theta,\varphi):=(\bar{U}^-(r){\bf e}_r, \bar{P}^-(r),\bar{K}^-),\ \ &\text{in }\Omega^b_-=(r_1,r_s)\times E_0,\\
({\bar {\bf u}}^+,\bar{P}^+,\bar{K}^+)(r,\theta,\varphi):=(\bar{U}^+(r){\bf e}_r, \bar{P}^+(r),\bar{K}^+),\ \ &\text{in }\Omega^b_+=(r_s,r_2)\times E_0
\end{cases}\ee
to \eqref{euler-sph} satisfying the exit pressure condition
\be\no
\bar{P}^+(r_2)=P_e
\ee
with a shock front located at $r=r_s\in (r_1,r_2)$. Across the shock, the Rankine-Hugoniot conditions and the physical entropy condition are satisfied:
\be\no
[\bar{\rho} \bar{U}]=0,\quad [\bar{\rho} \bar{U}^2+\bar{P}]=[\bar{B}]=0,\quad \bar{K}^+>\bar{K}^-,
\ee
where $[f]|_{r=r_s}:=f(r_s+)-f(r_s-)$ denotes the jump of $f$ at $r=r_s$. }
\end{proposition}

This special solution, $\overline{\bm{\Psi}}$, will be called the background solution. Clearly, one can extend the supersonic and subsonic parts of $\overline{\bm{\Psi}}$ in a natural way, respectively. With an abuse of notations, we still call the extended subsonic and supersonic solutions ${\overline{\bm{\Psi}}}^+$ and ${\overline{\bm{\Psi}}}^-$, respectively.  For detailed properties of this spherical symmetric transonic shock solution, we refer to \cite[Section 147]{cf48} or \cite[Theorem 1.1]{xy08b}. For the structural stability of the transonic shock problem in flat nozzles, one may refer to \cite{cf03,ccf07,chen05,chen08,fx21,xy05,xy08a}. For the structural stability of the transonic shock problem in divergent sectors and conic nozzles, see the references \cite{bf11,lxy09a,lxy09b,lxy10b,lxy13,lxy16,wxx,wxy21a,wx23a}. The main goal of this paper is to establish the structural stability of this spherical symmetric transonic shock solution under three dimensional perturbations of the incoming supersonic flows and the exit pressure.

To avoid the corner singularity near the intersection of the shock front and the nozzle wall, we assume the flow region to be a hemispherical shell $\Omega=\{(x_1,x_2,x_3)\in \mathbb{R}^3: x_1>0,\  r_1<r=\sqrt{x_1^2+x_2^2+x_3^2}<r_2\}$ so that a symmetric extension technique can be used to improve the regularity of flows. The flow region becomes $\Omega:=\{(r,\theta,\varphi): r_1<r<r_2,(\theta,\varphi)\in E_0\}$ with $E_0=[0,\frac{\pi}{2}]\times \mathbb{T}_{2\pi}$, where $\mathbb{T}_{2\pi}$ is a torus with period $2\pi$. The existence and uniqueness of cylindrical transonic shock flows under three dimensional perturbations has been established recently in \cite{wx23a}. Due to the more complicated structure of the nozzle (comparing with the cylindrical domain in \cite{wx23a}), the structural stability of the spherical transonic shock problem requires some new ideas.

Now we give a detailed description of our stability problem. It turns out that an appropriate choice of the coordinates will play a crucial role in the solvability of the problem. There are several ``singular terms" in \eqref{euler-sph} due to the spherical coordinates, it seems quite difficult to resolve the singularity especially when one solves the transport equations for the Bernoulli's quantity and the entropy. Thus the spherical coordinates do not provide us an appropriate problem setting.

It is well-known that the stereographic projection will map the sphere (excluding the projection point) one by one and onto the complex plane $\mathbb{C}$, and transform the Laplace-Beltrami operator $\Delta_{\mathbb{S}}=\p_{\theta}^2+\frac{\cos\theta}{\sin\theta}\p_{\theta}+\frac{1}{\sin^2\theta}\p_{\varphi}^2$ on the sphere to the standard Laplace operator on $\mathbb{C}$ up to a positive factor, see \cite{gm18} and the references therein. Inspired by these facts, we combine the spherical coordinates and the stereographic projection and introduce the ``spherical projection coordinates"\footnote{It is possible that someone in the history had used these coordinates, but the author can not locate it in the literatures.} as follows: $\mathscr{S}: x\in \Omega\rightarrow z\in \mathcal{N}:=(r_1,r_2)\times \{z'=(z_2,z_3):|z'|<1\}$ by
\be\label{st1}
&&z_1=r=\sqrt{x_1^2+x_2^2+x_3^2},\ \ z_2 = \frac{\sin \theta \cos\varphi}{1+\cos\theta}=\frac{x_2}{x_1+\sqrt{x_1^2+x_2^2+x_3^2}},\\\no
&&z_3= \frac{\sin \theta \sin\varphi}{1+\cos\theta}=\frac{x_3}{x_1+\sqrt{x_1^2+x_2^2+x_3^2}}.
\ee
The transform $\mathscr{S}$ is a one-to-one and onto smooth mapping from $\Omega$ to $\mathcal{N}$ and its inverse mapping $\mathscr{S}^{-1}:z\in \mathcal{N}\rightarrow x\in \Omega$ is
\be\no
x_1=\frac{1-|z'|^2}{1+|z'|^2} z_1,\ \ \ x_2=\frac{2 z_1 z_2}{1+|z'|^2},\ \ \ \ x_3=\frac{2 z_1 z_3}{1+|z'|^2}.
\ee

Define the vector field $(U_1, U_2,U_3)$ as
\be\no
U_1=U_r,\ \ U_{\theta}= U_2\cos \varphi+ U_3 \sin \varphi,\ \ \ U_{\varphi}=-U_2\sin\varphi +U_3 \cos\varphi,
\ee
then
\be\no
{\bf u}(x)= U_r {\bf e}_r + U_{\theta} {\bf e}_{\theta} + U_{\varphi} {\bf e}_{\varphi}= U_1(z) \tilde{{\bf e}}_1(z) +U_2(z) \tilde{{\bf e}}_2(z) +U_3(z) \tilde{{\bf e}}_3(z),
\ee
with
\be\label{basis}\begin{cases}
\tilde{{\bf e}}_1(z)={\bf e}_r=\frac{1}{1+|z'|^2}(1-|z'|^2, 2z_2,2z_3)^t,\ \ \ \ \ z'=(z_2,z_3),\\
\tilde{{\bf e}}_2(z)=\cos\varphi {\bf e}_{\theta}- \sin\varphi {\bf e}_{\varphi}=\frac{1}{1+|z'|^2}(-2z_2,1+z_3^2-z_2^2,-2z_2z_3)^t,\\
\tilde{{\bf e}}_3(z)=\sin\varphi {\bf e}_{\theta}+\cos\varphi {\bf e}_{\varphi}=\frac{1}{1+|z'|^2}(-2z_3,-2z_2z_3,1+z_2^2-z_3^2)^t.
\end{cases}\ee
It is noted that $\tilde{{\bf e}}_i, i=1,2,3$ constitute an orthonormal basis in $\mathbb{R}^3$. We list some basic facts about the divergence, the curl and the Laplace operator on the ``spherical projection coordinates". Define $\phi(\mathscr{S}^{-1}z)=\tilde{\phi}(z)$ and ${\bf u}(\mathscr{S}^{-1} z)=U_1(z) \tilde{{\bf e}}_1(z) +U_2(z) \tilde{{\bf e}}_2(z) +U_3(z) \tilde{{\bf e}}_3(z)$, then
\be\label{z1}
&&\Delta_x \phi(x)= \p_{z_1}^2\tilde{\phi}+\frac{(1+|z'|^2)^2}{4z_1^2}\sum_{j=2}^3\p_{z_j}^2\tilde{\phi}+\frac{2}{z_1}\p_{z_1}\tilde{\phi} ,\\\label{z2}
&&{\bf u}(x)\cdot\nabla_x= U_1\p_{z_1}+\frac{1+|z'|^2}{2z_1}(U_2\p_{z_2}+U_3\p_{z_3}),\\\label{z3}
&&\text{div }{\bf u}(x)= \p_{z_1}U_1(z) + \frac{1+|z'|^2}{2z_1}\sum_{j=2}^3\p_{z_j}U_j+\frac{2U_1}{z_1} -\frac{1}{z_1}\sum_{j=2}^3 z_j U_j,\\\label{z4}
&&\text{curl }{\bf u}(x)=\omega_1(z) \tilde{{\bf e}}_1 +\omega_2(z) \tilde{{\bf e}}_2 +\omega_3(z) \tilde{{\bf e}}_3,
\ee
where
\be\no
&&\omega_1= \frac{1+|z'|^2}{2z_1}(\p_{z_2} U_3- \p_{z_3} U_2) + \frac{1}{z_1}(z_3 U_2-z_2 U_3),\\\no
&&\omega_2=\frac{1+|z'|^2}{2z_1}\p_{z_3} U_1- \p_{z_1} U_3- \frac{U_3}{z_1},\\\no
&&\omega_3= \p_{z_1} U_2+\frac{U_2}{z_1}- \frac{1+|z'|^2}{2z_1}\p_{z_2} U_1.
\ee

Under the ``spherical projection coordinates", the steady Euler equations are transformed as
\be\label{euler-s}\begin{cases}
\p_{z_1} (\rho U_1) + \frac{1+|z'|^2}{2z_1}(\p_{z_2} (\rho U_2)+\p_{z_3} (\rho U_3))+\frac{2\rho U_1}{z_1} -\frac{\rho}{z_1}(z_2 U_2+z_3 U_3)=0,\\
(U_1\p_{z_1}+\frac{1+|z'|^2}{2z_1}(U_2\p_{z_2}+U_3\p_{z_3}))U_1+\frac{1}{\rho}\p_{z_1} P-\frac{U_2^2+U_3^2}{z_1}=0,\\
(U_1\p_{z_1}+\frac{1+|z'|^2}{2z_1}(U_2\p_{z_2}+U_3\p_{z_3}))U_2+\frac{1+|z'|^2}{2z_1}\frac{1}{\rho}\p_{z_2} P\\
\q\q\q \q+\frac{1}{z_1} U_1 U_2-\frac{1}{z_1} U_3(z_3 U_2-z_2 U_3)=0,\\
(U_1\p_{z_1}+\frac{1+|z'|^2}{2z_1}(U_2\p_{z_2}+U_3\p_{z_3}))U_3+\frac{1+|z'|^2}{2z_1}\frac{1}{\rho}\p_{z_3} P\\
\q\q\q\q+\frac{1}{z_1} U_1 U_3 +\frac{1}{z_1} U_2(z_3 U_2-z_2 U_3)=0,\\
(U_1\p_{z_1}+\frac{1+|z'|^2}{2z_1}(U_2\p_{z_2}+U_3\p_{z_3}))B=0.
\end{cases}\ee
There are no singular terms in \eqref{euler-s} which provides an appropriate setting for the stability analysis of the spherical transonic shock problem. The domain $\Omega$ is changed to be
\be\no
\mathcal{N}=\{(z_1,z'): r_1<z_1<r_2,\ z'\in E\}\ \ \  \text{with \ \ } \ E=\{z'=(z_2,z_3): |z'|<1\}.
\ee

In the coordinates $(z_1,z')$, the shock surface can be represented as  $z_1=\xi(z')$ and the corresponding Rankine-Hugoniot conditions become
\be\lab{rh}\begin{cases}
[\rho U_1] -\frac{1+|z'|^2}{2\xi(z')}(\p_{z_2}\xi [\rho U_2] + \p_{z_3}\xi [\rho U_3])=0,\\
[\rho U_1^2+ P]-\frac{1+|z'|^2}{2\xi(z')}(\p_{z_2}\xi [\rho U_1 U_2] +\p_{z_3}\xi [\rho U_1 U_3])=0,\\
[\rho U_1 U_2] -\frac{1+|z'|^2}{2\xi(z')}(\p_{z_2}\xi [\rho U_2^2+P] +\p_{z_3}\xi [\rho U_2 U_3])=0,\\
[\rho U_1 U_3]-\frac{1+|z'|^2}{2\xi(z')}(\p_{z_2}\xi [\rho U_2 U_3] +\p_{z_3}\xi [\rho U_3^2+P])=0,\\
[\frac{1}{2}\sum_{j=1}^3U_j^2+\frac{\gamma P}{(\gamma-1)\rho}]=0.
\end{cases}\ee

Moreover, the physical entropy condition is also satisfied
\be\lab{entropy}
K^+(\xi(z'),z')>K^-(\xi(z'),z'),\ \ \text{on}\ \ z_1=\xi(z').
\ee

To describe the compatibility conditions on the shell wall in a simple form, one needs to introduce the polar coordinates
\be\label{polar1}
z_2=a\cos\tau,\ \ \ z_3 =a\sin\tau,
\ee
and define
\be\label{polar2}\begin{cases}
U_a(z_1,a,\tau):= U_2(z_1,a\cos\tau,a\sin\tau) \cos\tau + U_3(z_1,a\cos\tau,a\sin\tau) \sin\tau,\\
U_{\tau}(z_1,a,\tau):= -U_2(z_1,a\cos\tau,a\sin\tau) \sin\tau + U_3(z_1,a\cos\tau,a\sin\tau) \cos\tau,\\
(U_1,K,P, B,\rho)(z_1,a,\tau):= (U_1,K, P,B,\rho)(z_1,a\cos\tau,a\sin\tau).
\end{cases}\ee

Under the coordinates $(z_1,a,\tau)$ introduced in \eqref{polar1}-\eqref{polar2}, the system \eqref{euler-s} is transformed to be
\be\label{euler-p}\begin{cases}
\p_{z_1} (\rho U_1) + \frac{1+a^2}{2z_1}(\p_{a} (\rho U_a)+\frac{1}{a}\rho U_a+\frac{1}{a}\p_{\tau} (\rho U_{\tau}))+\frac{2\rho U_1}{z_1} -\frac{a \rho U_a}{z_1}=0,\\
(U_1\p_{z_1}+\frac{1+a^2}{2z_1}(U_a\p_{a}+\frac{1}{a}U_{\tau}\p_{\tau}))U_1+\frac{1}{\rho}\p_{z_1} P-\frac{U_a^2+U_{\tau}^2}{z_1}=0,\\
(U_1\p_{z_1}+\frac{1+a^2}{2z_1}(U_a\p_{a}+\frac{1}{a}U_{\tau}\p_{\tau}))U_a+\frac{1+a^2}{2z_1}\frac{1}{\rho}\p_{a} P+\frac{1}{z_1} U_1 U_a-\frac{1-a^2}{2z_1a} U_{\tau}^2=0,\\
(U_1\p_{z_1}+\frac{1+a^2}{2z_1}(U_a\p_{a}+\frac{1}{a}U_{\tau}\p_{\tau}))U_{\tau}+\frac{1+a^2}{2z_1}\frac{1}{\rho}\frac{1}{a}\p_{\tau} P+\frac{1}{z_1} U_1 U_{\tau} +\frac{ 1-a^2}{2z_1a} U_a U_{\tau}=0,\\
(U_1\p_{z_1}+\frac{1+a^2}{2z_1}(U_a\p_{a}+\frac{1}{a}U_{\tau}\p_{\tau}))K=0.
\end{cases}\ee

Let the incoming supersonic flow at the inlet $r=r_1$ be prescribed as
\begin{equation}\lab{super1}
{\bf \Psi}^-(r_1, z')=\overline{\Psi}^-(r_1)+ \epsilon (U_{1,0}^-,U_{2,0}^-,U_{3,0}^-, P_0^-,K_0^-)(z'),
\end{equation}
where
$(U_{1,0}^-,U_{2,0}^-,U_{3,0}^-, P_0^-,K_0^-)\in (C^{2,\alpha}(\overline{E}))^5$ satisfy the compatibility conditions
\be\label{comp1}\begin{cases}
\p_a U_{1,0}^-(1,\tau)=\p_a U_{\tau,0}^-(1,\tau)=\p_a P_0^-(1,\theta)=\p_a K_0^-(1,\tau)=0,\ \ \forall\tau\in \mathbb{T}_{2\pi},\\
U_{a,0}^-(1,\tau)=\p_a^2 U_{a,0}^-(1,\tau)+\p_a U_{a,0}^-(1,\tau)=0,\forall\tau\in \mathbb{T}_{2\pi}.
\end{cases}\ee
with
\be\no\begin{cases}
U_{a,0}^-(a,\tau):=U_{2,0}^-(a\cos\tau,a\sin\tau)\cos\tau + U_{3,0}^-(a\cos\tau,a\sin\tau)\sin\tau,\\
U_{\tau,0}^-(a,\tau):=-U_{2,0}^-(a\cos\tau,a\sin\tau)\sin\tau + U_{3,0}^-(a\cos\tau,a\sin\tau)\cos\tau,\\
(U_{1,0}^-,P_0^-, K_0^-)(a,\tau):=(U_{1,0}^-,P_0^-, K_0^-)(a\cos\tau,a\sin\tau).
\end{cases}\ee

On the wall $\{(z_1,z'):r_1\leq z_1\leq r_2, |z'|=1\}$, the slip boundary condition yields that
\be\label{slip}
z_2 U_2(z_1,z')+ z_3 U_3(z_1,z')=0.
\ee

On the exit $z_1=r_2$, the exit pressure is prescribed by
\be\label{pres}
P^+(r_2,z')= P_e + \epsilon P_{ex}(z')\ \text{on}\ \ z'\in E,
\ee
here $\epsilon>0$ is sufficiently small,  $P_{ex}\in C^{2,\alpha}(\overline{E})$ satisfies the compatibility condition
\be\label{pressure-cp}
\p_a P_{ex}(1,\tau)=0,\ \ \ \forall \tau\in \mathbb{T}_{2\pi}.
\ee

The existence and uniqueness of supersonic flow to \eqref{euler-s} follows from the theory of classical solution to the boundary value problem for quasi-linear symmetric hyperbolic equations (See \cite{bs07}).
\begin{lemma}\label{supersonic}
{\it For the supersonic incoming flow given in \eqref{super1} satisfying \eqref{comp1}, there exists $\epsilon_0>0$ depending only on the background solution and the boundary data, such that for any $0<\epsilon<\epsilon_0$, there exists a unique $C^{2,\alpha}(\overline{\mathcal{N}})$ solution $\bm\Psi^-=(U_1^-,U_2^-,U_3^-,P^-,K^-)(z_1,z')$ to \eqref{euler-s} with \eqref{super1}, which satisfies
\be\no
\|(U_1^-,U_2^-,U_3^-,P^-,K^-)-(\bar{U}^-,0,0,\bar{P}^-, \bar{K}^-)\|_{C^{2,\alpha}(\overline{\mathcal{N}})}\leq C_*\epsilon
\ee
and the compatibility conditions
\be\label{comp2}\begin{cases}
(\p_a U_{1}^-,\p_a U_{\tau}^-,\p_a P^-,\p_a K^-)(z_1, 1,\tau)=0,\ \ \forall(z_1,\tau)\in [r_1,r_2]\times\mathbb{T}_{2\pi},\\
U_{a}^-(z_1, 1,\tau)=(\p_a^2 U_{a}^-+\p_a U_{a}^-)(z_1,1,\tau)=0,\forall(z_1,\tau)\in [r_1,r_2]\times \mathbb{T}_{2\pi}.
\end{cases}\ee
}\end{lemma}

\begin{remark}\label{r-super}
{\it A detailed proof of \eqref{comp2} in Lemma \ref{supersonic} will be given in \S\ref{appendix} by using the equations \eqref{euler-p}. If one considers the conic nozzle with the opening angle $\theta_0\in (0,\frac{\pi}{2})$, the transformed domain becomes $\mathcal{N}_{\theta_0}=\{(y_1,y'):r_1<y_1<r_2, |y'|<\frac{\sin\theta_0}{1+\cos\theta_0}<1\}$. Then the compatibility condition \eqref{comp1} at the entrance can not be preserved on the whole nozzle wall in general. This can be seen from the proof of Lemma \ref{supersonic} in \S\ref{appendix} where the derivative $\frac{d}{da}(a+\frac1a)|_{a=1}=0$ is used crucially.
}\end{remark}

Therefore, our problem is reduced to solve a free boundary value problem for the steady Euler equations in which the shock front and the downstream subsonic flows are unknown. Then the main result in this paper is stated as follows.
\begin{theorem}\label{existence}
{\it There exists $\epsilon_0>0$ depending only on the background solution $\overline{\bm{\Psi}}$ and the boundary data $(U_{1,0}^-,U_{2,0}^-,U_{3,0}^-, P_0^-,K_0^-)\in (C^{2,\alpha}(\overline{E}))^5$, $P_{ex}\in C^{2,\alpha}(\overline{E})$ satisfying the compatibility conditions \eqref{comp1} and \eqref{pressure-cp}, such that if $0\leq \epsilon<\epsilon_0$, the problem \eqref{euler-s} with \eqref{super1}, \eqref{slip}, \eqref{pres}, and \eqref{rh} has a unique solution $\bm{\Psi}^+=(U_1^+,U_2^+,U_3^+,P^+,K^+)(z_1,z')$ with the shock front $\mathcal{S}: z_1=\xi(z')$ satisfying
\begin{enumerate}[(i)]
  \item The function $\xi(z')\in C^{3,\alpha}(\overline{E})$ satisfies
  \be\no
  \|\xi(z')-r_s\|_{C^{3,\alpha}(\overline{E})}\leq C_*\epsilon,
  \ee
  where $C_*$ depends only on the background solution and the incoming flow and the exit pressure.
  \item The solution $\bm{\Psi}^+=(U_1^+,U_2^+,U_3^+,P^+,K^+)(z_1,z')\in C^{2,\alpha}(\overline{\mathcal{N}_+})$ with $\mathcal{N}_+=\{(z_1,z'): \xi(z')<z_1<r_2,z'\in E\}$ satisfies the entropy condition
  \begin{equation}\no
  K^+(\xi(z')+,z') >K^-(\xi(z')-, z')\quad \text{for}\,\, z'\in E
  \end{equation}
   and
  \be\no
  \|\bm{\Psi}^+ -\overline{{\bm{\Psi}}}\|_{C^{2,\alpha}(\overline{\mathcal{N}_+})}\leq C_*\epsilon.
  \ee
\end{enumerate}
}\end{theorem}

\begin{remark}\label{r3}
{\it Different from \cite{chen08,cy08} where they assumed that the shock front passed through a fixed point, the shock front in Theorem \ref{existence} is completely free and uniquely determined by the incoming flow and the exit pressure. The authors in \cite{lxy16} proved the stability of the spherical transonic shock by requiring that the background transonic shock satisfies the ``Structure Condition" which seems difficult to verify. Such a condition is also removed in Theorem \ref{existence}. This is essentially due to the use of the deformation-curl decomposition for the elliptic-hyperbolic coupled structure in the steady Euler equations developed by the authors in \cite{wx19} and the fact that the first order deformation-curl system for the velocity can be reduced to a second order nonlocal elliptic equation for a potential function with oblique boundary conditions on the shock front and the exit whose coefficients have correct signs (see the system \eqref{den41}) so that its unique solvability can be derived directly without any restrictions.
}\end{remark}

\begin{remark}\label{r4}
{\it For the transonic shock flows constructed in Theorem \ref{existence}, the velocity, the entropy and the pressure in the subsonic region have the same $C^{2,\alpha}(\overline{\Omega^+})$ regularity. This is an important difference between our works and the previous existence results \cite{chen08,cy08,lxy16}, where they require that the pressure has one order higher regularity than the velocity, thus the regularity requirement of the boundary datum prescribed in \cite{cy08,lxy16} are $C^{3,\alpha}$, stronger than ours. This regularity improvement is a crucial advantage of the deformation-curl decomposition and should play more important role when we consider the stability under generic perturbations of the shape of the conic nozzle.
}\end{remark}
\begin{remark}\label{r5}
{\it If one considers the spherical transonic shock flows in a conic nozzle with the opening angle $\theta_0\in (0,\frac{\pi}{2})$, then the compatibility condition \eqref{comp1} can not be preserved in general. In this case, the velocity field would develop singularity and only $C^{\alpha}(\overline{\mathcal{N}_+})$ regularity for the velocity field is expected in subsonic region (See also \cite[Remark 3.2 and Lemma 3.3]{xyy09}), which induces a loss of uniqueness for the streamlines. However, we emphasize that the ``spherical projection coordinates", the deformation-curl decomposition and the decomposition of the Rankine-Hugoniot conditions developed here are universal even for generic perturbations of supersonic incoming, the exit pressure and the nozzle geometry and may shed light on other related problems.
}\end{remark}

Mathematically, the steady Euler equations in the spherical coordinates contain several artificial ``singular terms" which may cause several trouble including a loss of derivatives when solving the transport equations. Therefore, one needs to find an appropriate coordinates and an associated basis of $\mathbb{R}^3$ to decompose the velocity field so that the transformed equations do not contain singular terms and share a similar structure with the original steady Euler equations. The introduction of  the ``spherical projection coordinates" \eqref{st1} and the basis \eqref{basis} is one of key observations and provide an appropriate setting for the study of the spherical transonic shock problem in a conic nozzle. Under the ``spherical projection coordinates", the divergent conic nozzle is transformed to a cylinder with a circular section. In general, singularities will form near the intersection of the shock front and the nozzle wall and propagate along the trajectories, thus the flow can not have good regularity up to the nozzle wall. In a very special case where the opening angle of the conic nozzle is $\pi/2$, we find that the compatibility conditions at the entrance will be preserved along the wall even crossing the shock front, and thus the regularity in subsonic region can be improved to be $C^{2,\alpha}$ up to the nozzle wall.

Note that the steady Euler equations \eqref{com-euler} are hyperbolic-elliptic mixed in subsonic region and the shock front is a free boundary on which only the Rankine-Hugoniot conditions and the entropy condition hold. we decompose the hyperbolic and elliptic modes in the steady Euler equations \eqref{euler-s} in terms of the deformation and the vorticity. An elaborate decomposition of the Rankine-Hugoniot jump conditions shows that the uniquely determination of the shock front by an algebraic equation and also suitable boundary conditions for the hyperbolic and elliptic modes, respectively. Especially, it yields an unusual second order differential boundary condition on the shock front to the first order nonlocal deformation-curl system, from which an oblique boundary condition can be derived after homogenizing the curl system and introducing the potential function. The compatibility conditions \eqref{comp1} and \eqref{pressure-cp} are prescribed to improve the regularity near the intersection of the entrance and the half spherical shell which is quite nontrivial since \eqref{comp1} is only prescribed at the entrance. We must show that they are preserved not only by the transportation along the streamlines but also when the fluid moves across the shock front.

This paper will be organized as follows. Section \S\ref{reformulation} will devote to the deformation-curl decomposition of the steady Euler equations \eqref{euler-s} in ``spherical projection coordinates", and the reformulation of the Rankine-Hugoniot jump conditions. We also introduce a coordinate transformation to fix the free boundary. In Section \S\ref{proof}, we design an iteration scheme and solve the deformation-curl system with nonlocal terms and the unusual second order differential boundary condition on the shock front. In the Appendix \S\ref{appendix}, we verify some compatibility conditions that are used in Section \S\ref{proof}.

\section{The reformulation of the spherical transonic shock problem}\label{reformulation}

\subsection{The deformation-curl decomposition of the steady Euler equations \eqref{euler-s}}
Let us give the details of the deformation-curl decomposition to the steady Euler system in the ``spherical projection coordinates". First, we describe the hyperbolic modes in \eqref{euler-s}. The Bernoulli's quantity and the entropy satisfy the transport equations
\be\lab{ent10}
&&\b(\p_{z_1}+\frac{1+|z'|^2}{2z_1 U_1}(U_2\p_{z_2}+U_3\p_{z_3})\b) K=0,\\\label{ber10}
&&\b(\p_{z_1}+\frac{1+|z'|^2}{2z_1 U_1}(U_2\p_{z_2}+U_3\p_{z_3})\b) B=0.
\ee

Recall the definition of the vorticity $\bm{\omega}=\text{curl }{\bf u}(x)= \omega_1 \tilde{e}_1+ \omega_2 \tilde{e}_2+ \omega_3 \tilde{e}_3$ in \eqref{z4}, we may rewrite the third and fourth equations in \eqref{euler-s} as
\be\no\begin{cases}
\displaystyle U_1 \om_3 - U_3 \om_1 + \frac{1+|z'|^2}{2z_1}\b(\p_{z_2} B-\frac{B-\frac{1}{2}|{\bf U}|^2} {\gamma K}\p_{z_2} K\b)=0,\\
\displaystyle-U_1 \om_2 + U_2 \om_1 + \frac{1+|z'|^2}{2z_1}\b(\p_{z_3} B-\frac{B-\frac{1}{2}|{\bf U}|^2} {\gamma K}\p_{z_3} K\b)=0,
\end{cases}\ee
from which one can derive that
\be\lab{vor13}\begin{cases}
\displaystyle\omega_2 =\frac{U_2\omega_1}{U_1} +\frac{1+|z'|^2}{2z_1 U_1} \b(\p_{z_3} B-\frac{B-\frac{1}{2}|{\bf U}|^2} {\gamma K}\p_{z_3} K\b),\\
\displaystyle\omega_3 =\frac{U_3\omega_1}{U_1}- \frac{1+|z'|^2}{2z_1 U_1} \b(\p_{z_2} B-\frac{B-\frac{1}{2}|{\bf U}|^2} {\gamma K}\p_{z_2} K\b).
\end{cases}\ee

Since $\text{div }\text{curl }{\bf u}=0$, then
\be\no
\p_{z_1} \omega_1 + \frac{1+|z'|^2}{2z_1}(\p_{z_2}\omega_2 + \p_{z_3}\omega_3) +\frac{2}{z_1}\omega_1-\frac{1}{z_1}(z_2\omega_2+z_3\omega_3)=0.
\ee
Substituting \eqref{vor13} into the above equation, one obtains
\be\no
&&\p_{z_1} \omega_1 +\frac{1+|z'|^2}{2z_1 U_1}\sum\limits_{i=2}^3U_i\p_{z_i}\omega_1+\frac{2}{z_1} \omega_1 +\frac{1+|z'|^2}{2z_1}\sum\limits_{i=2}^3\p_{z_i}\b(\frac{U_i}{U_1}\b)\omega_1\\\label{vor14}
&&\quad-\frac{z_2 U_2 +z_3 U_3}{z_1 U_1}\omega_1+\frac{(1+|z'|^2)^2}{4z_1^2}\bigg\{\p_{z_2}\b(\frac{1}{U_1}\b)\p_{z_3} B-\p_{z_3}\b(\frac{1}{U_1}\b)\p_{z_2} B\\\no
&&\quad-\p_{z_2}\b(\frac{B-\frac{1}{2}|{\bf U}|^2} {\gamma K U_1}\b)\p_{z_3}K+\p_{z_3}\b(\frac{B-\frac{1}{2}|{\bf U}|^2} {\gamma K U_1}\b)\p_{z_2}K\bigg\}=0.
\ee

Second, we study the elliptic modes in the steady Euler system \eqref{euler-s}. Using the Bernoulli's quantity $B=\frac12 |{\bf U}|^2 + \frac{\gamma K}{\gamma-1}\rho^{\gamma-1}$, one can represent the density as a function of $B, |{\bf U}|^2$ and $K$.
\be\label{den}
\rho= \rho(B,K, |{\bf U}|^2)=\b(\f{\ga-1}{\ga K}\b)^{\f{1}{\ga-1}}(B-\frac12 |{\bf U}|^2)^{\f{1}{\ga-1}}.
\ee
Substituting \eqref{den} into the continuity equation, we derive that
\be\label{den11}
&&\q (c^2(B,|{\bf U}|^2)-U_1^2)\p_{z_1} U_1 +\frac{1+|z'|^2}{2z_1} c^2(B,|{\bf U}|^2)\sum_{j=2}^3\p_{z_j} U_j\\\no
&&+ \frac{c^2(B,|{\bf U}|^2)}{z_1}(2 U_1-\sum_{j=2}^3z_j U_j)=U_1\sum_{j=2}^3 U_j\p_{z_1} U_j+ \frac{1+|z'|^2}{2z_1} \sum_{j=2}^3\sum_{i=1}^3U_j U_i\p_{z_j} U_i.
\ee

The equation \eqref{den11}, together with the vorticity equations, constitutes a deformation-curl system for the velocity field:
\be\label{dc}\begin{cases}
(c^2(B,|{\bf U}|^2)-U_1^2)\p_{z_1} U_1 +\frac{1+|z'|^2}{2z_1} c^2(B,|{\bf U}|^2)\sum_{j=2}^3\p_{z_j} U_j\\
\q\q+ \frac{c^2(B,|{\bf U}|^2)}{z_1}(2 U_1-\sum_{j=2}^3z_j U_j)
\\=U_1\sum_{j=2}^3 U_j\p_{z_1} U_j+ \frac{1+|z'|^2}{2z_1} \sum_{j=2}^3\sum_{i=1}^3U_j U_i\p_{z_j} U_i,\\
\frac{1+|z'|^2}{2z_1}(\p_{z_2} U_3- \p_{z_3} U_2) + \frac{1}{z_1}(z_3 U_2-z_2 U_3)=\omega_1,\\
\frac{1+|z'|^2}{2z_1}\p_{z_3} U_1- \p_{z_1} U_3- \frac{U_3}{z_1}=\omega_2,\\
\p_{z_1} U_2+\frac{U_2}{z_1}- \frac{1+|z'|^2}{2z_1}\p_{z_2} U_1=\omega_3.
\end{cases}\ee

\begin{lemma}\label{equiv}({\bf Equivalence.})
{\it Assume that $C^1$ smooth vector functions $(\rho, {\bf U}, K)$ defined on a domain $\mathbb{D}$ do not contain the vacuum (i.e. $\rho(z_1,z')>0$ in $\Omega$) and the radial velocity $U_1(z_1,z')>0$ in $\mathbb{D}$, then the following two statements are equivalent:
\begin{enumerate}[(i)]
  \item $(\rho, {\bf U}, K)$ satisfy the steady Euler system \eqref{euler-s} in $\mathcal{N}$;
  \item $({\bf U}, K, B)$ satisfy the equation \eqref{ent10},\eqref{ber10},\eqref{vor13} and \eqref{dc}.
\end{enumerate}
}\end{lemma}

\begin{proof}
It suffices to prove that (ii) implies (i). Define $\rho$ by \eqref{den}. Then using the last three equations in \eqref{dc} and the equations \eqref{vor13},  one can directly derive the second and third momentum equations in \eqref{euler-s}. These together with the equation in \eqref{ber10} would imply the first momentum equation in \eqref{euler-s}. Finally, according to the definition of the density \eqref{den}, the continuity equation in \eqref{euler-s} follows directly from the first equation in \eqref{dc} and the equations \eqref{ent10} and \eqref{ber10}.

\end{proof}

\begin{remark}
{\it It is interesting to discuss the role played by the equations \eqref{vor13}. On one hand, we utilize the equations \eqref{vor13} to derive a transport equation \eqref{vor14} for $\omega_1$ which is hyperbolic. The equation \eqref{vor14} is not independent and is regarded here as a byproduct of the divergence freeness of the vorticity and \eqref{vor13}. On the other hand, the equations \eqref{vor13} are also used to constitute a deformation-curl elliptic system for the velocity which represents their ellipticity. Thus the hyperbolicity and ellipticity are coupled in \eqref{vor13} and we decouple them in the above way which turns out to be effective and suit our purpose for solving various problems such as smooth transonic spiral flows in \cite{wxy21b} and the transonic shock problem herein.
}\end{remark}

Define
\be\no
&&V_1(z_1,z')= U_1(z_1,z')- \bar{U}^+(z_1),\q V_j(z_1,z')=U_j(z_1,z'), j=2,3,\\\no
&&V_4(z_1,z')=K(z_1,z')-\bar{K}^+,\ V_5(z_1,z')= B(z_1,z')-\bar{B}^+,\\\no
&& V_6(z')= \xi(z')-r_s,\ \ \ {\bf V}(z_1,z')=(V_1,\cdots, V_5)(z_1,z').
\ee
Then the density and the pressure can be expressed as
\be\label{denw1}
&&\rho=\rho({\bf V})=\b(\frac{\gamma-1}{\gamma (\bar{K}^++V_4)}\b)^{\frac{1}{\gamma-1}}\b(\bar{B}^++V_5-\frac12(\bar{U}^++V_1)^2-\frac12\sum_{i=2}^3 V_i^2\b)^{\frac{1}{\gamma-1}},\\\label{denw2}
&&P=P({\bf V})=\b(\frac{(\gamma-1)^{\gamma}}{\gamma^{\gamma}(\bar{K}^++V_4)}\b)^{\frac{1}{\gamma-1}}\b(\bar{B}^++V_5-\frac12(\bar{U}^++V_1)^2-\frac12\sum_{i=2}^3 V_i^2\b)^{\frac{\gamma}{\gamma-1}}.
\ee

We rewrite the equations \eqref{ent10},\eqref{ber10}, \eqref{vor13} and \eqref{vor14} in terms of $V_1, V_2,\cdots,V_5$. The equations for the hyperbolic quantities $V_4$ and $V_5$ are
\be\label{ent11}
&&\b(\p_{z_1}+\frac{1+|z'|^2}{2z_1 (\bar{U}+V_1)}(V_2\p_{z_2}+V_3\p_{z_3})\b) V_4=0,\\\label{ber11}
&&\b(\p_{z_1}+\frac{1+|z'|^2}{2z_1 (\bar{U}+V_1)}(V_2\p_{z_2}+V_3\p_{z_3})\b) V_5=0.
\ee
The equations for the vorticity $\omega$ are
\be\no
&&\p_{z_1} \omega_1 +\frac{1+|z'|^2}{2z_1 (\bar{U}+V_1)}(V_2\p_{z_2}+V_3\p_{z_3})\omega_1+\frac{2}{z_1} \omega_1-\frac{z_2 V_2 +z_3 V_3}{z_1 (\bar{U}+V_1)}\omega_1\\\no
&&\q\q +\frac{1+|z'|^2}{2z_1}\sum_{i=2}^3\p_{z_i}\b(\frac{V_i}{\bar{U}+V_1}\b)\omega_1+\frac{(1+|z'|^2)^2}{4z_1^2}\bigg\{\p_{z_2}\b(\frac{1}{\bar{U}+V_1}\b)\p_{z_3} V_5\\\no
&&\q\q-\p_{z_3}\b(\frac{1}{\bar{U}+V_1}\b)\p_{z_2} V_5-\p_{z_2}\b(\frac{\bar{B}-\frac{1}{2}\bar{U}^2+V_5-\bar{U}V_1-\frac{1}{2}\sum_{j=1}^3V_j^2} {\gamma (\bar{K}+V_4)(\bar{U}+V_1)}\b)\p_{z_3}V_4\\\label{vor21}
&&\q\q+\p_{z_3}\b(\frac{\bar{B}-\frac{1}{2}\bar{U}^2+V_5-\bar{U}V_1-\frac{1}{2}\sum_{j=1}^3V_j^2} {\gamma (\bar{K}+V_4)(\bar{U}+V_1)}\b)\p_{z_2}V_4\bigg\}=0.
\ee
and
\be\label{vor22}\begin{cases}
\omega_2 =\frac{V_2\omega_1}{\bar{U}+V_1} +\frac{1+|z'|^2}{2z_1(\bar{U}+V_1)} \b(\p_{z_3} V_5-\frac{\bar{B}-\frac{1}{2}\bar{U}^2+V_5-\bar{U}V_1-\frac{1}{2}\sum_{j=1}^3V_j^2} {\gamma (\bar{K}+V_4)}\p_{z_3} V_4\b),\\
\omega_3 =\frac{V_3\omega_1}{\bar{U}+V_1}- \frac{1+|z'|^2}{2z_1 (\bar{U}+V_1)} \b(\p_{z_2} V_5-\frac{\bar{B}-\frac{1}{2}\bar{U}^2+V_5-\bar{U}V_1-\frac{1}{2}\sum_{j=1}^3V_j^2} {\gamma (\bar{K}+V_4)}\p_{z_2} V_4\b).
\end{cases}\ee
The boundary condition on $|z'|=1$ is
\be\label{slip15}
(z_2V_2+z_3 V_3)(z_1,z')=0, \ \ \forall z_1\in [r_s+V_6(z'),r_2],\ \ |z'|=1.
\ee

Then we further compute the equation for $V_1$. Since
\be\no
&&(c^2(\bar{B},\bar{U}^2)-\bar{U}^2)\bar{U}'(r)+\frac{2\bar{U}c^2(\bar{B},\bar{U}^2)}{r}=0,\\\no
&&(c^2(B,|{\bf U}|^2)-U_1^2-c^2(\bar{B},\bar{U}^2)+\bar{U}^2)\\
&&\q=(\gamma-1)V_5-(\gamma+1) \bar{U}V_1-\frac{\gamma+1}{2} V_1^2-\frac{\gamma-1}{2}(V_2^2+V_3^2),
\ee
then it follows from \eqref{den11} that
\be\lab{den15}
d_1\p_{z_1} V_1 + \frac{1+|z'|^2}{2 z_1}\sum_{i=2}^3\p_{z_i} V_i +\frac{2V_1}{z_1}- \frac{\sum_{i=2}^3 z_i V_i}{z_1}+d_2V_1=d_0V_5+ F({\bf V}),
\ee
where
\be\no
&&d_1(z_1)=1-\bar{M}^2(z_1),\,\,  \bar{M}^2(z_1)=\frac{\bar{U}^2(z_1)}{\bar{c}^2(z_1)},\, \, \bar{c}^2(z_1)=c^2(\bar{B},\bar{U}^2(z_1)),\\\no &&d_2(z_1)=\frac{2\bar{M}^2(2+(\gamma-1)\bar{M}^2)}{z_1(1-\bar{M}^2(z_1))}, \ \ d_0(z_1)=-\frac{(\gamma-1)(\bar{U}'(z_1)+\frac{2\bar{U}(z_1)}{z_1})}{\bar{c}^2(z_1)}.
\ee
and
\be\no
&&\bar{c}^2(z_1)F({\bf V})=-\frac{(\gamma-1)}{z_1}(V_5-\bar{U}V_1-\frac{1}{2}\sum_{j=1}^3 V_j^2)\b(2 V_1+\sum_{i=2}^3\frac{1+|z'|^2}{2}\p_{z_i} V_i - z_i V_i\b)\\\no
&&\quad-\b((\gamma-1)V_5-(\gamma+1)\bar{U}V_1\b)\p_{z_1} V_1 +\frac{(\gamma-1)\bar{U}}{z_1}\sum_{j=1}^3 V_j^2 \\\no
&&\quad+ (\bar{U}'+\p_{z_1}V_1)\b(V_1^2+(\gamma-1)\sum_{j=1}^3V_j^2\b)+ (\bar{U}+V_1)\sum_{j=2}^3 V_j\p_{z_1} V_j\\\no
&&\q\q+ \frac{1+|z'|^2}{2z_1}\sum_{j=2}^3 V_j((\bar{U}+V_1)\p_{z_j}V_1+\sum_{i=2}^3 V_i\p_{z_j}V_i).
\ee

Here $F({\bf V})$ and the following $H_i, G_i$ are quadratic and high order terms. Since we need to verify that these terms satisfy some compatibility conditions, we present their exact form.

\subsection{The reformulation of the Rankine-Hugoniot conditions and boundary conditions}\label{22}

It follows from the third and fourth equations in \eqref{rh} that
\be\no
\p_{z_2} \xi(z') =\frac{2\xi(z')}{1+|z'|^2}\frac{J_2}{J},\ \ \  \ \ \p_{z_3}\xi(z') =\frac{2\xi(z')}{1+|z'|^2}\frac{J_3}{J},
\ee
where
\be\no
J   &=&[\rho U_2^2 + P] [\rho U_3^2 +P]- ([\rho U_2 U_3])^2,\\\no
J_2 &=&[\rho U_3^2+P] [\rho U_1 U_2]- [\rho U_1 U_3] [\rho U_2 U_3],\\\no
J_3 &=&[\rho U_2^2+P] [\rho U_1 U_3]- [\rho U_1 U_2] [\rho U_2 U_3].
\ee
Thus one has
\be\label{shock12}\begin{cases}
\p_{z_2} V_6= \frac{2}{1+|z'|^2}\{b_0 r_s V_2(\xi(z'),z')+g_2({\bf V}(\xi,z'), V_6)\},\\
\p_{z_3} V_6= \frac{2}{1+|z'|^2}\{b_0 r_s V_3(\xi(z'),z')+ g_3({\bf V}(\xi,z'), V_6)\},
\end{cases}\ee
where $b_0= \frac{(\bar{\rho}^+ \bar{U}^+)(r_s)}{[\bar{P}(r_s)]}>0$ and
\be\no
g_i=(r_s+V_6(z'))\frac{J_i}{J}(r_s+V_6(z'),z')-b_0 r_s V_i(\xi(z'),z'),\ \ i=2,3.
\ee
The functions $g_i, i=2,3$ are regarded as error terms which can be bounded by
\be\no
|g_i|\leq C_0(|\bm{\Psi}^-(r_s+V_6,z') - \overline{\bm{\Psi}}^-(r_s+V_6)|+|{\bf V}(\xi,z')|^2+|V_6|^2).
\ee

Substituting \eqref{shock12} into the first two equations in \eqref{rh}, one obtains
\be\label{rh1}\begin{cases}
[\rho U_1]=\frac{[\rho U_2] J_2+[\rho U_3] J_3}{J},\\
[\rho U_1^2+P] =\frac{[\rho U_1 U_2] J_2+[\rho U_1 U_3] J_3}{J},\\
[B]=0.
\end{cases}\ee

Note that
\be\no
\frac{d}{dr}(\bar{\rho} \bar{U})(r)= -\frac{2}{r}(\bar{\rho} \bar{U})(r),\ \ \ \frac{d}{dr}(\bar{\rho} \bar{U}^2+ \bar{P})=-\frac{2}{r}\bar{\rho} \bar{U}^2,
\ee
hence
\be\no
[\bar{\rho} \bar{U}](r_s+V_6)=O(V_6^2),\q [\bar{\rho}\bar{U}^2+\bar{P}](r_s+V_6)-\frac{2}{r_s}[\bar{P}(r_s)] V_6=O(V_6^2).
\ee

Denote $\dot{\rho}(z_1,z')= \rho^+(z_1,z')-\bar{\rho}^+(z_1)$, then the first equation in \eqref{rh1} implies that
\be\no
&&-[\bar{\rho} \bar{U}](\xi)+(\rho^- U_1^-)(\xi,z')-(\bar{\rho}^-\bar{U}^-)(\xi)+\frac{[\rho U_2] J_2+[\rho U_3] J_3}{J}\\\no
&&=(\rho^+ U_1^+)(\xi,z')-(\bar{\rho}^+\bar{U}^+)(\xi)\\\no
&&=\bar{\rho}^+(r_s)V_1(\xi,z')+ \bar{U}^+(r_s)\dot{\rho}(\xi,z')+(V_1+\bar{U}^+(\xi)-\bar{U}^+(r_s))\dot{\rho}(\xi,z')\\\no
&&\quad\quad+ (\bar{\rho}^+(\xi)-\bar{\rho}^+(r_s))V_1(\xi,z').
\ee
Thus
\be\no
&&\bar{\rho}^+(r_s) V_1(\xi,z')+{\bar U}^+(r_s) \dot{\rho}(\xi,z')=-[\bar{\rho} \bar{U}](\xi)+ \frac{[\rho U_2] J_2+[\rho U_3] J_3}{J}\\\no
&&\quad+(\rho^- U_1^-)(\xi,z')-(\bar{\rho}^-\bar{U}^-)(\xi)-(V_1+\bar{U}^+(r_s+V_6)-\bar{U}^+(r_s))\dot{\rho}(\xi,z')\\\no
&&\quad-(\bar{\rho}^+(r_s+V_6)-\bar{\rho}^+(r_s))V_1(\xi,z'):=R_{01}({\bf V}(\xi,z'), V_6).
\ee

Similarly, one can conclude from \eqref{rh1} that at points $(\xi(z'),z')$
\be\label{rh2}\begin{cases}
\bar{\rho}^+(r_s) V_1+{\bar U}^+(r_s) \dot{\rho} = R_{01}({\bf V}(\xi,z'), V_6),\\
2(\bar{\rho}^+ \bar{U}^+)(r_s)V_1+ \{(\bar{U}^+(r_s))^2+c^2(\bar{\rho}^+(r_s),\bar{K}^+)\}\dot{\rho}+(\bar{\rho}^+(r_s))^{\gamma} V_4\\
\quad\quad=-\frac{2}{r_s}[\bar{P}(r_s)] V_6+ R_{02}({\bf V}(\xi,z'), V_6),\\
\bar{U}^+(r_s) V_1+ \frac{c^2(\bar{\rho}^+(r_s),\bar{K}^+)}{\bar{\rho}^+(r_s)} \dot{\rho}+ \frac{\ga (\bar{\rho}^+(r_s))^{\gamma-1}}{(\ga-1)} V_4=R_{03}({\bf V}(\xi,z'), V_6),
\end{cases}\ee
where
\be\no
&&R_{02}=-\left\{[\bar{\rho} \bar{U}^2+ \bar{P}](r_s +V_6)-\frac{2}{r_s}[\bar{P}(r_s)] V_6\right\}\\\no
&&\q\q+ (\rho^- (U_1^-)^2 +P^-)(\xi,z')-(\bar{\rho}^-(\bar{U}^-)^2+\bar{P}^-)(\xi)\\\no
&&\quad\quad-\bigg\{(\rho^+ (U_1^+)^2 +P^+)(\xi,z')-(\bar{\rho}^+(\bar{U}^+)^2+\bar{P}^+)(\xi)-2(\bar{\rho}^+ \bar{U}^+)(r_s)V_1\\\no
&&\quad\quad- \{(\bar{U}^+(r_s))^2+c^2(\bar{\rho}^+(r_s),\bar{K}^+)\}\dot{\rho}-(\bar{\rho}^+(r_s))^{\gamma} V_4\bigg\}
+\frac{[\rho U_1 U_2] J_2+[\rho U_1 U_3] J_3}{J},\\\no
&&R_{03}= B^-(r_s+ V_6,z')- \bar{B}^-- \bar{U}^+(r_s+V_6) V_1(\xi,z')-\frac{1}{2}\sum_{j=1}^3 V_j^2(\xi,z')\\\no
&&\quad\quad -\frac{\gamma}{\gamma-1}\left((\bar{K}^+ +V_4(\xi,z'))(\rho(\xi,z'))^{\gamma-1}-\bar{K}^+ (\bar{\rho}^+(\xi))^{\gamma-1}\right)\\\no
&&\quad\quad + \bar{U}^+(r_s) V_1+ \frac{c^2(\bar{\rho}^+(r_s),\bar{K}^+)}{\bar{\rho}^+(r_s)} \dot{\rho}+ \frac{\ga (\bar{\rho}^+(r_s))^{\gamma-1}}{(\ga-1)} V_4.
\ee

Using the equation \eqref{denw1}, $\dot{\rho}$ can be represented as a function of ${\bf V}, V_6$ and there exists a constant $C_0>0$ depending only on the background solution, such that
\be\no
|R_{0i}|\leq C_0(|\bm{\Psi}^-(\xi(z'),z')-\overline{\bm{\Psi}}^-(\xi(z'))|+ |{\bf V}(\xi(z'),z')|^2+|V_6(z')|^2),\ \ i=1,2,3
\ee
In the following, the functions $R_{0i}, i=2,3$ (also the functions $R_j$ in this paper) will represent different quantities that can be bounded as above.

Then solving the algebraic equations in \eqref{rh2}, one gains
\be\label{shock13}\begin{cases}
V_1(\xi(z'),z')=b_1V_6(z')+ R_1({\bf V}(\xi(z'),z'), V_6),\\
V_4(\xi(z'),z')=b_2V_6(z')+ R_2({\bf V}(\xi(z'),z'), V_6),\\
V_5(\xi(z'),z')= B^-(r_s+V_6(z'),z')-\bar{B}^-,
\end{cases}\ee
where
\be\no
&&b_1=\f{\ga \bar{U}^+(r_s)[\bar{P}(r_s)]}{r_s\bar{\rho}^+(r_s) (c^2(\bar{\rho}^+(r_s),\bar{K}^+)-(\bar{U}^+(r_s))^2)}>0,\\\no
&&b_2=\frac{2(\gamma-1)[\bar{P}(r_s)]}{r_s(\bar{\rho}^+(r_s))^{\gamma}}>0
\ee
and
\be\no
&&R_1=\frac{(c^2(\bar{\rho}^+(r_s),\bar{K}^+)+\gamma(\bar{U}^+(r_s))^2) R_{01}-\gamma \bar{U}^+(r_s)R_{02} + (\gamma-1)(\bar{\rho}^+ \bar{U}^+)(r_s)R_{03}}{\bar{\rho}^+(r_s) (c^2(\bar{\rho}^+(r_s),\bar{K}^+)-(\bar{U}^+(r_s))^2)}\\\no
&&:= \sum_{i=1}^3 b_{1i} R_{0i},\\\no
&&R_2=\frac{\gamma-1}{(\bar{\rho}^+(r_s))^{\gamma-1}}\b(\bar{U}^+(r_s)R_{01}-R_{02}+ \bar{\rho}^+(r_s) R_{03}\b):= \sum_{i=1}^3 b_{2i} R_{0i}.
\ee

In the following, the superscript ``+" in $\bar{U}^+, \bar{P}^+,\bar{K}^+, \bar{B}^+$ is ignored to simplify the notations. To derive the boundary condition at the exit, by the definition of the Bernoulli's function one obtains
\be\no
&&V_5(r_2,z')=\bar{U}(r_2)V_1(r_2,z')+ \frac{\p h}{\p P}(\bar{P}(r_2), \bar{K})(P(r_2,z')-\bar{P}(r_2))\\\no
&&\quad+ \frac{\p h}{\p K}(\bar{P}(r_2), \bar{K})V_4+ \frac12 \sum_{j=1}^3 V_j^2 +E({\bf V}(r_2,z')),
\ee
where
\be\no
&&E({\bf V}(z_1,z'))=\frac{\gamma}{\gamma-1}(\bar{K}+V_4)^{\frac1{\gamma}}(P({\bf V}))^{\frac{\gamma-1}{\gamma}}-\frac{\gamma}{\gamma-1}\bar{K}^{\frac1{\gamma}}\bar{P}^{\frac{\gamma-1}{\gamma}}\\\no
&&\quad\quad\quad-\frac{1}{\bar{\rho}(r)}(P({\bf V})- \bar{P})-\frac{\bar{B}-\frac{1}{2}\bar{U}^2(r)}{\gamma \bar{K}} V_4.
\ee

Together with the exit pressure condition \eqref{pres}, one concludes that
\be\label{pres2}
&&V_1(r_2,z')=\frac{V_5(r_2,z')}{\bar{U}(r_2)}- \frac{\bar{B}-\frac12 \bar{U}^2(r_2)}{\gamma\bar{K}\bar{U}(r_2)} V_4(r_2,z') -\frac{\epsilon P_{ex}(z')}{(\bar{\rho}\bar{U})(r_2)}\\\no
&&\quad-\frac{1}{2\bar{U}(r_2)}\sum_{j=1}^3 V_j^2(r_2,z')+E({\bf V}(r_2,z')).
\ee
Here $E$ is an error function that can be bounded by
\be\no
|E({\bf V}(r_2,z'))|\leq C_0|{\bf V}(r_2,z')|^2.
\ee

Then to solve the problem \eqref{euler-s} with \eqref{rh},\eqref{super1},  \eqref{slip} and \eqref{pres} is equivalent to find a function $W_6$ defined on $E$ and vector functions $(V_1,\cdots, V_5)$ defined on the $\mathcal{N}_{V_6}:=\{(z_1,z'): r_s +V_6(z')<z_1<r_2,z'\in E\}$, which solves the equations \eqref{ent11}-\eqref{vor22},\eqref{den15} with boundary conditions \eqref{shock12},\eqref{shock13},\eqref{slip15} and \eqref{pres2}.

\subsection{The coordinate transformation} Since the shock front is a free boundary, to fix the subsonic domain, relabel $W_6(z')=\xi(z')-r_s$, we introduce the following coordinates transformation
\be\label{coor}
y_1=\frac{z_1-\xi(z')}{r_2-\xi(z')}(r_2-r_s) + r_s,\ y'=(y_2,y_3)=z',
\ee
then
\be\no\begin{cases}
z_1= y_1+\frac{r_2-y_1}{r_2-r_s}W_6:=D_0^{W_6},\\
\p_{z_1}=\frac{r_2-r_s}{r_2-r_s-W_6} \p_{y_1}:= D_1^{W_6},\\
\p_{z_2}=\p_{y_2}-\frac{(r_2-y_1)\p_{y_2}W_6}{r_2-r_s-W_6}\p_{y_1}:=D_2^{W_6},\\
\p_{z_3}=\p_{y_3}-\frac{(r_2-y_1)\p_{y_3}W_6}{r_2-r_s-W_6}\p_{y_1}:=D_3^{W_6}.
\end{cases}\ee

Then the domain $\mathcal{N}_+:=\{(z_1,z'):\xi(z')<z_1<r_2,z'\in E\}$ is changed to be
\be\no
\mathbb{D}=\{(y_1,y'): y_1\in (r_s, r_2), y'\in E\}, \ \ \ \ E:=\{y': |y'|<1\}.
\ee
Denote
\be\no
&&\Sigma_{s}=\{(y_1,y_2,y_3):y_1=r_s, y'\in E\},\ \ \ \Sigma_{ex}=\{(y_1,y_2,y_3):y_1=r_2, y'\in E\},\\\no
&&\Sigma_w=\{(y_1,y'): y_1\in [r_s,r_2], |y'|=1\}.
\ee

Set
\be\no\begin{cases}
W_j(y)= V_j(y_1+\f{r_2-y_1}{r_2-r_s}W_6(y'),y'), j=1,\cdots, 5,\\
\tilde{\omega}_j(y)=\omega_j(y_1+\f{r_2-y_1}{r_2-r_s}W_6(y'),y'), j=1,2,3.
\end{cases}\ee

Then the functions $\rho(z_1,z')$ and $P(z_1,z')$ in \eqref{denw1}-\eqref{denw2} are transformed to be
\be\no
&&\tilde{\rho}({\bf W},W_6)=\b(\frac{\gamma-1}{\gamma(\bar{K}+W_4)}\b)^{\frac{1}{\gamma-1}}\b(\bar{B}+W_5-\frac{1}{2}(\bar{U}(D_0^{W_6})+W_1)^2-\frac{1}{2}\sum_{i=2}^3 W_j^2\b)^{\frac{1}{\gamma-1}},\\\no
&&\tilde{P}({\bf W},W_6)=\b(\frac{(\gamma-1)^{\gamma}}{\gamma^{\gamma}(\bar{K}+W_4)}\b)^{\frac{1}{\gamma-1}}\b(\bar{B}+W_5
-\frac{1}{2}(\bar{U}(D_0^{W_6})+W_1)^2-\frac{1}{2}\sum_{i=2}^3 W_j^2\b)^{\frac{\gamma}{\gamma-1}}.
\ee

Now we describe the iteration scheme to solve the transonic shock problem. The shock front will be determined by the first formula in \eqref{shock13} as follows
\be\label{shock400}
W_6(y')=\frac{1}{b_1}W_1(r_s,y')- \frac{1}{b_1}R_1({\bf W}(r_s,y'),W_6(y'))
\ee
where $R_{1}({\bf W}(r_s,y'),W_6(y'))=\sum_{i=1}^3 b_{1i}R_{0i}({\bf W}(r_s,y'),W_6(y'))$. The exact formulas for $R_{0i}, i=1,2,3$ in the polar coordinates $(y_1,a,\tau)$-coordinates are presented in the Appendix to verify the compatibility conditions (See \eqref{rv11}-\eqref{rv41}):

The second and third formulas in \eqref{shock13} yield the boundary datum for the Bernoulli's quantity and the entropy. The function $W_5$ and $W_4$ would satisfy
\be\label{ber31}\begin{cases}
\b(D_1^{W_6}+\frac{1+|z'|^2}{2D_0^{W_6}(\bar{U}+W_1)}(W_2D_2^{W_6}+W_3D_3^{W_6})\b) W_5=0,\\
W_5(r_s,y')=B^-(r_s+W_6(y'),y')-\bar{B}^-,
\end{cases}\ee
and
\be\label{ent31}\begin{cases}
\b(D_1^{W_6}+\frac{1+|z'|^2}{2D_0^{W_6}(\bar{U}+W_1)}(W_2D_2^{W_6}+W_3D_3^{W_6})\b) W_4=0,\\
W_4(r_s,y')=b_2 W_6(y')+R_2({\bf W}(r_s,y'),W_6(y'))
\end{cases}\ee
where $R_{2}({\bf W}(r_s,y'),W_6(y'))=\displaystyle\sum_{i=1}^3 b_{2i}R_{0i}({\bf W}(r_s,y'),W_6(y'))$.

In the $y$-coordinates, \eqref{shock12} is changed to be
\be\label{shock1200}
\p_{y_i} W_6(y')=\frac{2b_0 r_s}{1+|y'|^2} W_i(r_s,y') +\frac{2}{1+|y'|^2} g_i({\bf W}(r_s,y'), W_6),\ \ i=2,3,
\ee
where
\be\label{g21}
g_i=(r_s+W_6(y'))\frac{J_i}{J}(r_s+W_6(y'),y')-b_0 r_s W_i(r_s,y'), \ \ i=2,3.
\ee
and the exact formulas for $J_2, J_3$ and $J$ are given in the Appendix. These will be used to verify the compatibility conditions required below (See \eqref{j211}).

The following reformulation of the jump conditions \eqref{shock1200} is crucial for us to solve the transonic shock problem. Note that \eqref{shock1200} is equivalent to that on $E$ there holds
\be\label{shock15}\begin{cases}
F_2(y'):=\p_{y_2} W_6- \frac{2b_0 r_s W_2(r_s, y')}{1+|y'|^2} -\frac{2}{1+|y'|^2}g_2({\bf W}(r_s,y'),W_6)\equiv 0,\\
F_3(y'):=\p_{y_3} W_6- \frac{2b_0r_s W_3(r_s,y')}{1+|y'|^2}-\frac{2}{1+|y'|^2}g_3({\bf W}(r_s,y'),W_6)\equiv 0.
\end{cases}\ee

Then a key reformulation is the following.
\begin{lemma}\label{equi0}
{\it Let $F_j, j=2,3$ be two $C^1$ smooth functions defined on $\overline{E}$. Then the following two statements are equivalent
\begin{enumerate}[(i)]
  \item $F_2=F_3\equiv 0$ on $\overline{E}$;
  \item $F_2$ and $F_3$ solves the following problem
  \be\label{equi00}\begin{cases}
  \p_{y_2}F_3-\p_{y_3} F_2=0,\ \ &\text{in}\ E,\\
  \p_{y_2} F_2 + \p_{y_3} F_3=0,\ \ &\text{in}\ E,\\
  y_2 F_2+y_3 F_3=0,\ \ &\text{on}\ \ y_2^2+y_3^2=1.
\end{cases}\ee
\end{enumerate}
}\end{lemma}

\begin{proof}
It suffices to show that (ii) will imply (i). Indeed the first equation in \eqref{equi00} implies that there exists a potential function $\Phi(y_2,y_3)\in C^2(\overline{E})$ such that $F_j=\p_{y_j} \Phi, j=2,3$ on $\overline{E}$. Then the second equation in \eqref{equi00} yields
\be\no\begin{cases}
\p_{y_2}^2\Phi+\p_{y_3}^2 \Phi = 0,\ \ &\text{in }E,\\
\sum_{j=2}^3y_j\p_{y_j} \Phi(y')=0,\ \ &\text{on}\ \ |y'|=1.
\end{cases}\ee
Thus one can conclude that $\Phi\equiv \text{const}$ on $\overline{E}$ and thus $F_2=F_3\equiv 0$ on $\overline{E}$.
\end{proof}

Substituting \eqref{shock15} into the first equation in \eqref{equi00} yields that
\be\label{shock17}
&&\frac{1+|y'|^2}{2r_s}(\p_{y_2} W_3-\p_{y_3} W_2)(r_s,y')+\frac{1}{r_s}(y_3 W_2-y_2 W_3)(r_s,y')\\\no
&&=\frac{(1+|y'|^2)^2}{2 b_0 r_s^2} \bigg(\p_{y_3}\bigg(\frac{g_2({\bf W}(r_s,y'), W_6(y'))}{1+|y'|^2}\bigg)-\p_{y_2}\bigg(\frac{g_3({\bf W}(r_s,y'), W_6(y'))}{1+|y'|^2}\bigg)\bigg).
\ee
This gives the boundary data on the shock front for the first component of the vorticity.

Substituting \eqref{shock15} into the second equation in \eqref{equi00} yields that
\be\no
(\p_{y_2}^2+\p_{y_3}^2) W_6(y')-b_0\sum_{j=2}^3 \p_{y_j} \b(\frac{2 r_s W_j}{1+|y'|^2}\b)=2\sum_{j=2}^3\p_{y_j}\b(\frac{g_j({\bf W}(r_s,y'), W_6(y'))}{1+|y'|^2}\b).
\ee
This, together with \eqref{shock400}, implies that
\be\label{shock19}
\sum_{j=2}^3\p_{y_j}^2 W_1(r_s,y')-b_0 b_1\p_{y_j} \b(\frac{2r_s W_j}{1+|y'|^2}\b)(r_s,y') =q_1({\bf W}(r_s,y'),W_6),
\ee
with
\be\no
q_1({\bf W}(r_s,y'),W_6)=\sum_{j=2}^3 2b_1\p_{y_j}\b(\frac{g_j({\bf W}(r_s,y'), W_6(y'))}{1+|y'|^2}\b)+\p_{y_j}^2\{R_1({\bf W}(r_s,y'), W_6(y'))\}.
\ee
The condition \eqref{shock19} is used as the boundary condition on the shock front for the deformation-curl system. Using the equation \eqref{shock400}, the boundary condition in \eqref{equi00} becomes
\be\no
&&\sum_{j=2}^3 y_j\p_{y_j} W_1(r_s,y')-b_0b_1 r_s\sum_{j=2}^3 y_j W_j(r_s,y')=b_1\sum_{j=2}^3 y_j g_j({\bf W}(r_s,y'), W_6(y'))\\\label{shock20}
&&\quad+ \sum_{j=2}^3 y_j\p_{y_j}\{R_1({\bf W}(r_s,y'), W_6(y'))\},\ \ \text{ on } y_1=r_s, |y'|=1.
\ee

We turn to the determination of the vorticity. It follows from \eqref{vor21} that
\be\label{vor400}
\bigg(D_1^{W_6} +\frac{(1+|y'|^2)(W_2D_{2}^{W_6}+W_3D_3^{W_6})}{2D_0^{W_6}(\bar{U}(D_0^{W_6})+W_1)}\bigg)\tilde{\omega}_1+\mu({\bf W},W_6)\tilde{\omega}_1=H_0({\bf W},W_6),
\ee
where
\be\no
&&\mu({\bf W},W_6)=\frac{2}{D_0^{W_6}}-\frac{y_2 W_2 +y_3 W_3}{D_0^{W_6}(\bar{U}(D_0^{W_6})+W_1)}+\frac{1+|y'|^2}{2D_0^{W_6}}\sum_{i=2}^3D_i^{W_6}\b(\frac{W_i}{\bar{U}(D_0^{W_6})+W_1}\b),\\\no
&&H_0({\bf W}, W_6)\\\no
&&=-\frac{(1+|y'|^2)^2}{4(D_0^{W_6})^2}\bigg\{D_2^{W_6}\b(\frac{1}{\bar{U}(D_0^{W_6})+W_1}\b)D_3^{W_6} W_5-D_3^{W_6}\b(\frac{1}{\bar{U}(D_0^{W_6})+W_1}\b)D_2^{W_6} W_5\\\no
&&\quad-D_2^{W_6}\b(\frac{\bar{B}-\frac{1}{2}\bar{U}^2(D_0^{W_6})+W_5-\bar{U}(D_0^{W_6})W_1-\frac{1}{2}\sum_{j=1}^3W_j^2} {\gamma (\bar{K}+W_4)(\bar{U}(D_0^{W_6})+W_1)}\b)D_3^{W_6}W_4\\\no
&&\quad+D_3^{W_6}\b(\frac{\bar{B}-\frac{1}{2}\bar{U}^2(D_0^{W_6})+W_5-\bar{U}(D_0^{W_6})W_1-\frac{1}{2}\sum_{j=1}^3W_j^2} {\gamma (\bar{K}+W_4)(\bar{U}(D_0^{W_6})+W_1)}\b)D_2^{W_6}W_4\bigg\}.
\ee
The equation \eqref{shock17} produces the boundary data for $\tilde{\omega}_1$ at $y_1=r_s$
\be\no
&&\tilde{\omega}_1(r_s,y')=\frac{(1+|y'|^2)^2}{2 b_0 r_s^2} \bigg(\p_{y_3}\bigg(\frac{g_2({\bf W}(r_s,y'), W_6)}{1+|y'|^2}\bigg)-\p_{y_2}\bigg(\frac{g_3({\bf W}(r_s,y'), W_6)}{1+|y'|^2}\bigg)\bigg)\\\label{vor401}
&&\q\q\q+g_4({\bf W}(r_s,y'), W_6),
\ee
with
\be\no
&&g_4({\bf W}(r_s,y'), W_6)=-\frac{W_6(y_3 W_2-y_2 W_3)}{r_s(r_s+W_6)}-\frac{1+|y'|^2}{2}\frac{W_6(\p_{y_2} W_3-\p_{y_3} W_2)}{r_s(r_s+W_6)}\\\label{vor402}
&&\quad-\frac{1+|y'|^2}{2}\frac{(r_2-r_s)(\p_{y_2}W_6\p_{y_1}W_3-\p_{y_3}W_6\p_{y_1}W_2)}{(r_s+W_6)(r_2-r_s-W_6)}).
\ee
After solving the problem \eqref{vor400}-\eqref{vor401} for $\tilde{\omega}_1$, we can conclude from \eqref{vor22} that
\be\no
&&\tilde{\omega}_2=\frac{1+|y'|^2}{2 D_0^{W_6}}D_3^{W_6} W_1- D_1^{W_6} W_3-\frac{W_3}{D_0^{W_6}}\\\no
&&=\frac{W_2\tilde{\omega}_1}{\bar{U}(D_0^{W_6})+W_1} +\frac{1+|y'|^2}{2D_0^{W_6}(\bar{U}(D_0^{W_6})+W_1)} \bigg(D_3^{W_6} W_5\\\no
&&\q\q-\frac{\bar{B}-\frac{1}{2}\bar{U}^2(D_0^{W_6})+W_5-\bar{U}(D_0^{W_6})W_1-\frac{1}{2}\sum_{j=1}^3W_j^2} {\gamma (\bar{K}+W_4)}D_3^{W_6} W_4\bigg),\\\no
&&\tilde{\omega}_3=D_1^{W_6} W_2+ \frac{W_2}{D_0^{W_6}}-\frac{1+|y'|^2}{2 D_0^{W_6}}D_2^{W_6} W_1\\\no
&&=\frac{W_3\omega_1}{\bar{U}(D_0^{W_6})+W_1}- \frac{1+|y'|^2}{2D_0^{W_6} (\bar{U}(D_0^{W_6})+W_1)} \b(D_2^{W_6} W_5\\\no
&&\q\q-\frac{\bar{B}-\frac{1}{2}\bar{U}^2(D_0^{W_6})+ W_5-\bar{U}(D_0^{W_6})W_1-\frac{1}{2}\sum_{j=1}^3W_j^2} {\gamma (\bar{K}+W_4)}D_2^{W_6} W_4\b).
\ee

Collecting the principal terms and putting the quadratic terms on the right hand side, we obtain
\be\label{vor404}
&&\frac{1+|y'|^2}{2y_1}(\p_{y_2} W_3- \p_{y_3} W_2)+\frac{1}{y_1} (y_3 W_2- y_2 W_3)\\\no
&&\q\q=\tilde{\omega}_1(y)+H_1({\bf W},W_6),\\\label{vor405}
&&\frac{1+|y'|^2}{2y_1}\p_{y_3} W_1- \p_{y_1} W_3-\frac{W_3}{y_1} + \frac{\bar{B}-\frac12 \bar{U}^2(y_1)}{\gamma \bar{K} \bar{U}(y_1)} \frac{1+|y'|^2}{2y_1} \p_{y_3} W_4\\\no
&&\quad\quad=\frac{W_2\tilde{\omega}_1}{\bar{U}(D_0^{W_6})+W_1}+\frac{(1+|y'|^2)D_3^{W_6}W_5}{2D_0^{W_6}(\bar{U}(D_0^{W_6})+W_1)}+ H_2({\bf W},W_6),\\\label{vor406}
&&\p_{y_1} W_2+\frac{W_2}{y_1}-\frac{1+|y'|^2}{2 y_1} \p_{y_2} W_1 - \frac{\bar{B}-\frac12 \bar{U}^2(y_1)}{\gamma \bar{K} \bar{U}(y_1)} \frac{1+|y'|^2}{2y_1}\p_{y_2} W_4\\\no
&&\quad\quad=\frac{W_3\tilde{\omega}_1}{\bar{U}(D_0^{V_6})+W_1}-\frac{(1+|y'|^2)D_2^{W_6}W_5}{2D_0^{W_6}(\bar{U}(D_0^{W_6})+W_1)}+H_3({\bf W},W_6),
\ee
where
\be\no
&&H_1({\bf W},W_6)=\frac{D_0^{W_6}-y_1}{y_1 D_0^{W_6}}(y_3W_2-y_2 W_3)\\\no
&&\quad-\frac{1+|y'|^2}{2}\bigg(\frac{D_2^{W_6}W_3}{D_0^{W_6}}-\frac{1}{y_1}\p_{y_2} W_3-\frac{D_3^{W_6}W_2}{D_0^{W_6}}+\frac{1}{y_1}\p_{y_3}W_2\bigg),
\ee
and
\be\no
&&H_2({\bf W},W_6)=(D_1^{W_6}-\p_{y_1}) W_3+(\frac{1}{D_0^{W_6}}-\frac1{y_1}) W_3-\frac{1+|y'|^2}{2}\bigg(\frac{D_3^{W_6}}{D_0^{W_6}}-\frac{1}{y_1}\p_{y_3}\bigg)W_1\\\no
&&\quad-\frac{(1+|y'|^2)}{2D_0^{W_6}(\bar{U}(D_0^{W_6})+W_1)}\frac{W_5-\bar{U}(D_0^{W_6})W_1-\frac{1}{2}\sum_{j=1}^3W_j^2} {\gamma (\bar{K}+W_4)}D_3^{W_6} W_4\\\no
&&\quad-\frac{(1+|y'|^2)}{2D_0^{W_6}(\bar{U}(D_0^{W_6})+W_1)}\frac{\bar{B}-\frac12 \bar{U}^2(D_0^{W_6})}{\gamma (\bar{K}+W_4)}D_3^{W_6} W_4+ \frac{\bar{B}-\frac12 \bar{U}^2(y_1)}{\gamma \bar{K} \bar{U}(y_1)} \frac{1+|y'|^2}{2y_1} \p_{y_3} W_4,\\\no
&&H_3({\bf W},W_6)=-(D_1^{W_6}-\p_{y_1}) W_2-(\frac{1}{D_0^{W_6}}-\frac1{y_1}) W_2+\frac{1+|y'|^2}{2}\bigg(\frac{D_2^{W_6}}{D_0^{W_6}}-\frac{1}{y_1}\p_{y_2}\bigg)W_1\\\no
&&\quad+\frac{(1+|y'|^2)}{2D_0^{W_6}(\bar{U}(D_0^{W_6})+W_1)}\frac{W_5-\bar{U}(D_0^{W_6})W_1-\frac{1}{2}\sum_{j=1}^3W_j^2} {\gamma (\bar{K}+W_4)}D_2^{W_6} W_4\\\no
&&\quad+\frac{(1+|y'|^2)}{2D_0^{W_6}(\bar{U}(D_0^{W_6})+W_1)}\frac{\bar{B}-\frac12 \bar{U}^2(D_0^{W_6})}{\gamma (\bar{K}+W_4)}D_2^{W_6} W_4-\frac{\bar{B}-\frac12 \bar{U}^2(y_1)}{\gamma \bar{K} \bar{U}(y_1)} \frac{1+|y'|^2}{2y_1} \p_{y_2} W_4.
\ee

The boundary condition on the nozzle wall is
\be\label{slip2}
(y_2W_2+y_3 W_3)(y_1,y')=0, \ \ \forall y_1\in [r_s,r_2],\ \ |y'|=1.
\ee

We can further derive from \eqref{den15} that
\be\no
&&\displaystyle d_1(y_1)\p_{y_1} W_1+ \frac{1+|y'|^2}{2y_1}\sum_{j=2}^3 \p_{y_j} W_j +\frac{2W_1}{y_1}-\frac{1}{y_1}\sum_{j=2}^3 y_j W_j+d_2(y_1) W_1\\\label{den20}
&&\quad\quad =d_0(D_0^{W_6}) W_5+ G_0({\bf W},W_6),
\ee
with
\be\no
&&G_0({\bf W},W_6)= \mathbb{F}({\bf W},W_6)-\b(d_1(D_0^{W_6})D_1^{W_6} W_1-d_1(y_1)\p_{y_1} W_1\b)\\\no
&&\q-\frac{1+|y'|^2}{2}\sum_{j=2}^3 (\frac{1}{D_0^{W_6}}D_j^{W_6} W_j-\frac{1}{y_1}\p_{y_j}W_j)\\\no
&&\quad-\b(\frac{1}{D_0^{W_6}}-\frac{1}{y_1}\b)(2 W_1-\sum_{j=2}^3 y_j W_j))-(d_2(D_0^{W_6})-d_2(y_1)) W_1.
\ee
and
\be\no
&&\bar{c}^2(D_0^{W_6})\mathbb{F}({\bf W},W_6)=-\b((\gamma-1)W_5-(\gamma+1)\bar{U}(D_0^{W_6})W_1\b)D_{1}^{W_6} W_1\\\no
&&\quad-(\gamma-1)(W_5-\bar{U}(D_0^{W_6})W_1-\frac{1}{2}\sum_{j=1}^3 W_j^2)\b(\frac{1+|y'|^2}{2 D_0^{W_6}}\sum_{i=2}^3D_{i}^{W_6} W_i \\\no
&&\quad+\frac{1}{D_0^{W_6}} (2W_1-\sum_{i=2}^3 y_i W_i)\b)+ \frac{(\bar{U}'(D_0^{W_6})+D_1^{W_6}W_1)}{2}\b(W_1^2+(\gamma-1)\sum_{j=1}^3W_j^2\b)\\\no
&&\quad+\frac{(\gamma-1)\bar{U}(D_0^{W_6})}{ D_0^{W_6}}\sum_{j=1}^3 W_j^2 + (\bar{U}(D_0^{W_6})+W_1)\sum_{j=2}^3 W_j D_1^{W_6} W_j\\\no
&&\quad+ \frac{1+|y'|^2}{2D_0^{W_6}}\sum_{j=2}^3 W_j((\bar{U}(D_0^{W_6})+W_1)D_{j}^{W_6}W_1+\sum_{i=2}^3 W_i D_{j}^{W_6}W_i).
\ee

Finally, the boundary condition \eqref{pres2} at the exit becomes
\be\label{pres3}
&&W_1(r_2,y')+\frac{\bar{B}-\frac{1}{2}\bar{U}^2(r_2)}{\gamma \bar{K}\bar{U}(r_2)}W_4(r_2,y')=\frac{W_5(r_2,y')}{\bar{U}(r_2)}-\frac{\epsilon P_{ex}(y')}{(\bar{\rho}\bar{U})(r_2)}\\\no
&&\quad\quad\quad-\frac{1}{2\bar{U}(r_2)}\sum\limits_{j=1}^3 W_j^2(r_2,y')+ E({\bf W}(r_2,y'),W_6(y')),
\ee
where
\be\no
&&E({\bf W}(r_2,y'))=\frac{\gamma}{\gamma-1}(\bar{K}+W_4(r_2,y'))^{\frac1{\gamma}}(\tilde{P}({\bf W})(r_2,y'))^{\frac{\gamma-1}{\gamma}}-\frac{\gamma}{\gamma-1}\bar{K}^{\frac1{\gamma}}\bar{P}^{\frac{\gamma-1}{\gamma}}\\\label{err3}
&&\quad\quad-\frac{1}{\bar{\rho}(r_2)}(\tilde{P}({\bf W})(r_2,y')- \bar{P}(r_2))-\frac{\bar{B}-\frac{1}{2}\bar{U}^2(r_2)}{\gamma \bar{K}} W_4(r_2,y').
\ee

Therefore after the coordinates transformation \eqref{coor}, to solve the problem \eqref{euler-s} with \eqref{super1}, \eqref{slip}, \eqref{pres}, and \eqref{rh} is equivalent to solve the following problem:

{\bf Problem TS.} To find a function $W_6$ defined on $E$ and vector functions $(W_1,\cdots, W_5)$ defined on the $\mathbb{D}$, which solve the equations \eqref{ber31}-\eqref{ent31},\eqref{vor400},\eqref{vor404}-\eqref{vor406},\eqref{den20} with boundary conditions \eqref{shock400} ,\eqref{shock17},\eqref{shock19}, \eqref{shock20},\eqref{vor401},\eqref{slip2} and \eqref{pres3}.

Theorem \ref{existence} will follow directly from the following result.
\begin{theorem}\label{main}
{\it Assume the compatibility conditions \eqref{comp1} and \eqref{pressure-cp} hold. There exists a small constant $\epsilon_0>0$ depending only on the background solution and the boundary data $(U_{1,0}^-,U_{2,0}^-,U_{3,0}^-, P_0^-,K_0^-)\in (C^{2,\alpha}(\overline{E}))^5$, $P_{ex}\in C^{2,\alpha}(\overline{E})$ such that if $0\leq \epsilon<\epsilon_0$, the problem \eqref{ber31},\eqref{ent31},\eqref{vor400},\eqref{vor404}-\eqref{vor406},\eqref{den20} with boundary conditions \eqref{shock400}, \eqref{shock17}, \eqref{shock19},\eqref{shock20},\eqref{vor401},\eqref{slip2} and \eqref{pres3} has a unique solution $(W_1,W_2,W_3,W_4,W_5)(y)$ with the shock front $\mathcal{S}: y_1=W_6(y')$ satisfying the following properties.
\begin{enumerate}[(i)]
  \item The function $W_6(y')\in C^{3,\alpha}(\overline{E})$ satisfies
  \be\no
  \|W_6(y')-r_s\|_{C^{3,\alpha}(\overline{E})}\leq C_*\epsilon,
  \ee
  and
  \be\label{shock91}
  &&\p_{a}W_6(1,\tau)=0,\ \ \forall \tau\in \mathbb{T}_{2\pi},
  \ee
  where $C_*$ is a positive constant depending only on the background solution and the supersonic incoming flow and the exit pressure.
  \item The solution $(W_1,W_2,W_3,W_4,W_5)(y)\in C^{2,\alpha}(\overline{\mathbb{D}})$ satisfies the estimate
  \be\no
  \sum\limits_{j=1}^5\|W_j\|_{C^{2,\alpha}(\overline{\mathbb{D}})}\leq C_*\epsilon
  \ee
  and the compatibility conditions
  \be\label{sub52}
  \begin{cases}
  W_a(y_1,1,\tau)=(\p_a^2 W_a+ \p_a W_a)(y_1,1,\tau)=0,\ \forall (y_1,\tau)\in [r_s,r_2]\times \mathbb{T}_{2\pi},\\
  (\p_a W_1,\p_a W_{\tau},\p_a W_4,\p_a W_5))(y_1,1,\tau)=0,\ \ \forall (y_1,\tau)\in [r_s,r_2]\times \mathbb{T}_{2\pi}.
  \end{cases}\ee
  Here the polar coordinate $y_2=a\cos\tau, y_3=a\sin\tau$ is also used and the functions $W_a,W_{\tau}$ are defined as in \eqref{polar2} similarly.
\end{enumerate}
}\end{theorem}

\section{Proof of Theorem \ref{main}}\label{proof}

We are ready to prove Theorem \ref{main}. The solution class $\Xi$ consists of the vector functions
$$(W_1,\cdots, W_5,W_6)\in (C^{2,\alpha}(\overline{\mathbb{D}}))^5\times C^{3,\alpha}(\overline{E})$$
satisfying the compatibility conditions (which is precisely \eqref{shock91} and \eqref{sub52})
\be\label{cp100}\begin{cases}
(\p_a W_1,\p_a W_{\tau},\p_a W_4, \p_a W_5)(y_1,1,\tau)=0,\ \ \ \forall (y_1,\tau)\in [r_s,r_2]\times \mathbb{T}_{2\pi},\\
W_a(y_1,1,\tau)=(\p_a^2 W_a+\p_a W_a)(y_1,1,\tau)=0,\ \ \forall (y_1,\tau)\in [r_s,r_2]\times \mathbb{T}_{2\pi},\\
\p_a W_6(1,\tau)=0,\ \ \ \ \forall \tau\in \mathbb{T}_{2\pi}
\end{cases}\ee
and the estimate
\be\no
\|({\bf W},W_6)\|_{\Xi}:=\sum_{j=1}^5 \|W_j\|_{C^{2,\alpha}(\overline{\mathbb{D}})}+\|W_6\|_{C^{3,\alpha}(\overline{E})}\leq \delta_0.
\ee
with $\delta_0$ being a small positive constant to be determined later.

For any $(\hat{{\bf W}},\hat{W}_6)\in \Xi$, we will define an operator $\mathcal{T}$ mapping $\Xi$ to itself, the unique fixed point of $\mathcal{T}$ in $\Xi$ will be the solution to the {\bf Problem TS}. Note that the principal terms in the deformation-curl system \eqref{vor404}-\eqref{vor406} and \eqref{den99} contain hyperbolic quantities $W_4$ and $W_5$, it is still hyperbolic-elliptic coupled. Since $W_4$ and $W_5$ are conserved along the trajectory, one may represent $W_4$ and $W_5$ as functions of their boundary data at the entrance. Note that the boundary data for $W_4$ and $W_5$ involves the shock position $W_6$, using the condition \eqref{shock400}, one can see that the principal term in $W_4(y)$ is given by a scalar multiple of $W_1(r_s,y')$ and $W_5$ can be regarded as a higher order term. Substituting these into \eqref{vor404}-\eqref{vor406}, \eqref{den20} and \eqref{pres3}, we derive a deformation-curl first order elliptic system for $(W_1, W_2,W_3)$ containing a nonlocal term $W_1(r_s,y')$ with some boundary conditions involving only $(W_1, W_2,W_3)$, from which $(W_1,W_2,W_3)$ can be solved uniquely. Then the entropy $W_4$ and the shock front $W_6$ are uniquely determined. Now we give a detailed derivation for this procedure.

{\bf Step 1.} The shock front is uniquely determined by
\be\label{shock41}
W_6(y')=\frac{1}{b_1} W_1(r_s,y')-\frac{R_1(\hat{{\bf W}}(r_s,y'),\hat{W}_6)}{b_1},
\ee
provided $W_1(r_s,y')$ is uniquely determined.

{\bf Step 2.} We solve the transport equations for the Bernoulli's quantity and the entropy respectively. The Bernoulli's quantity will be determined by
\be\no\begin{cases}
\bigg(D_1^{\hat{W}_6}+\frac{1+|y'|^2}{2D_0^{\hat{W}_6}(\bar{U}(D_0^{\hat{W}_6})+\hat{W}_1)}(\hat{W}_2D_2^{\hat{W}_6}+\hat{W}_3D_3^{\hat{W}_6})\bigg) W_5=0,\\
W_5(r_s,y')=B^-(r_s+\hat{W}_6(y'),y')-\bar{B}^-.
\end{cases}\ee
Set
\be\no\begin{cases}
K_2(y):=\frac{(1+|y'|^2)(r_2-r_s-\hat{W}_6)\hat{W}_2}{2(r_2-r_s)D_0^{\hat{W}_6}(\bar{U}(D_0^{\hat{W}_6})+\hat{W}_1)+(1+|y'|^2)(y_1-r_2)\sum_{j=2}^3 \hat{W}_j \p_{y_j}\hat{W}_6},\\
K_3(y):=\frac{(1+|y'|^2)(r_2-r_s-\hat{W}_6)\hat{W}_3}{2(r_2-r_s)D_0^{\hat{W}_6}(\bar{U}(D_0^{\hat{W}_6})+\hat{W}_1)+(1+|y'|^2)(y_1-r_2)\sum_{j=2}^3 \hat{W}_j \p_{y_j}\hat{W}_6},
\end{cases}\ee
Then $K_2,K_3\in C^{2,\alpha}(\overline{\mathbb{D}})$. The function $W_5$ is conserved along the trajectory determined by the following ODE system
\begin{eqnarray}\label{char1} \left\{\begin{array}{ll}
\frac{d \bar{y}_2(t; y)}{dt}=K_2(t,\bar{y}_2(t;y),\bar{y}_3(t;y)),\ \ \forall t\in [r_s,r_2],\\
\frac{d \bar{y}_3(t; x)}{dt}=K_3(t,\bar{y}_2(t;y),\bar{y}_3(t;y)),\ \ \forall t\in [r_s,r_2],\\
\bar{y}_2(y_1; y)=y_2,\ \bar{y}_3(y_1;y)=y_3.
\end{array}\right. \end{eqnarray}
To our purpose, we need the polar coordinates $y_2=a\cos\tau, y_3=a\sin\tau$ and set
\be\no
&&K_a(y_1,a,\tau)= \cos\tau K_2 + \sin\tau K_3\\\no
&&=\frac{(1+a^2)(r_2-r_s-\hat{W}_6)\hat{W}_a}{2(r_2-r_s)D_0^{\hat{W}_6}(\bar{U}(D_0^{\hat{W}_6})+\hat{W}_1)+(1+a^2)(y_1-r_2)(\hat{W}_a \p_a\hat{W}_6+\frac{1}{a}\p_{\tau}\hat{W}_6)},\\\no
&&K_{\tau}(y_1,a,\tau)= -\sin\tau K_2+ \cos\tau K_3\\\no
&&=\frac{(1+a^2)(r_2-r_s-\hat{W}_6)\hat{W}_{\tau}}{2(r_2-r_s)D_0^{\hat{W}_6}(\bar{U}(D_0^{\hat{W}_6})+\hat{W}_1)+(1+a^2)(y_1-r_2)(\hat{W}_a \p_a\hat{W}_6+\frac{1}{a}\p_{\tau}\hat{W}_6)}.
\ee
Denote
\be\no
&&\bar{y}_2(t; y_1,a,\tau)=A(t;y_1,a,\tau)\cos \vartheta(t;y_1,a,\tau),\\\no
&&\bar{y}_3(t; y_1,a,\tau)=A(t;y_1,a,\tau)\sin \vartheta(t;y_1,a,\tau),
\ee
then $A^2(t;y_1,a,\tau)=\sum_{j=2}^3 \bar{y}_j^2(t; y_1,a,\tau)$ and
\be\no
&&\frac{d}{dt}A(t;y_1,a,\tau)=\frac{1}{A(t;y_1,a,\tau)}\sum_{j=2}^3 \bar{y}_j(t;y_1,a,\tau) K_j(t,\bar{y}_2(t;y_1,a,\tau),\bar{y}_3(t;y_1,a,\tau))\\\no
&&=K_a(t;A(t;y_1,a,\tau),\vartheta(t;y_1,a,\tau)),\ \ \forall t\in [r_s,r_2].
\ee
Since $y_2 \hat{W}_2 +y_3 \hat{W}_3=0$ holds on $\Sigma_w$, $K_a(y_1,1,\tau)=0$ on $\Sigma_w$. Then the uniqueness theorem for the ODE theory implies that for any $y\in \Sigma_w$, i.e. $y_2^2+y_3^2=1$, one has
\be\label{char21}
A(t;y_1,1,\tau)=\sqrt{\bar{y}_2^2(t; y_1,1,\tau)+\bar{y}_3^2(t; y_1,1,\tau)}=1,\forall(t,y_1,\tau)\in [r_s,r_2]^2\times \mathbb{T}_{2\pi}.
\ee

Denote $\vec{\beta}(y)=(\beta_2(y),\beta_3(y))=(\bar{y}_2(r_s;y),\bar{y}_3(r_s;y))$, then
\be\label{char22}
\beta_2^2(y_1,1,\tau)+ \beta_3^2(y_1,1,\tau)= 1,\ \ (y_1,\tau)\in [r_s,r_2]\times \mathbb{T}_{2\pi}.
\ee
It follows from \eqref{char1} that
\be\no
\sum_{j=2}^3\|\beta_j(y)-y_j\|_{C^{2,\alpha}(\overline{\mathbb{D}})}\leq C_*\|(\hat{{\bf W}},\hat{W}_6)\|_{\Xi}.
\ee

Since $W_5$ is conserved along the trajectory, one has
\be\no
W_5(y)=W_5(r_s,\vec{\beta}(y))=B^-(r_s+\hat{W}_6(\vec{\beta}(y)),\vec{\beta}(y))-\bar{B}^-
\ee
and the following estimate holds
\be\label{ber43}
&&\|W_5\|_{C^{2,\alpha}(\overline{\mathbb{D}})}\leq C_*\epsilon+C_*\epsilon (\|\hat{W}_6\|_{C^{2,\alpha}(\overline{E})}+\sum_{j=2}^3\|\beta_j-y_j\|_{C^{2,\alpha}(\overline{\mathbb{D}})})\\\no
&&\leq C_*(\epsilon+\epsilon\|(\hat{{\bf W}}, \hat{W}_6)\|_{\Xi})\leq C_*(\epsilon+\epsilon \delta_0).
\ee

Employing the compatibility conditions \eqref{cp100}, simple calculations show that
\be\no
K_a(y_1,1,\tau)=\p_a\bigg(\frac{K_{\tau}}{a}\bigg)(y_1,1,\tau)=0,\ \ \forall (y_1,\tau)\in [r_s,r_2]\times \mathbb{T}_{2\pi}.
\ee
To prove that $W_5$ satisfies the compatibility condition, one may regard $W_5$ as a function of $(y_1,a,\tau)$, which satisfies
\be\no\begin{cases}
\p_{y_1} W_5 + K_a \p_{a} W_5 + \frac{K_{\tau}}{a} \p_{\tau} W_5=0,\ \\
W_5(r_s,a,\tau)=B^-(r_s+\hat{W}_6(a,\tau),a,\tau)-\bar{B}^-.
\end{cases}\ee
Differentiating with respect to $a$, and restricting the resulting equation on $a=1$, one gets
\be\no\begin{cases}
\bigg(\p_{y_1} (\p_a W_5) +\frac{K_{\tau}}{a}\p_{\tau} \p_a W_5+ \p_a K_a \p_{a} W_5\bigg)(y_1,1,\tau)=0, \forall (y_1,\tau)\in [r_s,r_2]\times \mathbb{T}_{2\pi},\\
(\p_a W_5)(r_s,1,\tau)=\p_{z_1} B^-(r_s+\hat{W}_6(1,\tau),1,\tau)\p_a \hat{W}_6(1,\tau)\\
\q\q\q\q\q+\p_a B^-(r_s+\hat{W}_6(1,\tau),1,\tau)=0,\ \ \forall \tau\in \mathbb{T}_{2\pi}.
\end{cases}\ee
By the uniqueness theorem of the ordinary differential equations, one has
\be\label{ber48}
\p_a W_5(y_1,1,\tau)=0, \forall (y_1,\tau)\in [r_s,r_2]\times \mathbb{T}_{2\pi}.
\ee

The function $W_4$ satisfies
\be\no\begin{cases}
\bigg(D_1^{\hat{W}_6}+\frac{1+|y'|^2}{2D_0^{\hat{W}_6}(\bar{U}(D_0^{\hat{W}_6})+\hat{W}_1)}(\hat{W}_2D_2^{\hat{W}_6}+\hat{W}_3D_3^{\hat{W}_6})\bigg)W_4=0,\\
W_4(r_s,y')= b_2 W_6(y')+R_2(\hat{{\bf W}}(r_s,y'),\hat{W}_6(y')),\ \ \forall y'\in \overline{E}.
\end{cases}\ee

By the characteristic method and the equation \eqref{shock41}, one has
\be\no
&&W_4(y)=W_4(r_s,\vec{\beta}(y))\\\label{ent42}
&&=b_2 W_6(\vec{\beta}(y))+ R_2(\hat{{\bf W}}(r_s,\vec{\beta}(y)),\hat{W}_6(\vec{\beta}(y)))\\\no
&&=b_2 W_6(y')+ b_2(W_6(\vec{\beta}(y))-W_6(y'))+ R_2(\hat{{\bf W}}(r_s,\vec{\beta}(y)),\hat{W}_6(\vec{\beta}(y)))\\\no
&&=\frac{b_2}{b_1} W_1(r_s,y')+b_2(W_6(\vec{\beta}(y))-W_6(y'))+R_3(\hat{{\bf V}}(r_s,\vec{\beta}(y)),\hat{W}_6(\vec{\beta}(y))),
\ee
where
\be\no
R_3(\hat{{\bf W}}(r_s,y'),\hat{W}_6)=R_2(\hat{{\bf W}}(r_s,y'),\hat{W}_6)-\frac{b_2 R_1(\hat{{\bf W}}(r_s,y'),\hat{W}_6)}{b_1}
\ee
Since $W_6(y')$ is still unknown, one may replace \eqref{ent42} by
\be\label{ent43}
W_4(y_1,y')=\frac{b_2}{b_1} W_1(r_s,y')+R_4(\hat{{\bf W}}(r_s,\vec{\beta}(y)),\hat{W}_6(\vec{\beta}(y))),
\ee
where
\be\no
R_4=b_2(\hat{W}_6(\vec{\beta}(y))-\hat{W}_6(y'))+R_3(\hat{{\bf W}}(r_s,\vec{\beta}(y)),\hat{W}_6(\vec{\beta}(y))).
\ee
Therefore $W_4$ is decomposed as a scalar multiple of $W_1(r_s,y')$ with higher order terms satisfying
\be\no
&&\|W_4\|_{C^{2,\alpha}(\overline{\mathbb{D}})}\leq C_*\|W_1(r_s,\cdot)\|_{C^{2,\alpha}(\overline{E})}+\|R_4\|_{C^{2,\alpha}(\overline{\mathbb{D}})}\\\no
&&\leq C_*(\|W_1(r_s,\cdot)\|_{C^{2,\alpha}(\overline{E})}+\|\hat{W}_6\|_{C^{3,\alpha}(\overline{E})}\sum_{j=2}^3\|\beta_j(y)-y_j\|_{C^{2,\alpha}(\overline{\mathbb{D}})})\\\no
&&\q\q+ C_*(\epsilon \|(\hat{{\bf W}}, \hat{W}_6)\|_{\Xi}+\|(\hat{{\bf W}}, \hat{W}_6)\|_{\Xi}^2)\\\no
&&\leq C_*\|W_1(r_s,\cdot)\|_{C^{2,\alpha}(\overline{E})}+ C_*(\epsilon \delta_0+ \delta_0^2).
\ee

Furthermore, since $(\hat{{\bf W}}, \hat{W}_6)\in\Sigma$ satisfies the compatibility conditions \eqref{cp100} and the supersonic incoming flows $\bm{\Psi}^-$ satisfies \eqref{comp2}, using the formulas in the Appendix, one could verify by direct but tedious computations that for $J_a=(\cos\tau J_2 + \sin\tau J_3)$ and $J_{\tau}=-\sin\tau J_2 + \cos\tau J_3$
\be\label{j211}\begin{cases}
J_a(\hat{{\bf W}}(r_s,1,\tau),\hat{W}_6(1,\tau))=\p_a\{ J_{\tau}(\hat{{\bf W}}(r_s,1,\tau),\hat{W}_6(1,\tau))\}=0,\ \forall \tau\in \mathbb{T}_{2\pi},\\
\cos\tau g_2(\hat{{\bf W}}(r_s,1,\tau),\hat{W}_6(1,\tau))+ \sin\tau g_3(\hat{{\bf W}}(r_s,1,\tau),\hat{W}_6(1,\tau))=0,\ \forall \tau\in \mathbb{T}_{2\pi},\\
\p_{a}\{J(\hat{{\bf W}}(r_s,1,\tau),\hat{W}_6(1,\tau))\}=0,\ \forall \tau\in \mathbb{T}_{2\pi},
\end{cases}\ee
and for all $j=1,2,3$
\be\label{rv11}
\p_{a}\{ R_{0j}(\hat{{\bf W}}(r_s,1,\tau),\hat{W}_6(1,\tau))\}=0,\ \ \forall \tau\in \mathbb{T}_{2\pi}.
\ee
Thus for $k=1,2,3$
\be\label{rv41}
\p_{a}\{ R_{k}(\hat{{\bf W}}(r_s,1,\tau),\hat{W}_6(1,\tau))\}=0,\ \ \forall \tau\in \mathbb{T}_{2\pi}.
\ee
The proof of \eqref{j211}-\eqref{rv41} is delayed to the Appendix \S\ref{appendix}.

Note that $b_2\hat{W}_6(\vec{\beta}(y))+R_3(\hat{{\bf W}}(r_s,\vec{\beta}(y)),\hat{W}_6(\vec{\beta}(y)))$ is the unique solution to the transport equation
\be\no\begin{cases}
\bigg(D_1^{\hat{W}_6}+\frac{1+|y'|^2}{2D_0^{\hat{W}_6}(\bar{U}(D_0^{\hat{W}_6})+\hat{W}_1)}(\hat{W}_2D_2^{\hat{W}_6}+\hat{W}_3D_3^{\hat{W}_6})\bigg) H(y)=0,\ \ \forall y\in \mathbb{D},\\
H(r_s,y')=b_2\hat{W}_6(y')+R_3(\hat{{\bf W}}(r_s,y'),\hat{W}_6(y')),\ \ \forall y'\in E.
\end{cases}
\ee
Using $\p_a \hat{W}_6(1,\tau)=0$ and \eqref{rv41}, as in \eqref{ber48} one can prove that for all $(y_1,\tau)\in [r_s,r_2]\times\mathbb{T}_{2\pi}$
\be\no
&&\p_a \{b_2\hat{W}_6(\vec{\beta}(y))+R_3(\hat{{\bf W}}(r_s,\vec{\beta}(y)),\hat{W}_6(\vec{\beta}(y)))\}(y_1,1,\tau)=0,\\\label{r4cp}
&&\p_a \{R_{4}(\hat{{\bf W}}(r_s,\vec{\beta}(y)),\hat{W}_6(\vec{\beta}(y)))\}(y_1,1,\tau)=0,
\ee
and
\be\no
\p_{a} W_4(y_1,1,\tau)=\frac{b_2}{b_1}\p_{a} W_1(r_s,1,\tau),\ \ \forall (y_1,\tau)\in [r_s,r_2]\times \mathbb{T}_{2\pi}.
\ee

\begin{remark}
{\it The $C^{3,\alpha}(\overline{E})$ norm of $\hat{W}_6$ is required for the $C^{2,\alpha}(\overline{\mathbb{D}})$ estimate of the first term in $R_4$.
}\end{remark}

{\bf Step 3.} Thanks to \eqref{vor400}-\eqref{vor401}, we solve the following problem for the first component of the vorticity
\be\label{vor501}\begin{cases}
\bigg(D_1^{\hat{W}_6}+\frac{(1+|y'|^2)(\hat{W}_2D_2^{\hat{W}_6}+\hat{W}_3D_3^{\hat{W}_6})}{2D_0^{\hat{W}_6}(\bar{U}(D_0^{\hat{W}_6})+\hat{W}_1)}\bigg) \tilde{{\omega}}_1 + \mu(\hat{{\bf W}},\hat{W}_6)\tilde{\omega}_1=H_0(\hat{{\bf W}},\hat{W}_6),\\
\tilde{\omega}_1(r_s,y')=R_6(\hat{{\bf W}}(r_s,y'), \hat{W}_6(y')),\ \ \forall y'\in \overline{E},
\end{cases}\ee
where
\be\no
&&R_6(\hat{{\bf W}}(r_s,y'), \hat{W}_6(y'))=\frac{(1+|y'|^2)^2}{2b_0r_s^2}\b(\p_{y_3} \b\{\frac{g_2(\hat{{\bf W}}(r_s,y'), \hat{W}_6(y'))}{1+|y'|^2}\b\}\\\no
&&\quad\quad\quad-\p_{y_2}\b\{\frac{g_3(\hat{{\bf W}}(r_s,y'), \hat{W}_6(y'))}{1+|y'|^2}\b\}\b)+g_4(\hat{{\bf W}}(r_s,y'), \hat{W}_6(y')).
\ee

Integrating the first equation in \eqref{vor501} along the trajectory $(\tau,\bar{y}_2(\tau;y),\bar{y}_3(\tau;y))$ yields
\begin{eqnarray}\label{vor502}
&&\tilde{\omega}_1(y)= R_6(\hat{{\bf W}}(r_s,\vec{\beta}(y)), \hat{W}_6(\vec{\beta}(y))) e^{-\int_{r_s}^{y_1} \mu(\hat{{\bf W}},\hat{W}_6)(t;\bar{y}_2(t;y),\bar{y}_3(t;y))dt}\\\nonumber
&&\quad\quad\quad+ \int_{r_s}^{y_1} H_0(\hat{{\bf W}},\hat{W}_6)(\tau,\bar{y}_2(\tau;y),\bar{y}_3(\tau;y)) e^{-\int_{\tau}^{y_1} \mu(\hat{{\bf W}},\hat{W}_6)(t;\bar{y}_2(t;y),\bar{y}_3(t;y))dt} d\tau.
\end{eqnarray}
Thus the following estimate holds
\be\no
&&\|\tilde{\omega}_1\|_{C^{1,\alpha}(\overline{\mathbb{D}})}\leq C_*(\|\tilde{\omega}_1(r_s,\cdot)\|_{C^{1,\alpha}(\overline{E})}+\|H_0(\hat{{\bf W}},\hat{W}_6)\|_{C^{1,\alpha}(\overline{\mathbb{D}})})\\\no
&&\leq C_*(\epsilon\|(\hat{{\bf W}}, \hat{W}_6)\|_{\Xi}+\|(\hat{{\bf W}}, \hat{W}_6)\|_{\Xi}^2)\leq C_*(\epsilon\delta_0+\delta_0^2).
\ee
It is verified in Appendix \S\ref{appendix} that
\be\label{om1}
&&\tilde{\omega}_1(r_s,1,\tau)=R_6(\hat{{\bf W}}(r_s,1,\tau), \hat{W}_6(1,\tau))=0,\ \ \ \forall \tau\in \mathbb{T}_{2\pi},\\\label{om2}
&&H_0(\hat{{\bf W}},\hat{W}_6)(y_1,1,\tau)=0,\ \ \ \forall (y_1,\tau)\in [r_s,r_2]\times\mathbb{T}_{2\pi}
\ee
Then it follows from \eqref{char21}-\eqref{char22} and \eqref{vor502} that
\be\label{om3}
\tilde{\omega}_1(y_1,1,\tau)=0,\ \ \ \forall (y_1,\tau)\in [r_s,r_2]\times\mathbb{T}_{2\pi}.
\ee

Substituting \eqref{vor502} and \eqref{ent43} into \eqref{vor404}-\eqref{vor406}, one gets
\be\label{vor504}
&&\frac{1+|y'|^2}{2y_1}(\p_{y_2} W_3-\p_{y_3} W_2++\frac{1}{y_1}(y_3 W_2-y_2 W_3)=G_1(\hat{{\bf W}},\hat{W}_6),\\\label{vor505}
&&\frac{1+|y'|^2}{2 y_1}\p_{y_3} W_1- \p_{y_1} W_3-\frac{W_3}{y_1}+ d_3(y_1)\frac{1+|y'|^2}{2y_1}\p_{y_3} W_1(r_s,y')\\\no
&&\q\q\q=G_2(W_5,\hat{{\bf W}},\hat{W}_6),\\\label{vor506}
&&\p_{y_1} W_2+\frac{W_2}{y_1}-\frac{1+|y'|^2}{2 y_1} \p_{y_2} W_1 - d_3(y_1) \frac{1+|y'|^2}{2y_1}\p_{y_2}W_1(r_s,y')\\\no
&&\q\q\q=G_3(W_5,\hat{{\bf W}},\hat{W}_6),
\ee
where
\be\no
&&d_3(y_1)=\frac{b_2}{b_1}\frac{\bar{B}-\frac12 \bar{U}^2(y_1)}{\gamma \bar{K} \bar{U}(y_1)},\ \ \ G_1(\hat{{\bf W}},\hat{W}_6)=\tilde{\omega}_1(y)+H_1(\hat{{\bf W}},\hat{W}_6),\\\no
&&G_2(W_5,\hat{{\bf W}},\hat{W}_6)=\frac{\hat{W}_2\tilde{\omega}_1}{\bar{U}(D_0^{\hat{W}_6})+\hat{W}_1}+\frac{(1+|y'|^2) D_3^{\hat{W}_6}W_5}{2D_0^{\hat{W}_6}(\bar{U}(D_0^{\hat{W}_6})+\hat{W}_1)}\\\no
&&\quad\quad+ H_2(\hat{{\bf W}},\hat{W}_6)+ \frac{\bar{B}-\frac12 \bar{U}^2(y_1)}{\gamma \bar{K} \bar{U}(y_1)} \frac{1+|y'|^2}{2y_1}\p_{y_3} \{R_4(\hat{{\bf W}},\hat{W}_6)\},\\\no
&&G_3(W_5,\hat{{\bf W}},\hat{W}_6)=\frac{\hat{W}_3\tilde{\omega}_1}{\bar{U}(D_0^{\hat{W}_6})+\hat{W}_1}-\frac{(1+|y'|^2) D_2^{\hat{W}_6}W_5}{2D_0^{\hat{W}_6}(\bar{U}(D_0^{\hat{W}_6})+\hat{W}_1)}\\\no
&&\quad\quad+H_3(\hat{{\bf W}},\hat{W}_6)-\frac{\bar{B}-\frac12 \bar{U}^2(y_1)}{\gamma \bar{K} \bar{U}(y_1)} \frac{1+|y'|^2}{2y_1}\p_{y_2}  \{R_4(\hat{{\bf W}},\hat{W}_6)\}.
\ee

Furthermore, \eqref{den20} implies that
\be\label{den99}
&&d_1(y_1) \p_{y_1} W_1+\frac{1+|y'|^2}{2y_1}\sum_{j=2}^3\p_{y_j} W_j+\frac{2W_1}{y_1} -\frac{1}{y_1}\sum_{j=2}^3 y_j W_j\\\no
&&\quad\quad\quad\quad\quad+ d_2(y_1)W_1=d_0(D_0^{\hat{W}_6}) W_5+G_0(\hat{{\bf W}},\hat{W}_6), \ \ \text{in }\mathbb{D}.
\ee

Using the equation \eqref{ent43}, the boundary condition \eqref{pres3} at the exit reduces
\be\label{pres4}
W_1(r_2,y')+ d_3(r_2) W_1(r_s,y')=q_4(y'),
\ee
where
\be\no
&&q_4(y')=d_3(r_2) R_4(\hat{{\bf W}}(r_s,\vec{\beta}(r_2,y')),\hat{W}_6(\vec{\beta}(r_2,y')))+
\frac{W_5(r_2,y')}{\bar{U}(r_2)}\\\no
&&\quad-\frac{\epsilon P_{ex}(y')}{(\bar{\rho}\bar{U})(r_2)}-\frac{1}{2\bar{U}(r_2)}\sum_{j=1}^3 \hat{W}_j^2(r_2,y')+ E(\hat{{\bf W}}(r_2,y'),\hat{W}_6(y')).
\ee
It follows from \eqref{pressure-cp},\eqref{cp100}, \eqref{ber48},\eqref{r4cp} and the exact expression of $E(\hat{{\bf W}}(r_2,\cdot),\hat{W}_6(\cdot))$ in \eqref{err3} that
\be\label{pres5}
\p_a q_4(1,\tau)=0,\ \ \ \forall \tau\in \mathbb{T}_{2\pi}.
\ee

{\bf Step 4.} We have derived a deformation-curl system for the velocity field which consists of the equations \eqref{vor504}-\eqref{vor506}, \eqref{den99} supplemented with the boundary conditions \eqref{shock19},\eqref{shock20},\eqref{slip2} and \eqref{pres4}, where $q_1, R_1$ and $g_i (i=2,3)$ are evaluated at $(\hat{{\bf W}},\hat{W}_6)$. However, due to the linearization, the vector field $(G_1,G_2,G_3)(\hat{{\bf W}},\hat{W}_6)$ may not be divergence free and thus the solvability condition for the curl system does not hold in general. To overcome this obstacle, we enlarge the deformation-curl system by adding a new unknown function $\Upsilon$ with additional mixed boundary condition for $\Upsilon$:
\be\lab{den32}\begin{cases}
d_1(y_1) \p_{y_1} W_1+\frac{1+|y'|^2}{2y_1}\sum_{j=2}^3\p_{y_j} W_j+\frac{2W_1}{y_1} \\
\quad\quad\quad\quad-\frac{1}{y_1}\sum_{j=2}^3 y_j W_j+ d_2(y_1)W_1=d_0(D_0^{\hat{W}_6}) W_5+G_0(\hat{{\bf W}},\hat{W}_6),\\
\frac{1+|y'|^2}{2y_1}(\p_{y_2} W_3-\p_{y_3} W_2)+ \frac{1}{y_1}(y_3 W_2-y_2 W_3)+\p_{y_1}\Upsilon=G_1(\hat{{\bf W}},\hat{W}_6),\\
\frac{1+|y'|^2}{2y_1}\p_{y_3}(W_1+d_3(y_1) W_1(r_s,y'))-\p_{y_1} W_3-\frac{W_3}{y_1}+\frac{1+|y'|^2}{2y_1}\p_{y_2}\Upsilon= G_2(W_5,\hat{{\bf W}},\hat{W}_6),\\
\p_{y_1} W_2+\frac{W_2}{y_1}-\frac{1+|y'|^2}{2 y_1}\p_{y_2}(W_1+d_3(y_1) W_1(r_s,y'))+\frac{1+|y'|^2}{2y_1}\p_{y_3}\Upsilon = G_3(W_5,\hat{{\bf W}},\hat{W}_6),
\end{cases}\ee
with the boundary conditions
\be\label{den320}\begin{cases}
\sum_{j=2}^3\p_{y_j}^2W_1(r_s,y')-b_0 b_1\p_{y_j}\b(\frac{2 r_s W_j}{1+|y'|^2}\b)(r_s,y')+q_1(\hat{{\bf W}}(r_s,y'),\hat{W}_6(y')),\forall y'\in E,\\
W_1(r_2,y')+ d_3(r_2)W_1(r_s,y')=q_4(y'),\ \forall y'\in E,\\
\p_{y_1}\Upsilon(r_s,y')=\p_{y_1}\Upsilon(r_2,y')=0,\ \forall y'\in E,\\
\sum_{j=2}^3y_j W_j(y_1,y')=\Upsilon(y_1,y')=0,\ \forall (y_1,y')\in \Sigma_w,\\
\sum_{j=2}^3 \b(y_j\p_{y_j} W_1(r_s,y')-b_0b_1 r_s y_j W_j(r_s,y')\b)=b_1\sum_{j=2}^3 y_j g_j(\hat{{\bf W}}(r_s,y'), \hat{W}_6(y'))\\
\quad+ \sum_{j=2}^3 y_j\p_{y_j}\{R_1(\hat{{\bf W}}(r_s,y'), \hat{W}_6(y'))\},\ \text{on }\ y_1=r_s, |y'|=1.
\end{cases}\ee

Due to \eqref{j211} and \eqref{rv41}, the last boundary condition in \eqref{den320} reduces to
\be\no
\sum_{j=2}^3\bigg(y_j\p_{y_j} W_1(r_s,y')-b_0b_1 r_s y_j W_j(r_s,y')\bigg)=0,\ \ \ \text{on }\ y_1=r_s \ |y'|=1.
\ee

One can verify the following compatibility conditions (See the Appendix \S\ref{appendix} for the detailed proof):
\be\label{g123}\begin{cases}
G_1(\hat{{\bf W}},\hat{W}_6)(y_1,1,\tau)=0,\ \ \ \forall (y_1,\tau)\in [r_s,r_2]\times \mathbb{T}_{2\pi},\\
G_{\tau}(W_5,\hat{{\bf W}},\hat{W}_6)(y_1,1,\tau)=0,\ \ \forall (y_1,\tau)\in [r_s,r_2]\times \mathbb{T}_{2\pi},\\
\p_a \{G_a(W_5,\hat{{\bf W}},\hat{W}_6)\}(y_1,1,\tau)=0,\ \ \forall (y_1,\tau)\in [r_s,r_2]\times \mathbb{T}_{2\pi},
\end{cases}\ee
where $G_a=\cos\tau G_2 +\sin\tau G_3$ and $G_{\tau}=-\sin\tau G_2+ \cos\tau G_3$.

The unique solvability of the problem \eqref{den32} is divided into several steps.

{\bf Step 4.1} First, taking the divergence operator for the second, third and fourth equations in \eqref{den32} leads to
\be\label{den33}\begin{cases}
\p_{y_1}^2 \Upsilon + \frac{1+|y'|^2}{4y_1^2}\sum_{j=2}^3\p_{y_j}((1+|y'|^2)\p_{y_j}\Upsilon) + \frac{2}{y_1}\p_{y_1}\Upsilon-\frac{1+|y'|^2}{2 y_1^2}\sum_{j=2}^3y_j\p_{y_j} \Upsilon\\
\quad\quad\quad= \p_{y_1} G_1 + \frac{1+|y'|^2}{2y_1}\sum_{j=2}^3\p_{y_j} G_j +\frac{2 G_1}{y_1}- \frac{1}{y_1}\sum_{j=2}^3 y_j G_j,\ \text{in }\mathbb{D},\\
\p_{y_1}\Upsilon(r_s,y')=\p_{y_1}\Upsilon(r_2,y')=0,\ \forall y'\in E,\\
\displaystyle\Upsilon(y_1,y')=0,\ \ \forall (y_1,y')\in \Sigma_w.
\end{cases}\ee

\begin{lemma}\label{lapace1}
{\it Given any vector field $(G_1,G_2,G_3)\in (C^{1,\alpha}(\overline{\mathbb{D}}))^3$ satisfying \eqref{g123}, there exists a unique solution $\Upsilon\in (C^{2,\alpha}(\overline{\mathbb{D}}))^3$ to the problem \eqref{den32} with
\be\no
\|\Upsilon\|_{C^{2,\alpha}(\overline{\mathbb{D}})}\leq C_*\sum_{i=1}^3\|\dot{G}_i\|_{C^{1,\alpha}(\overline{\mathbb{D}})}\leq C_*(\epsilon\|(\hat{{\bf W}}, \hat{W}_6)\|_{\Xi}+\|(\hat{{\bf W}}, \hat{W}_6)\|_{\Xi}^2).
\ee
}\end{lemma}

\begin{proof}

Define the spherical projection coordinates
\be\no
&&\mathscr{S}:x\in \Omega_+=\{(x_1,x_2,x_3): r_s<\sqrt{x_1^2+x_2^2+x_3^2}<r_2, x_1>0\}\to y\in \mathbb{D}\\\label{ss}
&&\quad\quad y_1=\sqrt{x_1^2+x_2^2+x_3^2}, \ \ y_i=\frac{x_i}{x_1+\sqrt{x_1^2+x_2^2+x_3^2}},\ \ i=2,3,
\ee
then $\mathscr{S}$ is a one to one and onto smooth mapping from $\Omega_+$ to $\mathbb{D}$ with its inverse $\mathscr{S}^{-1}$.

Define
\be\no
\bm{\mathcal{G}}(x)=\sum_{i=1}^3 \mathcal{G}_i(x) {\bf e}_i(x)=\left(\begin{array}{lll}
\mathcal{G}_1(x)\\
\mathcal{G}_2(x)\\
\mathcal{G}_3(x)
\end{array}\right)=\sum_{i=1}^3 G_i(W_5,\hat{{\bf W}},\hat{W}_6)\tilde{{\bf e}}_i(y),
\ee
where $\{\tilde{{\bf e}}_i(y)\}_{i=1,2,3}$ is the basis of $\mathbb{R}^3$ defined in \eqref{basis} with $z$ replaced by $y$. Then $\mathcal{G}_i\in C^{1,\alpha}(\overline{\Omega_+})$. By \cite[Theorem 1.1]{lie86}, there exists a unique smooth solution $\tilde{\Upsilon}(x)\in C^{2,\alpha}(\Omega_+)\cap C^0(\overline{\Omega_+})$ to the following boundary value problem:
\be\label{d331}\begin{cases}
\displaystyle\sum_{i=1}^3 \p_{x_i}^2 \tilde{\Upsilon}(x)= \sum_{i=1}^3 \p_{x_i}\mathcal{G}_i(x),\ \ \forall x\in \Omega_+,\\
\tilde{\Upsilon}(0,x')=0,\ \ \ \forall r_s\leq |x'|=\sqrt{x_2^2+x_3^2}\leq r_2,\\
\sum_{i=1}^3 x_j \p_{x_j}\tilde{\Upsilon}(x)=0,\ \ \forall x\in \{x\in \mathbb{R}^3: x_1>0, |x|=r_s\},\\
\sum_{i=1}^3 x_j \p_{x_j}\tilde{\Upsilon}(x)=0,\ \ \forall x\in \{x\in \mathbb{R}^3: x_1>0, |x|=r_2\}.
\end{cases}\ee
Away from the circles $\{(0,x'): |x'|=r_s \text{ or } |x'|=r_2\}$, $\tilde{\Upsilon}$ is $C^{2,\alpha}$ smooth up to the boundary. Define the function $\Upsilon(y)$ on $\mathbb{D}$ through the transformation \eqref{ss}:
\be\no
\Upsilon(y)=\tilde{\Upsilon}(\mathscr{S}^{-1}(y)), \ \ \forall y\in \mathbb{D}
\ee
then $\Upsilon\in C^{2,\alpha}(\mathbb{D})$ and away from the circles $\{(r_s,y'):|y'|=1\}\cup \{(r_2,y'):|y'|=1\}$ it is also $C^{2,\alpha}$ smooth up to the boundary. Thanks to \eqref{z1}, $\Upsilon$ solves the boundary value problem \eqref{den33}. It remains to improve the $C^{2,\alpha}$ regularity of $\Upsilon$ up to the whole boundary.

Note that
\be\no
&&\mathcal{G}_2(x_1,x')=\frac{x_2}{r}G_1+(1-\frac{x_2^2}{r(x_1+r)})G_2-\frac{x_2x_3}{r(x_1+r)} G_3,\\\no
&&\mathcal{G}_2(x_1,x')=\frac{x_3}{r}G_1-\frac{x_2x_3}{r(x_1+r)}G_2 +(1-\frac{x_3^2}{r(x_1+r)})G_3,\\\no
&&\p_{x_1}\mathcal{G}_1(x_1,x')=\sum_{i=1}^3\b(\frac{x_1 x_i}{r^2}\p_{z_1} G_i-\frac{x_i}{r^2(x_1+r)}(x_2\p_{z_2}+x_3\p_{y_3})G_i\b)\\\no
&&\q\q\q+\b(\frac1{r}-\frac{x_1^3}{r^3}\b)G_1+\frac{x_1}{r^2}(x_2G_2+x_3 G_3),\\\no
&& a\p_a G_a(y)=(y_2\p_{y_2}+y_3\p_{y_3})(\frac{y_2G_2+y_3 G_3}{a})=\sum_{i=2}^3y_i(y_2\p_{y_2}G_i+y_3 \p_{y_3} G_i).
\ee

Thus, in the coordinates $x=(x_1,x')$, the compatibility condition \eqref{g123} becomes
\be\no\begin{cases}
\p_{x_1}\mathcal{G}_1(0,x')=0,\ \ \ \forall x'\in \{x'\in\mathbb{R}^2: r_s<|x'|<r_2\},\\
\mathcal{G}_2(0,x')=\mathcal{G}_3(0,x')=0,\ \ \ \forall x'\in \{x'\in\mathbb{R}^2: r_s<|x'|<r_2\}.
\end{cases}\ee
Together with the first equation in \eqref{d331}, this implies that
\be\no
\tilde{\Upsilon}(0,x')=\p_{x_1}^2\tilde{\Upsilon}(0,x')=0,\ \ \ \forall x'\in \{x'\in\mathbb{R}^2: r_s<|x'|<r_2\}.
\ee

Extending $\tilde{\Upsilon}(x)$ and $\mathcal{G}_i(x), i=1,2,3$ from $\Omega_+$ to $\Omega_e=\{x\in \mathbb{R}^3: r_s<|x|<r_2\}$ as follows
\be\no
(\tilde{\Upsilon}_e,\mathcal{G}_{2e},\mathcal{G}_{3e})(x)=\begin{cases}
(\tilde{\Upsilon},\mathcal{G}_{2},\mathcal{G}_{3})(x_1,x'),\ \ &\forall x=(x_1,x')\in \overline{\Omega_+},\\
-(\tilde{\Upsilon},\mathcal{G}_{2},\mathcal{G}_{3})(-x_1,x'),\ \ &\forall x\in \Omega_e\setminus \overline{\Omega_+},
\end{cases}
\ee
and
\be\no
\mathcal{G}_{1e}(x)=\begin{cases}
\mathcal{G}_{1}(x_1,x'),\ \  &\forall x=(x_1,x')\in \overline{\Omega_+},\\
\mathcal{G}_{1}(-x_1,x'), \ \ &\forall x\in \Omega_e\setminus \overline{\Omega_+},
\end{cases}\ee
then $\tilde{\Upsilon}_e\in C^{2,\alpha}(\Omega_e),\mathcal{G}_{ie}\in C^{1,\alpha}(\overline{\Omega_e})$ and satisfies the following equation in $\Omega_e$:
\be\no\begin{cases}
\sum_{i=1}^3 \p_{x_i}^2 \tilde{\Upsilon}_e(x)= \sum_{i=1}^3 \p_{x_i}\mathcal{G}_{ie}(x),\ \ \forall x\in \Omega_e,\\
\sum_{i=1}^3 x_j \p_{x_j}\tilde{\Upsilon}_e(x)=0,\ \ \forall x\in \{x\in \mathbb{R}^3: |x|=r_s\},\\
\sum_{i=1}^3 x_j \p_{x_j}\tilde{\Upsilon}_e(x)=0,\ \ \forall x\in \{x\in \mathbb{R}^3: |x|=r_2\}.
\end{cases}\ee
Then by the standard theory in \cite[Chapter 6]{gt98}, $\tilde{\Upsilon}_e\in C^{2,\alpha}(\overline{\Omega_e})$ and $\Upsilon\in C^{2,\alpha}(\overline{\mathbb{D}})$ with the estimate
\be\no
&&\|\Upsilon\|_{C^{2,\alpha}(\overline{\mathbb{D}})}\leq C_*\|\tilde{\Upsilon}\|_{C^{2,\alpha}(\overline{\Omega_+})}\leq C_*\|\tilde{\Upsilon}_e\|_{C^{2,\alpha}(\overline{\Omega_e})}\\\no
&&\leq C_*\sum_{i=1}^3\|\mathcal{G}_{ie}\|_{C^{1,\alpha}(\overline{\Omega_e})}\leq C_*\sum_{i=1}^3\|\mathcal{G}_{i}\|_{C^{1,\alpha}(\overline{\Omega_+})}\leq C_*\sum_{i=1}^3\|G_{i}\|_{C^{1,\alpha}(\overline{\mathbb{D}})}\\\no
&&\leq C_*(\epsilon\|(\hat{{\bf W}}, \hat{W}_6)\|_{\Xi}+\|(\hat{{\bf W}}, \hat{W}_6)\|_{\Xi}^2)\leq C_*(\epsilon \delta_0 +\delta_0^2).
\ee
Furthermore, rewrite the first equation in \eqref{den33} in the polar coordinates $(y_1,a,\tau)$, one gets
\be\no
&&\p_{y_1}^2\Upsilon+ \frac{(1+a^2)^2}{4y_1^2}(\p_a^2 \Upsilon+ \frac1a\p_a \Upsilon+ \frac1{a^2}\p_{\tau}^2 \Upsilon)+ \frac2{y_1}\p_{y_1}\Upsilon\\\no
&&\quad\quad=\p_{y_1}G_1+\frac{1+a^2}{2y_1}(\p_a G_a+ \frac1a G_a+ \frac1a \p_{\tau} G_{\tau})+\frac{2G_1-a G_a}{y_1}.
\ee
Then employing the Dirichlet boundary and the compatibility condition \eqref{g123}, we find that
\be\label{d334}
\p_a^2 \Upsilon(y_1,1,\tau)+ \p_a\Upsilon(y_1,1,\tau)=0,\ \  \ \ \forall (y_1,\tau)\in [r_s,r_2]\times \mathbb{T}_{2\pi}.
\ee

\end{proof}

{\bf Step 4.2} Next we solve the divergence-curl system with normal boundary conditions
\be\label{den34}\begin{cases}
\p_{y_1} \dot{W}_1+\frac{1+|y'|^2}{2y_1}\sum_{j=2}^3\p_{y_j} \dot{W}_j +\frac{2\dot{W}_1}{y_1}-\frac{1}{y_1}\sum_{j=2}^3 y_j \dot{W}_j=0,\ \text{in }\mathbb{D},\\
\frac{1+|y'|^2}{2y_1}(\p_{y_2} \dot{W}_3-\p_{y_3} \dot{W}_2)+ \frac{1}{y_1}(y_3 \dot{W}_2-y_2 \dot{W}_3)=G_1-\p_{y_1}\Upsilon:=\dot{G}_1,\ \text{in }\mathbb{D},\\
\frac{1+|y'|^2}{2y_1}\p_{y_3} \dot{W}_1-\p_{y_1} \dot{W}_3-\frac{\dot{W}_3}{y_1}= G_2-\frac{1+|y'|^2}{2y_1}\p_{y_2}\Upsilon:=\dot{G}_2,\ \text{in }\mathbb{D},\\
\p_{y_1} \dot{W}_2+\frac{\dot{W}_2}{y_1}-\frac{1+|y'|^2}{2 y_1}\p_{y_2} \dot{W}_1= G_3-\frac{1+|y'|^2}{2y_1}\p_{y_3}\Upsilon:=\dot{G}_3,\ \text{in }\mathbb{D},\\
\dot{W}_1(r_s,y')= \dot{W}_1(r_2,y')=0,\ \forall y'\in E,\\
\sum_{j=2}^3y_j \dot{W}_j(y_1,y')=0,\ \forall (y_1,y')\in \Sigma_w.
\end{cases}\ee

It follows from \eqref{g123} and \eqref{d334} that on $\Sigma_w$ there holds
\be\label{g341}\begin{cases}
\dot{G}_1(\hat{{\bf W}},\hat{W}_6)(y_1,1,\tau)=0,\ \ \ \\
\dot{G}_{\tau}(W_5,\hat{{\bf W}},\hat{W}_6)(y_1,1,\tau)=(G_{\tau}-\frac{1}{y_1}\p_{\tau}\Upsilon)(y_1,1,\tau)=0,\ \ \\
\p_a \{\dot{G}_a(W_5,\hat{{\bf W}},\hat{W}_6)\}(y_1,1,\tau)=(\p_a G_a-\frac{1}{y_1}(\p_a^2 \Upsilon+\p_a \Upsilon))(y_1,1,\tau)=0,
\end{cases}\ee
where $\dot{G}_a=\cos\tau \dot{G}_2 +\sin\tau \dot{G}_3=G_a-\frac{1+a^2}{2y_1}\p_a \Upsilon$ and $\dot{G}_{\tau}=-\sin\tau \dot{G}_2 +\cos\tau \dot{G}_3$.

Since $\Upsilon$ satisfies the equation in \eqref{den33}, there holds
\be\label{g342}
\displaystyle \p_{y_1} \dot{G}_1+ \frac{1+|y'|^2}{2y_1}\sum_{j=2}^3\p_{y_j}\dot{G}_j+\frac{2}{y_1} \dot{G}_1-\frac{1}{y_1}\sum_{j=2}^3y_j\dot{G}_j\equiv 0,\ \text{in }\mathbb{D}.
\ee

\begin{lemma}\label{dc-1}
{\it Given any vector field $(\dot{G}_1,\dot{G}_2,\dot{G}_3)\in (C^{1,\alpha}(\overline{\mathbb{D}}))^3$ satisfying \eqref{g341} and \eqref{g342}, there exists a unique solution $(\dot{W}_1,\dot{W}_2,\dot{W}_3)\in (C^{2,\alpha}(\overline{\mathbb{D}}))^3$ to the problem \eqref{den34} with
\be\no
\sum_{i=1}^3\|\dot{W}_i\|_{C^{2,\alpha}(\overline{\mathbb{D}})}\leq C_*\sum_{i=1}^3\|\dot{G}_i\|_{C^{1,\alpha}(\overline{\mathbb{D}})}
\ee
and the compatibility conditions
\be\label{d350}\begin{cases}
(\dot{W}_a,\p_a \dot{W}_1,\p_a \dot{W}_{\tau})(y_1,1,\tau)=0,\ \ \ &\forall (y_1,\tau)\in [r_s,r_2]\times \mathbb{T}_{2\pi},\\
(\p_a^2\dot{W}_a+\p_a \dot{W}_a)(y_1,1,\tau)=0,\ \ \ &\forall (y_1,\tau)\in [r_s,r_2]\times \mathbb{T}_{2\pi},
\end{cases}\ee
where $\dot{W}_a=\cos\tau \dot{W}_2+ \sin \tau \dot{W}_3$ and $\dot{W}_{\tau}=-\sin\tau \dot{W}_2+ \cos\tau \dot{W}_3$.
}\end{lemma}

\begin{proof}
Define
\be\no
\bm{\dot{\mathcal{G}}}(x)=\sum_{i=1}^3 \dot{\mathcal{G}}_i(x) {\bf e}_i(x)=\sum_{i=1}^3 \dot{G}_i(W_5,\hat{{\bf W}},\hat{W}_6)\tilde{{\bf e}}_i(y),
\ee
then in the coordinates $x=(x_1,x')$, the compatibility condition \eqref{g341} becomes
\be\label{g343}\begin{cases}
\p_{x_1}\dot{\mathcal{G}}_1(0,x')=0,\ \ \ \forall x'\in \{x'\in\mathbb{R}^2: r_s<|x'|<r_2\},\\
\dot{\mathcal{G}}_2(0,x')=\dot{\mathcal{G}}_3(0,x')=0,\ \ \ \forall x'\in \{x'\in\mathbb{R}^2: r_s<|x'|<r_2\}.
\end{cases}\ee
Using \eqref{z3}, the formula \eqref{g342} implies that
\be\label{g344}
\sum_{i=1}^3 \p_{x_i} \dot{\mathcal{G}}_i(x)=0,\ \ \ \forall x\in \Omega_+.
\ee

Consider the following divergence-curl system in $\Omega_+$:
\be\label{d341}\begin{cases}
\sum_{i=1}^3 \p_{x_i}\dot{\mathcal{W}}_i(x)=0,\ \ \text{in } \Omega_+,\\
\p_{x_2} \dot{\mathcal{W}}_3(x)- \p_{x_3} \dot{\mathcal{W}}_2(x)= \dot{\mathcal{G}}_1(x),\ \ \text{in } \Omega_+,\\
\p_{x_3} \dot{\mathcal{W}}_1(x)- \p_{x_1} \dot{\mathcal{W}}_3(x)= \dot{\mathcal{G}}_2(x),\ \ \text{in } \Omega_+,\\
\p_{x_1} \dot{\mathcal{W}}_2(x)- \p_{x_2} \dot{\mathcal{W}}_1(x)= \dot{\mathcal{G}}_3(x),\ \ \text{in } \Omega_+,\\
\dot{\mathcal{W}}_1(0,x')=0,\ \ \forall r_s\leq |x'|=\sqrt{x_2^2+x_3^2}\leq r_2,\\
\sum_{i=1}^3 x_j \dot{\mathcal{W}}_i(x)=0,\ \ \forall x\in \{x\in \mathbb{R}^3: x_1>0, |x|=r_s \ \text{or } |x|=r_2\}.
\end{cases}\ee
Thanks to \eqref{g344}, the standard result (cf. \cite{ky09} and the references therein) implies that there exists a unique solution $\bm{\dot{\mathcal{W}}}(x)=\sum_{i=1}^3 \dot{\mathcal{W}}_i(x) {\bf e}_i(x)\in C^{2,\alpha}(\Omega_+)$ to \eqref{d341}, and away from the circles $\{(0,x'): |x'|=r_s\}\cup \{(0,x'): |x'|=r_2\}$, the regularity can be improved to be $C^{2,\alpha}$ up to the boundary. It follows from \eqref{g343} and \eqref{d341} that
\be\label{d342}
(\p_{x_1} \dot{\mathcal{W}}_2,\p_{x_1} \dot{\mathcal{W}}_3,\dot{\mathcal{W}}_1,\p_{x_1}^2\dot{\mathcal{W}}_1)(0,x')=0,\ \forall x'=\{x'\in\mathbb{R}^2: r_s<|x'|<r_2\}.
\ee
Define the vector field $(\dot{W}_1,\dot{W}_2,\dot{W}_3)(y)\in C^{2,\alpha}(\mathbb{D})\cap C^0(\overline{\mathbb{D}})$ through the transformation \eqref{ss}
\be\no
\bm{\dot{\mathcal{W}}}(x)=\sum_{i=1}^3 \dot{\mathcal{W}}_i(x) {\bf e}_i(x)=\sum_{i=1}^3 \dot{W}_i(y) \tilde{{\bf e}}_i(y),
\ee
then by \eqref{z3} and \eqref{z4}, $\dot{W}_i(y), i=1,2,3$ solves the problem \eqref{den34}.

To improve the regularity near the corner circles, we extend $\dot{\mathcal{W}}_i(x)$ and $\mathcal{G}_i(x), i=1,2,3$ from $\Omega_+$ to $\Omega_e=\{x\in \mathbb{R}^3: r_s<|x|<r_2\}$ as follows
\be\no
(\dot{\mathcal{W}}_{1e},\dot{\mathcal{G}}_{2e},\dot{\mathcal{G}}_{3e})(x)=\begin{cases}
(\dot{\mathcal{W}}_1,\dot{\mathcal{G}}_{2},\dot{\mathcal{G}}_{3})(x_1,x'), \forall x=(x_1,x')\in \Omega_+,\\
-(\dot{\mathcal{W}}_1,\dot{\mathcal{G}}_{2},\dot{\mathcal{G}}_{3})(-x_1,x'), \forall x\in\{(x_1,x'): x_1<0, r_s<|x|<r_2\}
\end{cases}
\ee
and
\be\no
(\dot{\mathcal{W}}_{2e},\dot{\mathcal{W}}_{3e},\dot{\mathcal{G}}_{1e})(x)=\begin{cases}
(\dot{\mathcal{W}}_2,\dot{\mathcal{W}}_{3},\dot{\mathcal{G}}_{1})(x_1,x'), \forall x=(x_1,x')\in \Omega_+,\\
(\dot{\mathcal{W}}_2,\dot{\mathcal{W}}_{3},\dot{\mathcal{G}}_{1})(-x_1,x'), \forall x\in\{(x_1,x'): x_1<0, r_s<|x|<r_2\}
\end{cases}\ee

Then \eqref{d342} yields that $\dot{\mathcal{W}}_{ie}\in C^{2,\alpha}(\Omega_e)\cap C^0(\overline{\Omega_e}), i=1,2,3$ and they satisfy
\be\no\begin{cases}
\sum_{i=1}^3 \p_{x_i}\dot{\mathcal{W}}_{ie}(x)=0,\ \ &\forall x\in \Omega_e,\\
\p_{x_2} \dot{\mathcal{W}}_{3e}(x)- \p_{x_3} \dot{\mathcal{W}}_{2e}(x)= \dot{\mathcal{G}}_{1e}(x),\ \ &\forall x\in \Omega_e,\\
\p_{x_3} \dot{\mathcal{W}}_{1e}(x)- \p_{x_1} \dot{\mathcal{W}}_{3e}(x)= \dot{\mathcal{G}}_{2e}(x),\ \ &\forall x\in \Omega_e,\\
\p_{x_1} \dot{\mathcal{W}}_{2e}(x)- \p_{x_2} \dot{\mathcal{W}}_{1e}(x)= \dot{\mathcal{G}}_{3e}(x),\ \ &\forall x\in \Omega_e,\\
\sum_{i=1}^3 x_i \dot{\mathcal{W}}_{ie}(x)=0,\ \ &\forall x\in \{x\in \mathbb{R}^3: |x|=r_s\ \text{or } |x|=r_2 \}.
\end{cases}\ee
Therefore $\dot{\mathcal{W}}_{ie}\in C^{2,\alpha}(\overline{\Omega_e}), i=1,2,3$ with the estimate
\be\no
&&\sum_{i=1}^3\|\dot{W}_{i}\|_{C^{2,\alpha}(\overline{\mathbb{D}})}\leq C_*\sum_{i=1}^3\|\dot{\mathcal{W}}_i\|_{C^{2,\alpha}(\overline{\Omega_+})}\leq C_*\sum_{i=1}^3\|\dot{\mathcal{W}}_{ie}\|_{C^{2,\alpha}(\overline{\Omega_e})}\\\no
&&\leq C_*\sum_{i=1}^3\|\dot{\mathcal{G}}_{ie}\|_{C^{1,\alpha}(\overline{\Omega_e})}\leq C_*\sum_{i=1}^3\|\dot{\mathcal{G}}_{i}\|_{C^{1,\alpha}(\overline{\Omega_+})}\leq C_*\sum_{i=1}^3\|\dot{G}_{i}\|_{C^{1,\alpha}(\overline{\mathbb{D}})}\\\no&&\leq C_*(\sum_{j=1}^3\|G_j\|_{C^{1,\alpha}(\overline{\mathbb{D}})}+
\|\Upsilon\|_{C^{2,\alpha}(\overline{\mathbb{D}})})\leq C_*\sum_{j=1}^3\|G_j\|_{C^{1,\alpha}(\overline{\mathbb{D}})}\\\no
&&\leq C_*(\epsilon\|(\hat{{\bf W}}, \hat{W}_6)\|_{\Xi}+\|(\hat{{\bf W}}, \hat{W}_6)\|_{\Xi}^2)\leq C_*(\epsilon \delta_0 +\delta_0^2).
\ee
It follows from the second, the third and fourth equations in \eqref{den34} that
\be\no\begin{cases}
\frac{1+a^2}{2y_1}(\p_a \dot{W}_{\tau}+\frac1a \dot{W}_{\tau}-\frac{1}{a}\p_{\tau}\dot{W}_{a})(y_1,a,\tau)-\frac{a}{y_1} \dot{W}_{\tau}(y_1,a,\tau)=\dot{G}_{1}(y_1,a,\tau),\\
\frac{1+a^2}{2y_1}\p_a \dot{W}_1(y_1,a,\tau)-\p_{y_1} \dot{W}_{a}(y_1,a,\tau)-\frac{1}{y_1} \dot{W}_{a}(y_1,a,\tau)= -\dot{G}_{\tau}(y_1,a,\tau).
\end{cases}\ee
Since $\dot{W}_a(y_1,1,\tau)=0$, then $(\p_a\dot{W}_{\tau},\p_a \dot{W}_1)(y_1,1,\tau)=0$ for all $(y_1,\tau)\in \Sigma_w$. Rewrite the first equation in \eqref{den34} in the coordinates $(y_1,a,\tau)$:
\be\no
\p_{y_1}\dot{W}_1+\frac{1+a^2}{2y_1}(\p_a \dot{W}_a+ \frac1a \dot{W}_a+ \frac1a\p_{\tau} \dot{W}_{\tau})+\frac{2\dot{W}_1}{y_1}-\frac{a \dot{W}_a}{y_1}=0.
\ee
Differentiating the above equation and evaluating at $a=1$, one obtains that $(\p_a^2 \dot{W}_a+\p_a \dot{W}_a)(y_1,1,\tau)=0$ for all $(y_1,\tau)\in \Sigma_w$.

\end{proof}


Let $(W_1,W_2,W_3)$ be the solution to \eqref{den32}, and denote
\be\no
M_j(y)= W_j(y)- \dot{W}_j(y), j=1,2,3.
\ee
Then $M_j, j=1,2,3$ satisfy the following problem
\be\label{den36}\begin{cases}
d_1(y_1)\p_{y_1} M_1+\frac{1+|y'|^2}{2y_1}\sum_{j=2}^3 \p_{y_j} M_j +\frac{2M_1}{y_1}-\frac{1}{y_1}\sum_{j=2}^3 y_j M_j\\
\q\q+ d_2(y_1) M_1 =G_4(\hat{{\bf W}},\hat{W}_6),\ \text{in }\mathbb{D},\\
\frac{1+|y'|^2}{2y_1}(\p_{y_2} M_3-\p_{y_3} M_2)+ \frac{1}{y_1}(y_3 M_2-y_2 M_3)=0,\  \text{in }\mathbb{D},\\
\frac{1+|y'|^2}{2y_1}\p_{y_3} (M_1+d_3(y_1) M_1(r_s,y'))-\p_{y_1} M_3-\frac{M_3}{y_1}= 0,\ \text{in }\mathbb{D},\\
\p_{y_1} M_2-\frac{1+|y'|^2}{2y_1}\p_{y_2} (M_1+d_3(y_1) M_1(r_s,y'))+\frac{M_2}{y_1} = 0,\ \text{in }\mathbb{D},\\
\sum_{i=2}^3\p_{y_i}^2M_1(r_s,y')-b_0 b_1\p_{y_i}\b(\frac{2r_s M_i}{1+|y'|^2}\b)(r_s,y')=q_5(\hat{{\bf W}}(r_s,y'),\hat{W}_6(y')),\ \forall y'\in E,\\
M_1(r_2,y')+ d_3(r_2) M_1(r_s,y')=q_4(y'),\ \forall y'\in E,\\
y_2 M_2(y_1,y')+y_3 M_3(y_1,y')=0,\ \forall (y_1,y')\in \Sigma_w,\\
\sum_{j=2}^3 y_j\p_{y_j} M_1(r_s,y')-\frac{2b_0b_1 r_s}{1+|y'|^2}\sum_{j=2}^3 y_j M_j(r_s,y')=0, \text{on }\ \ y_1=r_s, |y'|=1.
\end{cases}\ee
where
\be\no
&&G_4(\hat{{\bf W}},\hat{W}_6)=G_0(\hat{{\bf W}},\hat{W}_6)+\bar{M}^2(y_1)\p_{y_1}\dot{W}_1-d_2(y_1)\dot{W}_1,\\\no
&&q_5(\hat{{\bf W}}(r_s,y'),\hat{W}_6(y'))=q_1(\hat{{\bf W}}(r_s,y'),\hat{W}_6(y'))+2r_s b_0b_1\sum_{j=2}^3\p_{y_j}\b(\frac{\dot{W}_j}{1+|y'|^2}\b)(r_s,y').
\ee

The second, third and fourth equations in \eqref{den36} can be rewritten as
\be\no
&&\p_{y_2}\b(\frac{2 y_1 M_3}{1+|y'|^2}\b)-\p_{y_3}\b(\frac{2 y_1 M_2}{1+|y'|^2}\b)=0,\\\no
&&\p_{y_3}(M_1(y_1,y') + d_3(y_1) M_1(r_s,y'))-\p_{y_1}\b(\frac{2 y_1 M_3}{1+|y'|^2}\b)=0,\\\no
&&\p_{y_1}\b(\frac{2 y_1 M_2}{1+|y'|^2}\b)-\p_{y_2}(M_1(y_1,y') + d_3(y_1) M_1(r_s,y'))=0,
\ee
which implies that the existence of a potential function $\phi$ such that
\be\no
&&M_1(y_1,y') + d_3(y_1) M_1(r_s,y')=\p_{y_1}\phi(y_1,y'),\\\no
&& M_j(y_1,y')=\frac{1+|y'|^2}{2y_1}\p_{y_j}\phi(y_1,y'),\ \ j=2,3.
\ee
Therefore
\be\no
M_1(r_s,y')=\frac 1{a_3}\p_{y_1}\phi(r_s,y'),\ M_1(y_1,y')=\p_{y_1}\phi(y_1,y')- \frac{1}{b_3} d_3(y_1)\p_{y_1}\phi(r_s,y')
\ee
with
\be\no
b_3=\frac{(\gamma-2) \bar{M}^2(r_s)+2}{\gamma \bar{M}^2(r_s)}>0.
\ee
Thus the problem \eqref{den36} is equivalent to
\be\label{den37}\begin{cases}
d_1(y_1) \p_{y_1}^2 \phi+\frac{1+|y'|^2}{4 y_1^2}\sum_{j=2}^3\p_{y_j}(1+|y'|^2)\p_{y_j}\phi)+\frac{2}{y_1}\p_{y_1}\phi+ d_2(y_1)\p_{y_1}\phi\\
\quad\quad-\frac{1+|y'|^2}{2y_1^2}\sum_{j=2}^3 y_j \p_{y_j}\phi-\frac{1}{b_3} d_4(y_1)\p_{y_1}\phi(r_s,y')=G_4(\hat{{\bf W}},\hat{W}_6),\ \text{in }\mathbb{D},\\
y_2 \p_{y_2}\phi(y_1,y') + y_3 \p_{y_3}\phi (y_1,y')=0,\ \forall (y_1,y')\in\Sigma_w,\\
\p_{y_1}\phi(r_2,y')=q_4(y'),\ \forall y'\in E,\\
(\p_{y_2}^2+\p_{y_3}^2)\left(\p_{y_1}\phi(r_s,y')-b_4 \phi(r_s,y')\right)=b_3q_5(\hat{{\bf W}}(r_s,y'),\hat{W}_6(y')),\ \forall y'\in E,\\
(y_2\p_{y_2}+y_3\p_{y_3})(\p_{y_1}\phi-b_4 \phi)(r_s,y')=0,\ \forall \ y_1=r_s, |y'|=1
\end{cases}\ee
where
\be\no
&&d_4(y_1)=d_1(y_1)d_3'(y_1)+(\frac{2}{y_1}+d_2(y_1)) d_3(y_1)\\\no
&&\quad=- \frac{b_2}{b_1}d_1(y_1)\frac{\bar{B}+\frac{1}{2}\bar{U}^2(y_1)}{\gamma\bar{K}}\frac{\bar{U}'(y_1)}{\bar{U}^2(y_1)}+(\frac{2}{y_1}+d_2(y_1)) d_3(y_1)\\\no
&&\quad=\frac{2b_2}{b_1\gamma\bar{K}y_1\bar{U}(y_1)}\bigg(\bar{B}+\frac{1}{2}\bar{U}^2(y_1)+(\bar{B}-\frac{1}{2}\bar{U}^2(y_1))\bigg(1+\frac{\bar{M}^2(2+(\gamma-1)\bar{M}^2)}{1-\bar{M}^2}\bigg)\bigg)\\\no
&&\quad=\frac{2b_2}{\gamma b_1\bar{K} y_1 \bar{U}(y_1)}\bigg(2\bar{B}+\frac{\bar{U}^2(y_1)(2+(\gamma-1)\bar{M}^2(y_1))}{(\gamma-1)(1-\bar{M}^2(y_1))}\bigg)>0,\\\no
&& b_4=b_0 b_1 b_3>0.
\ee

A key issue here is the solvability of the Poisson equation with homogeneous Neumann boundary condition (the last two boundary conditions in \eqref{den37}) for the derivation of the oblique boundary condition for the potential $\phi$ on the boundary $\{(r_s,y'): y'\in E\}$.
\begin{lemma}\label{oblique}({\bf The oblique boundary condition on the shock front.})
On the shock front $\{(r_s, y'): y'\in E\}$, there exists a unique $C^{2,\alpha}(\overline{E})$ function $m_s(y')$ such that
\be\no
\p_{y_1}\phi(r_s,y')-b_4 \phi(r_s,y')= m_s(y')
\ee
where $m_s(y')$ satisfies the Poisson equation with Neumann boundary conditions
\be\label{den38}\begin{cases}
(\p_{y_2}^2+\p_{y_3}^2)m_s(y')=b_3q_5(\hat{{\bf W}}(r_s,y'),\hat{W}_6(y')),\ \ &\text{in } E,\\
(y_2\p_{y_2}+y_3\p_{y_3})m_s(y')=0,\ &\forall |y'|=1,\\
\iint_{E} m_s(y') dy'=0
\end{cases}\ee
and the following estimate holds
\be\label{den39}
\|m_s\|_{C^{2,\alpha}(\overline{E})}\leq C_*\|q_5(\hat{{\bf W}}(r_s,y'),\hat{W}_6(y'))\|_{C^{\alpha}(\overline{E})}.
\ee
\end{lemma}

\begin{proof}
Due to the compatibility conditions \eqref{cp100}, \eqref{j211} and \eqref{rv41}, the following solvability condition for \eqref{den38} holds
\be\no
&&\iint_{E} q_5(\hat{{\bf W}}(r_s,y'),\hat{W}_6(y')) dy'=\iint_{E} q_1(\hat{{\bf W}}(r_s,y'),\hat{W}_6(y')) dy'\\\no
&&=\oint_{|y'|=1}\sum_{j=2}^3 y_j g_j(\hat{{\bf W}}(r_s,y'), \hat{W}_6(y'))+y_j \p_{y_j}\{R_1(\hat{{\bf W}}(r_s,y'), \hat{W}_6(y'))\} ds =0.
\ee
There exists a unique solution $m_s(y')\in C^{2,\alpha}(\overline{E})$ to \eqref{den38} satisfying \eqref{den39} with the estimate
\be\no
&&\|m_s\|_{C^{2,\alpha}(\overline{E})}\leq C_*\|q_5(\hat{{\bf W}}(r_s,y'),\hat{W}_6(y'))\|_{C^{\alpha}(\overline{E})}\\\no
&&\leq C_*(\|R_1(y')\|_{C^{2,\alpha}(\overline{E})}+\sum_{j=2}^3\|\dot{W}_j(r_s,\cdot)\|_{C^{1,\alpha}(\overline{E})}+\|g_j(\hat{{\bf W}}(r_s,y'), \hat{W}_6(y'))\|_{C^{1,\alpha}(\overline{E})})\\\no
&&\leq C_*(\epsilon\|(\hat{{\bf W}}, \hat{W}_6)\|_{\Xi}+\|(\hat{{\bf W}}, \hat{W}_6)\|_{\Xi}^2\leq C_*(\epsilon \delta_0 +\delta_0^2).
\ee

\end{proof}

Using the oblique boundary condition at $y_1=r_s$, we could replace the nonlocal term $\frac{d_4(y_1)}{b_3}\p_{y_1}\phi(r_s,y')$ in the first equation of \eqref{den37} by $b_0b_1d_4(y_1)\phi(r_s,y')$. Thus the problem \eqref{den37} is simplified to a second order elliptic equation with a nonlocal term involving only the trace of the potential function and oblique boundary conditions:
\be\label{den41}\begin{cases}
\displaystyle d_1(y_1) \p_{y_1}^2 \phi+\frac{1+|y'|^2}{4 y_1^2}\sum_{j=2}^3\p_{y_j}(1+|y'|^2)\p_{y_j}\phi)+\frac{2}{y_1}\p_{y_1}\phi+ d_2(y_1)\p_{y_1}\phi\\
\quad\quad-\frac{1+|y'|^2}{2y_1^2}\sum_{j=2}^3 y_j \p_{y_j}\phi-b_0 b_1 d_4(y_1)\phi(r_s,y')=G_5(\hat{{\bf W}},\hat{W}_6),\ \text{in }\mathbb{D},\\
\p_{y_1}\phi(r_s,y')-b_4 \phi(r_s,y')=m_s(y'),\ \forall y'\in E,\\
y_2 \p_{y_2}\phi(y_1,y') + y_3 \p_{y_3}\phi (y_1,y')=0,\ \forall (y_1,y')\in\Sigma_w,\\
\p_{y_1}\phi(r_2,y')=m_2(y'),\ \forall y'\in E,
\end{cases}\ee
where
\be\no
&&G_5(y)=G_4(y)+ \frac{d_4(y_1)}{b_3} m_s(y'),\\\no
&&m_2(y')=q_4(\hat{{\bf W}}(r_s,y'),\hat{W}_6(y')).
\ee
It follows from \eqref{pres5} and the boundary condition in \eqref{den38} that
\be\label{mcp}
\p_a m_j(1,\tau)=0,\ \ \ \forall \tau\in \mathbb{T}_{2\pi}, j=s,2.
\ee
Furthermore, one can verify that (See the Appendix \S\ref{appendix}):
\be\label{g5cp}
&&\p_a G_5(y_1,1,\tau)=0,\ \ \ \forall (y_1,\tau)\in [r_s,r_2]\times \mathbb{T}_{2\pi}.
\ee

To show the existence and uniqueness of smooth solutions to \eqref{den41}, we first investigate the following eigenvalue problem.
\begin{lemma}\label{eig}{\it There exists a monotone increasing nonnegative real number sequence $\{\lambda_m\}_{m=1}^{\infty}$ of eigenvalues and the associated eigenvectors $\{Y_m(y')\}_{j=1}^{\infty}$ to the following eigenvalue problem
\be\label{eigen}\begin{cases}
\displaystyle (1+|y'|^2)\sum_{j=2}^3\p_{y_j}((1+|y'|^2)\p_{y_j} Y)-2(1+|y'|^2)\sum_{j=2}^3 y_j \p_{y_j} Y + \lambda Y=0,\ \ in \ E,\\
(y_2\p_{y_2} + y_3 \p_{y_3}) Y=0,\ \ \ \text{on}\ \ \p E.
\end{cases}\ee
Furthermore, the eigenvectors $\{Y_m(y')\}_{j=1}^{\infty}$ consists of an orthonormal basis of $L^2(E)$ with the inner product
\be\label{l2}
(f_1(y'), f_2(y'))_0=\iint_E\frac{f_1(y') f_2(y')}{(1+|y'|^2)^2} dy'.
\ee
}\end{lemma}

\begin{proof}
We use the spectral theorem \cite[Theorem 6.3.4]{bb92} to prove the Lemma. Let $H^1(E)$ be the Hilbert space with the inner product
\be\no
(f_1, f_2)_1= \iint_{E} \b(\sum_{j=2}^3 \p_{y_j} f_1 \p_{y_j} f_2 +\frac{f_1(y') f_2(y')}{(1+|y'|^2)^2} \b)dy'
\ee
Then $H^1(E)$ is dense in $L^2(E)$ and the identical embedding $H^1(E)\hookrightarrow L^2(E)$ is continuous and compact. Consider the symmetric bilinear form $Q$ in $H^1(E)$ as follows
\be\no
Q(f_1,f_2)=\iint_{E} \b(\sum_{j=2}^3 \p_{y_j} f_1 \p_{y_j} f_2 +\frac{f_1(y') f_2(y')}{(1+|y'|^2)^2} \b)dy',
\ee
then $Q$ is continuous and coercive. Then by the spectral theorem \cite[Theorem 6.3.4]{bb92}, there exists a monotone increasing sequence $\{\mu_m\}_{m=1}^{\infty}$ of eigenvalues
\be\no
0<\mu_1\leq \mu_2\leq \cdots\leq \mu_m\to +\infty
\ee
and an orthonormal basis $\{Y_m(y')\}_{j=1}^{\infty}\subset H^1(E)$ of $L^2(E)$ with the inner product in \eqref{l2} such that
\be\no
Q(Y_m, Y)= (\mu_m Y_m, Y)_0,\ \ \forall Y\in H^1(E),\ \ \forall m\in \mathbb{N}.
\ee
A standard argument shows that $Y_m$ satisfies the problem \eqref{eigen} with $\lambda_m=\mu_m-1$ for all $m\in \mathbb{N}$. Multiplying the equation in \eqref{eigen} by $\frac{Y(y')}{(1+|y'|^2)^2}$ and integrating over $E$, after an integration by parts, one gets
\be\no
&&\lambda \iint_{E} \frac{Y^2(y')}{(1+|y'|^2)^2} dy'=\sum_{j=2}^3\iint_{E}\p_{y_j}\big(\frac{Y}{1+|y'|^2}\big)(1+|y'|^2)\p_{y_j} Y+\frac{2 y_j Y\p_{y_j} Y}{1+|y'|^2} dy'\\\no
&&= \iint_{E} \sum_{j=2}^3 |\p_{y_j} Y|^2dy'.
\ee
Thus $0=\lambda_1 <\lambda_2\leq\cdots\leq \lambda_m \to +\infty$ as $m\to\infty$.

\end{proof}

\begin{remark}
{\it The measure $\frac{1}{(1+|y'|^2)^2} dy'$ used in \eqref{l2} and \eqref{den411} is the correspondence of the canonical measure $\sin\theta d\theta d\varphi$ on the unit sphere under the stereographic projection.
}\end{remark}

Now we can use the Lax-Milgram theorem and the Fredholm alternative of the second order elliptic equation to prove the existence and uniqueness of weak solution to \eqref{den41}. A function $\phi\in H^1(\mathbb{D})$ is said to be a weak solution to \eqref{den41}, if for any $\psi\in H^1(\mathbb{D})$, the following equality holds
\be\label{den411}
\mathcal{B}(\phi,\psi)= \mathcal{L}(\psi),\ \ \forall \psi\in H^1(\mathbb{D}),
\ee
where
\be\no
&&\mathcal{B}(\phi,\psi)=\iiint_{\mathbb{D}} \frac{d_1(y_1)\p_{y_1}\phi \p_{y_1}\psi}{(1+|y'|^2)^2} + \frac{1}{4y_1^2}\sum_{j=2}^3\p_{y_j}\phi \p_{y_j}\psi - \frac{d_5(y_1) \p_{y_1}\phi \psi}{(1+|y'|^2)^2}\\\no
&&\quad\quad +\frac{b_0b_1 d_4(y_1)\phi(r_s,y')\psi(y_1,y')}{(1+|y'|^2)^2}dy_1 dy' + \iint_{E}\frac{d_1(r_s)b_4 \phi(r_s,y')\psi(r_s,y')}{(1+|y'|^2)^2} dy',\\\no
&&\mathcal{L}(\psi)= -\iiint_{\mathbb{D}} \frac{\psi G_5}{(1+|y'|^2)^2} dy_1 dy'\\\no
&&\q\q\q\q\q\q + \iint_{E} \frac{d_1(r_2) m_2(y')\psi(r_2,y') -d_1(r_s) m_s(y')\psi(r_s,y')}{(1+|y'|^2)^2} dy',\\\no
&&d_5(y_1):=\frac{2}{y_1}+d_2(y_1)-d_1'(y_1).
\ee

\begin{lemma}\label{weak}
{\it There exists a positive constant $K$ depending only on the background solution such that the following problem has a unique weak solution in $H^1(\mathbb{D})$
\be\label{den42}\begin{cases}
\p_{y_1}(d_1(y_1) \p_{y_1} \phi)+\frac{1+|y'|^2}{4y_1^2}\sum_{j=2}^3\p_{y_j}((1+|y'|^2)\p_{y_j}\phi)-\frac{1+|y'|^2}{2y_1^2} \sum_{j=2}^3 y_j\p_{y_j}\phi\\
\quad\quad+ d_5(y_1)\p_{y_1}\phi-b_0b_1 d_4(y_1)\phi(r_s,y')-K\phi=G_5(y),\ \text{in }\mathbb{D},\\
\p_{y_1}\phi(r_s,y')-b_4 \phi(r_s,y')=m_s(y'),\ \ \forall y'\in E,\\
y_2 \p_{y_2}\phi(y_1,y') + y_3 \p_{y_3}\phi (y_1,y')=0,\ \ \forall (y_1,y')\in\Sigma_w,\\
\p_{y_1}\phi(r_2,y')=m_2(y'),\ \ \forall y'\in E.
\end{cases}\ee
}\end{lemma}

\begin{proof}
The weak formulation to \eqref{den42} is the existence of a $H^1(\mathbb{D})$ function $\phi$ such that
\be\label{den421}
\mathcal{B}_K(\phi,\psi):=\mathcal{B}(\phi,\psi) + K\iiint_{\mathbb{D}} \frac{\phi\psi}{(1+|y'|^2)^2} dy= \mathcal{L}(\psi),\ \ \forall \psi\in H^1(\mathbb{D}).
\ee

For any $\epsilon>0$, there holds
\be\no
\iint_{E}\frac{\phi^2(r_s,y')}{(1+|y'|^2)^2}dy'\leq  \iiint_{\mathbb{D}} \frac{C_0}{\epsilon}\frac{\phi^2(y_1,y')}{(1+|y'|^2)^2}+ \frac{\epsilon(\p_{y_1}\phi)^2(y_1,y')}{(1+|y'|^2)^2} dy' dy_1.
\ee
One can verify the boundedness and coercivity of $\mathcal{B}_K$ as follows
\be\no
&&|\mathcal{B}_K(\phi,\psi)|\leq C_0\|\phi\|_{H^1(\mathbb{D})}\|\psi\|_{H^1(\mathbb{D})},\\\no
&&|\mathcal{L}(\psi)|\leq C_0(\|G_5\|_{L^2(\mathbb{D})}+\sum_{j=s,2}\|m_j\|_{L^2(E)})\|\psi\|_{H^1(\mathbb{D})}
\ee
and
\be\no
&&\mathcal{B}_K(\phi,\phi)=\iiint_{\mathbb{D}} \frac{d_1(y_1)(\p_{y_1}\phi)^2}{(1+|y'|^2)^2} +\frac{1}{4y_1^2}\sum_{j=2}^3(\p_{y_j}\phi)^2-\frac{d_5(y_1) \p_{y_1}\phi \phi}{(1+|y'|^2)^2} \\\no
&&+ \frac{b_0b_1 d_4\phi(r_s,y')\phi(y_1,y')+K (\phi(y_1,y'))^2}{(1+|y'|^2)^2}dy_1 dy' + \iint_{E} \frac{d_1(r_s)b_4 (\phi(r_s,y'))^2}{(1+|y'|^2)^2} dy',\\\no
&&\geq C_*(\|\nabla\phi\|_{L^2(\mathbb{D})}^2+\|\phi(r_s,\cdot)\|_{L^2(E)}^2)+K\|\phi\|_{L^2(\mathbb{D})}^2-\frac{C_*}{4}\|\p_{y_1}\phi\|_{L^2(\mathbb{D})}^2\\\no
&&\quad-\frac{C_*}{4}\|\phi(r_s,\cdot)\|_{L^2(E)}^2-\tilde{C}_*\|\phi\|_{L^2(\mathbb{D})}^2\\\no
&&\geq \frac{C_*}{2}(\|\nabla\phi\|_{L^2(\mathbb{D})}^2+\|\phi(r_s,\cdot)\|_{L^2(E)}^2)+\frac{K}{2}\|\phi\|_{L^2(\mathbb{D})}^2,
\ee
provided that $K$ is large enough. Then by the Lax-Milgram theorem, there exists a unique $H^1(\mathbb{D})$ solution $\phi$ satisfying \eqref{den421}, which completes the proof of Lemma \ref{weak}.
\end{proof}

Now we are ready to solve the problem \eqref{den41}.

\begin{proposition}\label{solvability}
{\it Suppose that $(m_s,m_2)\in (C^{2,\alpha}(\overline{E}))^2$ and $G_5\in C^{1,\alpha}(\overline{\mathbb{D}})$ satisfy the compatibility conditions \eqref{mcp} and \eqref{g5cp}. Then there exists a unique solution $\phi\in C^{3,\alpha}(\overline{\mathbb{D}})$ to the problem \eqref{den41} with the estimate
\be\label{den414}
\|\phi\|_{C^{3,\alpha}(\overline{\mathbb{D}})}\leq C_*(\|G_5\|_{C^{1,\alpha}(\overline{\mathbb{D}})}+\sum_{j=s,2}\|m_j\|_{C^{2,\alpha}(\overline{E})}),
\ee
where the constant $C_*$ depends only on the coefficients $d_1,d_4,d_5, b_3,b_4$ and thus depends only the background solution.
}\end{proposition}

\begin{proof}

We first improve the regularity of any $H^1(\mathbb{D})$ weak solutions to \eqref{den41}. The goal is to show that for any weak solution $\phi\in H^1(\mathbb{D})$ to \eqref{den41}, $\phi$ indeed has a better regularity $\phi\in C^{3,\alpha}(\overline{\mathbb{D}})$ with the following estimate:
\be\label{den412}
\|\phi\|_{C^{3,\alpha}(\overline{\mathbb{D}})}\leq C_*(\|\phi\|_{H^1(\mathbb{D})}+\|G_5\|_{C^{1,\alpha}(\overline{\mathbb{D}})}+\sum_{j=1}^2\|m_j\|_{C^{2,\alpha}(\overline{E})}).
\ee

To this end, one can rewrite the first equation in \eqref{den41} as
\be\no
\p_{y_1}(d_1(y_1) \p_{y_1} \phi)+\frac{1+|y'|^2}{4y_1^2}\sum_{j=2}^3\p_{y_j}((1+|y'|^2)\p_{y_j}\phi)-\frac{1+|y'|^2}{2y_1^2} \sum_{j=2}^3 y_j\p_{y_j}\phi\\\no
\quad\quad+ d_5(y_1)\p_{y_1}\phi=G_6(y):=G_5+b_0b_1 d_4(y_1)\phi(r_s,y').
\ee
Since $\phi\in H^1(\mathbb{D}_+)$, $\phi(r_s,y')\in L^2(E)$ by the trace theorem. Together with the boundary conditions in \eqref{den41}, one can apply \cite[Theorems 5.36 and 5.45]{lie13} to obtain global $L^{\infty}$ bound and $C^{\alpha_1}$ estimates (for some $\alpha_1\in (0,1)$) for $\phi$ as follows,
\be\no
\|\phi\|_{C^{0,\alpha_1}(\overline{\mathbb{D}})}&\leq& C_*\bigg(\|a_0a_1d_4(y_1)\phi(r_s,y')\|_{L^2(E)}+\|G_5\|_{L^4(\mathbb{D})}+\sum_{j=1}^2\|m_j\|_{C^{1,\alpha}(\overline{E})}\bigg)\\\no
&\leq&C_*(\|\phi\|_{H^1(\mathbb{D})}+\|G_5\|_{C^{1,\alpha}(\overline{\mathbb{D}})}+\sum_{j=1}^2\|m_j\|_{C^{1,\alpha}(\overline{E})}).
\ee
Hence the term $b_0b_1 d_4(y_1) \phi(r_s,y')\in C^{\alpha_1}(\overline{\mathbb{D}})$ and the Schauder type estimate (cf. \cite[Theorem 4.6]{lie13}) would imply that
\be\no
\|\phi\|_{C^{1,\alpha}(\overline{\mathbb{D}})}\leq C_*(\|\phi\|_{H^1(\mathbb{D})}+\|G_5\|_{C^{1,\alpha}(\overline{\mathbb{D}})}+\sum_{j=1}^2\|m_j\|_{C^{\alpha}(\overline{E})}).
\ee
To further improve the regularity of $\phi$, we transform the function $\phi$ on $\mathbb{D}$ to a function $\Phi$ on $\Omega_+$ by the spherical projection coordinates defined in \eqref{ss}:
\be\no
\Phi(x)=\phi(\mathscr{S}(x)), \ \ \mathcal{G}_5(x)= G_5(\hat{{\bf W}},\hat{W}_6)(\mathscr{S}(x)), \ \ \forall x\in \Omega_+
\ee
then $\Phi(x)$ satisfies a uniformly elliptic equation with oblique boundary conditions:
\be\no\begin{cases}
\sum\limits_{i=1}^3 \p_{x_i}^2 \Phi-\frac{\bar{M}^2(|x|)}{|x|^2} \sum\limits_{i,j=1}^3 x_i x_j\p_{x_i x_j}^2 \Phi+\frac{d_2(|x|)}{|x|}\sum\limits_{j=1}^3 x_j\p_{x_j}\Phi\\
\quad\quad =b_0 b_1 d_4(|x|)f(x_1,x')+\mathcal{G}_5(x),\ \ \text{in }\Omega_+,\\
\bigg(\frac{1}{r_s}\sum\limits_{j=1}^3 x_j\p_{x_j}\Phi-b_4 \Phi\bigg)(\sqrt{r_s^2-|x'|^2},x')=\tilde{m}_s(\sqrt{r_s^2-|x'|^2},x'),\ \forall |x'|<r_s,\\
\p_{x_1}\Phi(0,x')=0,\ \forall r_s<|x'|<r_2,\\
\frac{1}{r_2}\sum\limits_{j=1}^3 x_j\p_{x_j}\Phi(\sqrt{r_2^2-|x'|^2},x')=\tilde{m}_2(\sqrt{r_2^2-|x'|^2},x'),\ \forall |x'|<r_2,
\end{cases}\ee
where
\be\no
f(x_1,x')=\phi\b(r_s,\frac{x_2}{|x|+x_1},\frac{x_3}{|x|+x_1}\b),\ \ \ \forall (x_1,x')\in \Omega_+
\ee
and
\be\no\begin{cases}
\tilde{m}_1(\sqrt{r_s^2-|x'|^2},x')=m_1\b(\frac{x_2}{r_s+\sqrt{r_s^2-|x'|^2}},\frac{x_3}{r_s+\sqrt{r_s^2-|x'|^2}}\b),\ \ &\forall |x'|<r_s,\\
\tilde{m}_2(\sqrt{r_2^2-|x'|^2},x')=m_2\b(\frac{x_2}{r_2+\sqrt{r_2^2-|x'|^2}},\frac{x_3}{r_2+\sqrt{r_2^2-|x'|^2}}\b),\ \ &\forall |x'|<r_2.
\end{cases}\ee
The Neumann boundary condition for $\phi$ on $\Sigma_w$ in \eqref{den41} and the compatibility condition \eqref{g5cp} imply that
\be\label{g5cp2}\begin{cases}
\p_{x_1}f(0,x')=-\frac{1}{|x'|}\sum\limits_{j=2}^3\frac{x_j}{|x'|}\p_{y_j}\phi(r_s,\frac{x'}{|x'|})=0\ \ \ &\forall r_s\leq |x'|\leq r_2,\\
\p_{x_1}\mathcal{G}_5(0,x')=0,\ \ &\forall r_s\leq |x'|\leq r_2.
\end{cases}\ee

We extend $\Phi, \mathcal{G}_5$ from $\Omega_+$ to $\Omega_e$ as follows:
\be\no
(\Phi_e,f_e,\mathcal{G}_{5e})(x)=\begin{cases}
(\Phi,f,\mathcal{G}_{5})(x_1,x'), &\forall x=(x_1,x')\in \Omega_+,\\
(\Phi,f,\mathcal{G}_{5})(-x_1,x'), &\forall x\in\{(x_1,x'): x_1<0, r_s<|x|<r_2\}
\end{cases}
\ee
and extend $\tilde{m}_j ,j=s,2$ from the half sphere $S_{r_j}^+=\{(x_1,x'): x_1>0, x_1^2+|x'|^2=r_j^2\}, j=s,2$ to $S_{r_j}=\{(x_1,x'):x_1^2+|x'|^2=r_j^2\}, j=s,2$:
\be\no
\tilde{m}_{je}(x_1,x')=\begin{cases}
\tilde{m}_{j}(\sqrt{r_j^2-|x'|^2},x'),\ \  &\forall x=(x_1,x')\in S_{r_j}^+,\\
\tilde{m}_{j}(-\sqrt{r_j^2-|x'|^2},x'),\ \ &\forall x=(x_1,x')\in S_{r_j}\setminus S_{r_j}^+.
\end{cases}\ee
Thanks to \eqref{g5cp2}, $\Phi_e, \mathcal{G}_{5e}\in C^{2,\alpha}(\Omega_e)\cap C^{1,\alpha}(\overline{\Omega_e})$ and $f_e\in C^{1,\alpha}(\overline{\Omega_e})$. Then $\Phi_e$ would satisfy
\be\no\begin{cases}
\sum_{i=1}^3 \p_{x_i}^2 \Phi_e-\frac{\bar{M}^2(|x|)}{|x|^2} \sum_{i,j=1}^3 x_i x_j\p_{x_i x_j}^2 \Phi_e+\frac{d_2(|x|)}{|x|}\sum_{j=1}^3 x_j\p_{x_j}\Phi_e\\
\quad\quad\quad=b_0 b_1 d_4(|x|)f_e(x_1,x')+\mathcal{G}_{5e}(x_1,x'),\ \text{in } \Omega_e,\\
\frac{1}{r_s}\sum_{j=1}^3 x_j\p_{x_j}\Phi(x_1,x')-b_4 \Phi(x_1,x')=\tilde{m}_{se}(x_1,x'),\ \text{on } x_1^2+|x'|^2=r_s^2,\\
\frac{1}{r_2}\sum_{j=1}^3 x_j\p_{x_j}\Phi(x_1,x')=\tilde{m}_{2e}(x_1,x'),\ \ \text{on } x_1^2+|x'|^2=r_2^2.
\end{cases}\ee

To show that $\tilde{m}_{je}\in C^{2,\alpha}(S_{r_j})$ for $j=s,2$, we use the stereographic projection $\mathscr{S}:(x_1,x')\in S_{r_j}\setminus\{(-r_j,0,0)\}\rightarrow (y_2,y_3)\in \mathbb{R}^2$
\be\no
(y_2,y_3):=\mathscr{S}(x_1,x')=\left(\frac{x_2}{r_j+x_1},\frac{x_3}{r_j+x_1}\right)
\ee
and pull $\tilde{m}_{je}$ back to $m_{je}$ on the $y'$ coordinates so that $\tilde{m}_{je}(x_1,x')=m_{je}(\mathscr{S}(x))$ for all $(x_1,x')\in S_{r_j}\setminus\{(-r_j,0,0)\}$. Simple calculations yield that
\be\no
m_{je}(y')=\begin{cases}
m_j(y'),\  &\forall |y'|\leq 1,\\
m_j(y'/|y'|^2),\ &\forall |y'|>1
\end{cases}
\ee
or in $(a,\tau)$ coordinates
\be\no
m_{je}(a,\tau)=\begin{cases}
m_j(a,\tau),\  &\forall a\in [0,1],\tau\in \mathbb{T}_{2\pi},\\
m_j(\frac1a,\tau),\ &\forall a\in (1,\infty),\tau\in \mathbb{T}_{2\pi}.
\end{cases}
\ee

Since $m_j\in C^{2,\alpha}(\overline{E}),j=s,2$ satisfy the compatibility conditions \eqref{g5cp}, then $m_{je}\in C^{2,\alpha}(\mathbb{R}^2)$ and thus $\tilde{m}_{je}\in C^{2,\alpha}(S_{r_j})$ with
\be\no
\|\tilde{m}_{je}\|_{C^{2,\alpha}(S_{r_j})}\leq C_*\|m_j\|_{C^{2,\alpha}(\overline{E})}, \ \ \ j=s,2.
\ee

Finally, a bootstrap argument and the standard Schauder estimate for the elliptic equation (See \cite[Chapter 6]{gt98}) implies that $\Phi_e\in C^{3,\alpha}(\overline{\Omega_e})$ and
\be\no
&&\|\phi\|_{C^{3,\alpha}(\overline{\mathbb{D}})}\leq \|\Phi\|_{C^{3,\alpha}(\overline{\Omega_+})}\leq \|\Phi_e\|_{C^{3,\alpha}(\overline{\Omega_e})}\\\no
&&\leq C_*(\|F\|_{C^{1,\alpha}(\overline{\Omega_e})}+\sum_{j=s,2}\|\tilde{m}_{je}\|_{C^{2,\alpha}(S_{r_j})})\\\no
&&\leq C_*(\|\Phi\|_{C^{1,\alpha}(\overline{S_{r_s}^+})}+\|G_5\|_{C^{1,\alpha}(\overline{\Omega_+})}+\sum_{j=s,2}\|\tilde{m}_{j}\|_{C^{2,\alpha}(S_{r_j})})\\\no
&&\leq C_*(\|G_5\|_{C^{1,\alpha}(\overline{\mathbb{D}})}+ \sum_{j=s,2}\|m_j\|_{C^{2,\alpha}(\overline{E})}).
\ee

Next we show the uniqueness of the $H^1(\mathbb{D})$ weak solution to \eqref{den41}, i.e. if $G_5\equiv 0, m_1=m_2=0$, and $\phi(y)\in H^1(\mathbb{D})$ is a weak solution to \eqref{den41}, then $\phi\equiv 0$ on $\mathbb{D}$.

Suppose $\phi(y)\in H^1(\mathbb{D})$, by \eqref{den412} one has $\phi\in C^{3,\alpha}(\overline{\mathbb{D}})$. Using Lemma \ref{eig}, $\phi(y_1,y')$ can be represented as
\be\no
\phi(y_1,y')= \displaystyle\sum_{n=1}^{\infty} X_n(y_1) Y_n(y'),
\ee
where $X_n(y_1)= \iint_{E} \phi(y_1,y') Y_n(y') dy'\in C^{2,\alpha}([r_s,r_2])$ solves the problem
\be\no\begin{cases}
d_1(y_1) X_n''(y_1) +(\frac{2}{y_1}+d_2(y_1)) X_n'(y_1)-\frac{\lambda_n}{4y_1^2}  X_n(y_1)-a_0 a_1 d_4(y_1)X_n(r_s)=0,\\
X_n'(r_s)-a_4 X_n(r_s)=0,\\
X_n'(r_2)=0.
\end{cases}\ee
Since $\lambda_n\geq 0$ for all $n\geq 1$, a simple argument using the maximum principle and Hopf's Lemma will imply that $X_n(y_1)\equiv 0$ for $\forall y_1\in [r_s,r_2]$ and thus $\phi(y)\equiv 0$ for all $y\in \mathbb{D}$.

Thus we have proved the uniqueness of $H^1$ weak solution to \eqref{den41}. This together with Lemma \ref{weak} and the Fredholm alternative for the elliptic equation, we obtain the existence and uniqueness of $H^1(\mathbb{D})$ weak solution to \eqref{den41} (See the argument in \cite[Theorem 8.6]{gt98}). With the aid of uniqueness, the estimate \eqref{den414} follows from \eqref{den412}. Hence the proof of the proposition is completed.

\end{proof}

Thus $M_1(y)=\p_{y_1}\phi(y)-\frac{1}{b_3} d_3(y_1)\p_{y_1}\phi(r_s,y'), M_j(y)=\frac{1+|y'|^2}{2y_1}\p_{y_j}\phi(y_1,y'), j=2,3$ would solve the problem \eqref{den36}. Furthermore, one can derive the following compatibility conditions as we have done in Lemma \ref{dc-1}:
\be\no\begin{cases}
M_a(y_1,1,\tau)=\p_a M_1(y_1,1,\tau)=\p_a M_{\tau}(y_1,1,\tau)=0,\ \ \ &\forall (y_1,\tau)\in [r_s,r_2]\times \mathbb{T}_{2\pi},\\
(\p_a^2 M_a+\p_a M_a)(y_1,1,\tau)=0,\ \ \ &\forall (y_1,\tau)\in [r_s,r_2]\times \mathbb{T}_{2\pi},
\end{cases}\ee
where $M_a=\cos\tau M_2+ \sin \tau M_3$ and $M_{\tau}=-\sin\tau M_2+ \cos\tau M_3$.

Then
\be\no
&&W_1(y_1,y')=\dot{W}_1(y)+M_1(y)=\dot{W}_1(y)+\p_{y_1}\phi(y)-\frac{1}{b_3} d_3(y_1)\p_{y_1}\phi(r_s,y'),\\\no
&&W_j(y_1,y')=\dot{W}_j(y_1,y')+M_j(y)=\dot{W}_j(y)+\frac{1+|y'|^2}{2y_1}\p_{y_j}\phi(y_1,y'),\ j=2,3,
\ee
will solve the problem \eqref{den32} and satisfy the estimate
\be\no
&&\sum_{j=1}^3\|W_j\|_{C^{2,\alpha}(\overline{\mathbb{D}})}\leq C_*(\sum_{j=1}^3\|\dot{W}_j\|_{C^{2,\alpha}(\overline{\mathbb{D}})}+\|\nabla \phi\|_{C^{2,\alpha}(\overline{\mathbb{D}})}+ \|\p_{y_1}\phi(r_s,y')\|_{C^{2,\alpha}(\overline{E})})\\\label{den46}
&&\leq C_*(\epsilon +C_*(\epsilon\|(\hat{{\bf W}}, \hat{W}_6)\|_{\Xi}+\|(\hat{{\bf W}}, \hat{W}_6)\|_{\Xi}^2)\leq C_*(\epsilon+\epsilon \delta_0 +\delta_0^2).
\ee
and the compatibility conditions
\be\no\begin{cases}
W_a(y_1,1,\tau)=\p_a W_1(y_1,1,\tau)=\p_a W_{\tau}(y_1,1,\tau)=0,\ \ \ &\forall (y_1,\tau)\in [r_s,r_2]\times \mathbb{T}_{2\pi},\\
(\p_a^2W_a+\p_a W_a)(y_1,1,\tau)=0,\ \ \ &\forall (y_1,\tau)\in [r_s,r_2]\times \mathbb{T}_{2\pi},
\end{cases}\ee
where $W_a=\cos\tau W_2+ \sin \tau W_3$ and $W_{\tau}=-\sin\tau W_2+ \cos\tau W_3$.

{\bf Step 5.} Once $W_1,W_2,W_3$ is obtained, the function $W_4$ is uniquely determined by \eqref{ent43}:
\be\no
W_4(y_1,y')=\frac{b_2}{b_1} W_1(r_s,y')+R_4(\hat{{\bf W}}(r_s,\beta_2(y),\beta_3(y)),\hat{W}_6(\beta_2(y),\beta_3(y))).
\ee
There also holds with the estimate
\be\no
&&\|W_4\|_{C^{2,\alpha}(\overline{\mathbb{D}})}\leq C_*(\|W_1(r_s,\cdot)\|_{C^{2,\alpha}(\overline{E})}+ C_*(\epsilon \|(\hat{{\bf W}}, \hat{W}_6)\|_{\Xi}+\|(\hat{{\bf W}}, \hat{W}_6)\|_{\Xi}^2)\\\label{ent54}
&&\leq C_*(\epsilon \delta_0+ \delta_0^2)
\ee
and the compatibility condition
\be\no
\p_a W_4(y_1,1,\tau)= \frac{b_2}{b_1}\p_a W_1(r_s,1,\tau)=0,\ \ \forall (y_1,\tau)\in\Sigma_w.
\ee

Finally, the shock front is uniquely determined by
\be\label{shock50}
W_6(y')=\frac{1}{b_1} W_1(r_s,y')-\frac{1}{b_1} R_1(\hat{{\bf W}}(r_s,y'),\hat{W}_6(y'))
\ee
which implies that $W_6\in C^{2,\alpha}(\overline{E})$. And the compatibility condition also holds
\be\no
\p_a W_6(1,\tau)=\frac{1}{b_1} \p_a W_1(r_s,1,\tau)-\frac{1}{b_1}\p_a\{R_1(\hat{{\bf W}}(r_s,\cdot),\hat{W}_6)\}(1,\tau)=0,\ \ \forall \tau\in \mathbb{T}_{2\pi}.
\ee
It remains to improve the regularity of $W_6$ to be $C^{3,\alpha}(\overline{E})$. To this end, define

\be\no
&&F_i(y'):= \p_{y_i} W_1(r_s,y')-\frac{2 b_0 r_s W_i(r_s,y')}{1+|y'|^2}-\frac{2b_1g_i(\hat{{\bf W}}(r_s,y'),\hat{W}_6(y'))}{1+|y'|^2}\\\no
&&\q\q\q\q-\p_{y_i} \{R_1(\hat{{\bf W}}(r_s,y'),\hat{W}_6(y'))\},\ \ i=2,3.
\ee
Then it follows from the first boundary condition in \eqref{den32}, the boundary data in \eqref{vor501} that
\be\no\begin{cases}
\p_{y_2} F_2 + \p_{y_3} F_3=0,\ \ &\text{in }E,\\
\p_{y_2} F_3-\p_{y_3} F_2=0,\ \ &\text{in }E,\\
y_2 F_2 + y_3 F_3 =0,\ \ \ &\text{on }\partial E.
\end{cases}\ee
Thus by Lemma \ref{equi0}, $F_2=F_3\equiv 0$ in $E$. Using the equation \eqref{shock50}, there holds
\be\no\begin{cases}
\p_{y_2} W_6(y')= \frac{2b_0r_s W_2(r_s,y')}{1+|y'|^2} + \frac{2g_2(\hat{{\bf W}}(r_s,y'),\hat{W}_6(y'))}{1+|y'|^2}, \ \ &\text{in }E,\\
\p_{y_3} W_6(y')= \frac{2b_0r_s W_3(r_s,y')}{1+|y'|^2} + \frac{2g_3(\hat{{\bf W}}(r_s,y'),\hat{W}_6(y'))}{1+|y'|^2}, \ \ &\text{in }E.
\end{cases}\ee
Therefore $W_6\in C^{3,\alpha}(\overline{E})$ with the estimate
\be\no
&&\|W_6\|_{C^{3,\alpha}(\overline{E})}\leq C_*\|W_1(r_s,\cdot)\|_{C^{2,\alpha}(\overline{E})}+ C_*\|R_1(\hat{{\bf W}}(r_s,y'),\hat{W}_6(y'))\|_{C^{2,\alpha}(\overline{E})})\\\label{shock54}
&&\quad + C_*\sum_{j=2}^3(\|W_j(r_s,\cdot)\|_{C^{2,\alpha}(\overline{E})}+\|g_j(\hat{{\bf W}}(r_s,y'),\hat{W}_6(y'))\|_{C^{2,\alpha}(\overline{E})})\\\no
&&\leq C_*(\epsilon + \epsilon  \|(\hat{{\bf W}}, \hat{W}_6)\|_{\Xi}+\|(\hat{{\bf W}}, \hat{W}_6)\|_{\Xi}^2)\leq C_*(\epsilon +\epsilon \delta_0+\delta_0^2).
\ee

Combining the estimates \eqref{ber43}, \eqref{den46}, \eqref{ent54} and \eqref{shock54}, one concludes that
\be\no
\|({\bf W}, W_6)\|_{\Xi}= \sum_{j=1}^5 \|W_j\|_{C^{2,\alpha}(\overline{\mathbb{D}})}+\|W_6\|_{C^{3,\alpha}(\overline{E})}\leq C_*(\epsilon +\epsilon \delta_0+\delta_0^2)\leq C_*(\epsilon +\delta_0^2).
\ee
Choose $\delta_0=\sqrt{\epsilon}$ and let $\epsilon<\epsilon_0=\frac{1}{4 C_*^2}$. Then $\|({\bf W}, W_6)\|_{\Xi}\leq 2C_*\epsilon\leq \delta_0$, thus $({\bf W}, W_6)\in \Xi$. We now can define the operator $\mathcal{T}:(\hat{{\bf W}}, \hat{W}_6)\mapsto ({\bf W}, W_6)$ which maps $\Xi$ to itself.

{\bf Step 6.} It remains to show that $\mathcal{T}$ is a contraction in the weak norm
\be\no
\|({\bf W}, W_6)\|_w:=\sum_{j=1}^5 \|W_j\|_{C^{1,\alpha}(\overline{\mathbb{D}})}+\|W_6\|_{C^{2,\alpha}(\overline{E})}.
\ee
This can be done by taking the difference for two solutions, we omit the details. Since $\delta_0=\sqrt{\epsilon}$, if $\epsilon<\epsilon_0=\frac{1}{16C_*^2}$, then $\|({\bf Z}, Z_6)\|_w\leq \frac{1}{2}\|(\hat{{\bf Z}}, \hat{Z}_6)\|_w$ so that the mapping $\mathcal{T}$ is a contraction operator in the weak norm $\|\cdot\|_{w}$. Thus there exists a unique point $({\bf W},W_6)\in \Xi$ such that $\mathcal{T}({\bf W},W_6)=({\bf W},W_6)$. Let us recall the auxiliary function $\Upsilon$ that is associated with the point $({\bf W},W_6)$ in solving the problem \eqref{den32}. We still need to prove that $\Upsilon\equiv 0$ in $\mathbb{D}$. Thanks to the definitions of $G_j({\bf W},W_6)$ for $j=1,2,3$, one may infer from \eqref{den32} that
\be\label{pi0}\begin{cases}
-\p_{y_1}\Upsilon=\frac{1+|y'|^2}{2 D_0^{W_6}}(D_2^{W_6} W_3- D_3^{W_6} W_2)+\frac{1}{D_0^{W_6}}(y_3 W_2-y_2 W_3)-\tilde{\omega}_1,\\
-\frac{1+|y'|^2}{2y_1}\p_{y_2}\Upsilon=\frac{1+|y'|^2}{2D_0^{W_6}} D_3^{W_6} W_1- D_1^{W_6} W_3-\frac{W_3}{D_0^{W_6}}-\frac{W_2\tilde{\omega}_1}{\bar{U}(D_0^{W_6})+W_1}\\
\quad-\frac{(1+|y'|^2)D_3^{W_6} W_5}{2D_0^{W_6}(\bar{U}(D_0^{W_6})+W_1)}+\frac{\bar{B}+W_5-\frac{1}{2}(\bar{U}(D_0^{W_6})+W_1)^2-\frac{1}{2}(W_2^2+W_3^2)}{\gamma(\bar{K}+W_4)}\frac{(1+|y'|^2)D_3^{W_6} W_4}{2D_0^{W_6}(\bar{U}(D_0^{W_6})+W_1)},\\
-\frac{1+|y'|^2}{2y_1}\p_{y_3}\Upsilon=D_1^{W_6} W_2+\frac{W_2}{D_0^{W_6}}-\frac{1+|y'|^2}{2D_0^{W_6}} D_2^{W_6} W_1-\frac{W_3\tilde{\omega}_1}{\bar{U}(D_0^{W_6})+W_1}\\
\quad\quad+\frac{(1+|y'|^2)D_2^{W_6} W_5}{2D_0^{W_6}(\bar{U}(D_0^{W_6})+W_1)}-\frac{\bar{B}+W_5-\frac{1}{2}(\bar{U}(D_0^{W_6})+W_1)^2-\frac{1}{2}(W_2^2+W_3^2)}{\gamma(\bar{K}+W_4)}\frac{(1+|y'|^2)D_2^{W_6} W_4}{2D_0^{W_6}(\bar{U}(D_0^{W_6})+W_1)}.
\end{cases}\ee
Since the vorticity $\tilde{\omega}_1$ satisfies the equation \eqref{vor400} and the following commutator relations hold
\be\no
D_1^{W_6}D_2^{W_6}=D_2^{W_6}D_1^{W_6},\ \ D_2^{W_6}D_3^{W_6}=D_3^{W_6}D_2^{W_6}, \ \ D_1^{W_6}D_3^{W_6}=D_3^{W_6}D_1^{W_6},
\ee
one can conclude from \eqref{pi0} that in $\mathbb{D}$
\be\no
D_1^{W_6}(\p_{y_1}\Upsilon)+\frac{1+|y'|^2}{4y_1 D_0^{W_6}}\sum_{j=2}^3 D_j^{W_6}((1+|y'|^2)\p_{y_j}\Upsilon)+\frac{2\p_{y_1}\Upsilon}{D_0^{W_6}}-\frac{1+|y'|^2}{2 y_1 D_0^{W_6}}\sum_{j=2}^3 y_j\p_{y_j}\Upsilon=0.
\ee
Since $\|W_6\|_{C^{3,\alpha}(\overline{E})}\leq \delta_0$, where $\delta_0$ is sufficiently small, thus $\Upsilon$ satisfies a second order uniformly elliptic equation without zeroth order term. Thanks to the homogeneous mixed boundary conditions for $\Upsilon$ on $\partial\mathbb{D}$ in \eqref{den32}, it follows directly from the maximum principle that $\Upsilon\equiv0$ in $\mathbb{D}$. Thus $({\bf W},W_6)$ is the desired solution. The proof of Theorem \ref{existence} is completed.

\section{Appendix}\label{appendix}

\subsection{The compatibility conditions in Lemma \ref{supersonic}} We first prove the compatibility conditions hold for the purely supersonic flows and they are also preserved when the flows move across the shock front.

\begin{proof}[Proof of Lemma \ref{supersonic}.] We only prove the compatibility conditions \eqref{comp2}. The slip boundary condition \eqref{slip} becomes
\be\label{slip-p}
U_a(z_1,1,\tau)=0,\ \ \ \forall(z_1,\tau)\in [r_1,r_2]\times \mathbb{T}_{2\pi}.
\ee
By \eqref{euler-p}, this implies that
\be\no
\p_{a} P^-(z_1,1,\tau)= 0,\ \ \forall (z_1,\tau)\in [r_1,r_2]\times \mathbb{T}_{2\pi}.
\ee

Differentiating the second, fourth and fifth equations in \eqref{euler-p} with respect to $a$, restricting the resulting equations on the surface $a=1$, one obtains
\be\no\begin{cases}
(U_1^-\p_{z_1}+\frac{1}{z_1}U_{\tau}^-\p_{\tau})(\p_a U_1^-) + (\p_{z_1} U_1^-+\frac{1}{z_1}\p_{a} U_a^-) \p_{a} U_1^-\\
\q\q + \frac{1}{z_1}(\p_{\tau} U_1^-- 2 U_{\tau}^-)\p_{a} U_{\tau}^- -\frac{\p_{z_1} P^-}{\rho^2}\frac{\p \rho^-}{\p K} \p_{a} K^-=0,\\
(U_1^-\p_{z_1}+\frac{1}{z_1}U_{\tau}^-\p_{\tau})(\p_{a} U_{\tau}^-) + (\p_{z_1}U_{\tau}^-+\frac{1}{z_1}U_{\tau}^-)\p_{a} U_1^-\\
\quad\quad+ \frac{1}{z_1}(\p_{a} U_a^-+ \p_{\tau} U_{\tau}^-+U_{1}^-)\p_{a} U_{\tau}^- -\frac{1}{z_1\rho^2}\p_{\tau} P^- \frac{\p \rho^-}{\p K}\p_{a} K^-=0,\\
(U_1^-\p_{z_1}+\frac{1}{z_1}U_{\tau}^-\p_{\tau})(\p_{a} K^-) + \p_{z_1} K^- \p_{a} U_1^- + \frac{1}{z_1}\p_{\tau} K^- \p_{a} U_{\tau}^- + \frac{1}{z_1}\p_{a}U_a^- \p_{a} K^-=0.
\end{cases}\ee
This is a homogeneous system of transport equations. It follows from \eqref{comp1} that
\be\no
(\p_{a} U_1^-,\p_{a} U_{\tau}^-, \p_{a} K^-)(z_1,1,\tau)=0,\ \ \forall (z_1,\tau)\in [r_1,r_2]\times \mathbb{T}_{2\pi}.
\ee
Differentiating the first equation in \eqref{euler-p} with respect to $a$, restricting the resulting equations on $a=1$, one gets
\be\no
\p_{a}^2 U_a^-(z_1,1,\tau)+\p_{a} U_a^-(z_1,1,\tau)=0,\ \ \forall (z_1,\tau)\in [r_1,r_2]\times \mathbb{T}_{2\pi}.
\ee

\end{proof}

We further prove that the compatibility conditions \eqref{comp2} is preserved when the supersonic flow moves across the shock. This is proved under the assumption that the existence of a transonic shock solution with $C^{2,\alpha}(\overline{\mathbb{D}_+})$ regularity in subsonic region and the $C^{3,\alpha}(\overline{E})$ regularity of the shock front.

\begin{lemma}({\bf Compatibility conditions on the intersection of the shock front and the wall.})\label{c11}
{\it Suppose the supersonic incoming flow satisfies the compatibility conditions \eqref{comp1}. Assume further that the system \eqref{euler-s},\eqref{rh} and \eqref{slip}-\eqref{pres} has a piecewise smooth solution $(U_1^{\pm},U_2^{\pm},U_3^{\pm},P^{\pm},K^{\pm})$ defined on $\mathcal{N}_{\pm}$ respectively and the shock front $z_1=\xi(z')$ with the properties $(U_1^+,U_2^+,U_3^+,P^+,K^+)\in C^{2,\alpha}(\overline{\mathcal{N}_+})$ and $\xi\in C^{3,\alpha}(\overline{E})$. Then the following compatibility conditions hold on the intersection of the shock front and the cylinder wall
\be\label{c1}\begin{cases}
(\p_{a}U_{1}^+,\p_a U_{\tau}^+, \p_a P^+,\p_a K^+)(z_1,1, \tau)=0,\ \forall \xi(1,\tau)\leq z_1\leq r_2, \tau \in \mathbb{T}_{2\pi},\\
U_a^+(z_1,1,\tau)=(\p_a^2 U_a^++\p_a U_a^+)(z_1,1,\tau)=0,\ \forall \xi(1,\tau)\leq z_1\leq r_2, \tau \in \mathbb{T}_{2\pi}.
\end{cases}\ee

}\end{lemma}

\begin{proof}

Assume that the shock front is represented as $z_1=\xi(a,\tau)$, then it follows from \eqref{rh} that
\be\lab{rh-p}\begin{cases}
[\rho U_1] -\frac{1+a^2}{2\xi(a,\tau)}(\p_{a}\xi [\rho U_a] + \frac{1}{a}\p_{\tau}\xi [\rho U_{\tau}])=0,\\
[\rho U_1^2+ P]-\frac{1+a^2}{2\xi(a,\tau)}(\p_{a}\xi [\rho U_1 U_a] +\frac{1}{a}\p_{\tau}\xi[\rho U_1 U_3])=0,\\
[\rho U_1 U_a] -\frac{1+a^2}{2\xi(a,\tau)}(\p_{a}\xi [\rho U_a^2+P] +\frac{1}{a}\p_{\tau}\xi[\rho U_a U_{\tau}])=0,\\
[\rho U_1 U_{\tau}]-\frac{1+a^2}{2\xi(a,\tau)}(\p_{a}\xi [\rho U_a U_{\tau}] +\frac{1}{a}\p_{\tau}\xi[\rho U_{\tau}^2+P])=0,\\
[B]=0.
\end{cases}\ee

It suffices to show that \eqref{c1} holds on the intersection of the shock front and the nozzle wall, the rest can be done as for Lemma \ref{supersonic}.  It follows from the third equation in \eqref{rh-p} and \eqref{slip-p} that
\be\label{perpen}
\p_{a}\xi(1, \tau)=0,\ \ \forall \tau \in \mathbb{T}_{2\pi}.
\ee
Substitute \eqref{perpen} into the last equation in \eqref{rh} to get
\be\label{shock-cyl1}
\frac{1}{\xi(1,\tau)}\p_{\tau}\xi(1,\tau) =\frac{[\rho U_1 U_{\tau}]}{[P+\rho U_{\tau}^2]},\ \ \forall \tau\in \mathbb{T}_{2\pi}.
\ee

Differentiating the Rankine-Hugoniot jump conditions in \eqref{rh-p} with respect to $a$ and restricting the resulting equations on $a=1$, utilizing \eqref{slip-p}, \eqref{perpen} and \eqref{shock-cyl1}, then at points $(\xi(1,\tau),1,\tau)$ one obtains
\be\no\begin{cases}
\p_{a}(\rho^+ U_1^+)-\frac{[\rho U_1 U_{\tau}]}{[P+\rho U_{\tau}^2]} \p_{a}(\rho^+ U_{\tau}^+)=0,\\
\rho^+ U_1^+\p_{\theta} U_1^+-\frac{[\rho U_1 U_{\tau}]}{[P+\rho U_{\tau}^2]} \rho^+ U_{\tau}^+ \p_{a} U_1^+=0,\\
\rho^+ U_1^+\p_{\theta} U_{\tau}^+-\frac{[\rho U_1 U_{\tau}]}{[P+\rho U_{\tau}^2]}\rho^+ U_{\tau}^+\p_{\theta} U_{\tau}^+=0.
\end{cases}\ee

Then
\be\no
(\p_{a} U_1^+,\p_{a} U_{\tau}^+,\p_{a}\rho^+)(\xi(1,\tau),1, \tau)=0,\ \ \ \forall \tau\in \mathbb{T}_{2\pi}.
\ee
The proof of Lemma \eqref{c11} is completed.

\end{proof}

\subsection{The verification of the compatibility conditions in \S\ref{proof}}

We give the explicit expressions of $J_i({\bf W}(r_s,y'), W_6) (i=2,3),J({\bf W}(r_s,y'), W_6)$ and $R_{0i}, i=1,2,3$ needed in \eqref{g21}-\eqref{shock400} and \eqref{ent31}:
\be\no
&& J_2({\bf W}(r_s,y'), W_6)\\\no
&&=\big(\tilde{\rho}({\bf W}(r_s,y'), W_6)W_3^2(r_s,y')+\tilde{P}({\bf W}(r_s,y'), W_6)-(\rho^-(U_3^-)^2+P^-)(r_s+W_6,y')\big)\\\no
&&\times \big(\tilde{\rho}({\bf W}(r_s,y'), W_6)(\bar{U}(r_s+W_6)+W_1(r_s,y'))W_2(r_s,y')-(\rho^- U_1^- U_2^-)(r_s+W_6,y')\big)\\\no
&&-\big(\tilde{\rho}({\bf W}(r_s,y'), W_6)(\bar{U}(r_s+W_6)+W_1(r_s,y'))W_3(r_s,y')-(\rho^- U_1^- U_3^-)(r_s+W_6,y')\big)\\\no
&&\quad\times \big(\tilde{\rho}({\bf W}(r_s,y'), W_6)(W_2W_3)(r_s,y')-(\rho^-U_2^-U_3^-)(r_s+W_6,y')\big),\\\no
&& J_3({\bf W}(r_s,y'), W_6)\\\no
&&=\big(\tilde{\rho}({\bf W}(r_s,y'), W_6)W_2^2(r_s,y')+\tilde{P}({\bf W}(r_s,y'), W_6)-(\rho^-(U_2^-)^2+P^-)(r_s+W_6,y')\big)\\\no
&&\times\big(\tilde{\rho}({\bf W}(r_s,y'), W_6)(\bar{U}(r_s+W_6)+W_1(r_s,y'))W_3(r_s,y')-(\rho^- U_1^- U_3^-)(r_s+W_6,y')\big)\\\no
&&-\big(\tilde{\rho}({\bf W}(r_s,y'), W_6)(\bar{U}(r_s+W_6)+W_1(r_s,y'))W_2(r_s,y')-(\rho^- U_1^- U_2^-)(r_s+W_6,y')\big)\\\no
&&\times \big(\tilde{\rho}({\bf W}(r_s,y'), W_6)(W_2W_3)(r_s,y')-(\rho^-U_2^-U_3^-)(r_s+W_6,y')\big),
\ee
and
\be\no
&& J({\bf W}(r_s,y'), W_6)\\\no
&&=\big(\tilde{\rho}({\bf W}(r_s,y'), W_6)W_2^2(r_s,y')+\tilde{P}({\bf W}(r_s,y'), W_6)-(\rho^-(U_2^-)^2+P^-)(r_s+W_6,y')\big)\\\no
&&\times\big(\tilde{\rho}({\bf W}(r_s,y'), W_6)W_3^2(r_s,y')+\tilde{P}({\bf W}(r_s,y'), W_6)-(\rho^-(U_3^-)^2+P^-)(r_s+W_6,y')\big)\\\label{j10}
&&-\big(\tilde{\rho}({\bf W}(r_s,y'), W_6)(W_2W_3)(r_s,y')-(\rho^- U_2^- U_3^-)(r_s+W_6,y')\big)^2,
\ee
and
\be\no
&&R_{01}({\bf W}(r_s,\cdot), W_6)=-[\bar{\rho} \bar{U}](r_s+W_6)+(\rho^- U_1^-)(r_s+W_6,\cdot)\\\no
&&\quad-(\bar{\rho}^-\bar{U}^-)(r_s+W_6)-(\bar{\rho}^+(r_s+W_6)-\bar{\rho}^+(r_s))W_1(r_s,\cdot)\\\no
&&\q-(W_1(r_s,y')+\bar{U}^+(r_s+W_6)-\bar{U}^+(r_s))(\tilde{\rho}({\bf W}(r_s,\cdot),W_6)-\bar{\rho}^+(r_s+W_6))\\\no
&&\quad+ \sum_{i=2}^3(\rho({\bf W}(r_s,\cdot),W_6)W_{i}(r_s,\cdot)-(\rho^- U_i^-)(r_s+W_6,y'))\frac{J_i({\bf W}(r_s,\cdot),W_6)}{J({\bf W}(r_s,\cdot),W_6)},
\ee
\be\no
&&R_{02}({\bf W}(r_s,\cdot), W_6)\\\no
&&=-\big\{[\bar{\rho} \bar{U}^2+ \bar{P}](r_s +W_6)-\frac{2}{r_s}[\bar{P}(r_s)] W_6\big\}+ (\rho^- (U_1^-)^2 +P^-)(r_s+W_6,y')\\\no
&&-(\bar{\rho}^-(\bar{U}^-)^2+\bar{P}^-)(r_s+W_6)-\bigg\{\tilde{\rho}({\bf W}(r_s,y'),W_6)(\bar{U}(r_s+W_6)+W_1(r_s,y'))^2 \\\no
&&+\tilde{P}({\bf W}(r_s,\cdot),W_6)-(\bar{\rho}^+(\bar{U}^+)^2+\bar{P}^+)(r_s+W_6)-2(\bar{\rho}^+ \bar{U}^+)(r_s)W_1(r_s,\cdot)\\\no
&&- \{(\bar{U}^+(r_s))^2+c^2(\bar{\rho}^+(r_s),\bar{K}^+)\}(\tilde{\rho}({\bf W}(r_s,\cdot),W_6)-\bar{\rho}^+(r_s+W_6))-(\bar{\rho}^+(r_s))^{\gamma} W_4(r_s,\cdot)\bigg\}
\\\no
&&+\sum_{i=2}^3\big(\tilde{\rho}({\bf W}(r_s,\cdot),W_6)(\bar{U}(r_s+W_6)+W_1(r_s,\cdot))W_i(r_s,\cdot)\\\no
&&-(\rho^- U_1^- U_i^-)(r_s+W_6,\cdot)\big)\frac{J_i({\bf W}(r_s,\cdot), W_6)}{J({\bf W}(r_s,\cdot), W_6)},\\\no
&&R_{03}({\bf W}(r_s,\cdot), W_6)= B^-(r_s+ W_6,\cdot)- \bar{B}^-- \bar{U}^+(r_s+W_6) W_1(r_s,\cdot)-\frac{1}{2}\sum_{j=1}^3 W_j^2(r_s,\cdot)\\\no
&&\quad -\frac{\gamma}{\gamma-1}\big((\bar{K}^+ +W_4(r_s,\cdot))(\tilde{\rho}({\bf W}(r_s,\cdot), W_6))^{\gamma-1}-\bar{K}^+ (\bar{\rho}^+(r_s+W_6))^{\gamma-1}\big)\\\no
&&\quad + \bar{U}^+(r_s) W_1(r_s,\cdot)+ \frac{c^2(\bar{\rho}^+(r_s),\bar{K}^+)}{\bar{\rho}^+(r_s)} (\tilde{\rho}({\bf W}(r_s,\cdot), W_6)-\bar{\rho}^+(r_s+W_6))\\\no
&&\quad+ \frac{\gamma (\bar{\rho}^+(r_s))^{\gamma-1}}{(\gamma-1)} W_4(r_s,\cdot).
\ee

Now we start to verify the compatibility conditions needed in \S\ref{proof}. Before we do it, we make some preparations.  Since
\be\no
\p_{y_2}=\cos\tau\p_a -\frac{1}{a}\sin\tau \p_{\tau},\ \ \p_{y_3}=\sin\tau\p_a +\frac{1}{a}\cos\tau \p_{\tau},
\ee
then there holds
\be\label{ad1}
&&\frac{1+|y'|^2}{2}(\p_{y_2}\hat{W}_3-\p_{y_3}\hat{W}_2)+(y_3\hat{W}_2-y_2\hat{W}_3)\\\no
&&\q\q=\frac{1+a^2}{2}(\p_a \hat{W}_{\tau}-\frac1a\p_{\tau} \hat{W}_{a})+\frac{(1-a^2)\hat{W}_{\tau}}{2a},\\\label{ad2}
&&\p_{y_2}\hat{W}_6\p_{y_1}\hat{W}_3-\p_{y_3}\hat{W}_6\p_{y_1} \hat{W}_2=\p_a \hat{W}_6 \p_{y_1}\hat{W}_{\tau}-\frac{1}{a}\p_{\tau}\hat{W}_6 \p_{y_1}\hat{W}_a,\\\label{ad3}
&&\p_{y_2}\hat{W}_i\p_{y_3}\hat{W}_j-\p_{y_3}\hat{W}_i\p_{y_2}\hat{W}_j=\frac1a \p_a \hat{W}_i\p_{\tau}\hat{W}_{j} -\frac1a \p_{\tau} \hat{W}_i\p_{a} \hat{W}_j,\ \\\label{ad4}
&&\frac1a\sum_{i=2}^3y_i D_i^{\hat{W}_6} \hat{W}_j=\p_a\hat{W}_j-\frac{r_2-y_1}{r_2-r_s-\hat{W}_6}\p_{y_1}\hat{W}_j \p_a \hat{W}_6, j=1,4,5.
\ee

\framebox{{\bf The verification of \eqref{j211}, \eqref{rv11} and \eqref{rv41}.}} Some computations yield that
\be\no
&&J_a(r_s,a,\tau)=\cos\tau J_2 + \sin\tau J_3\\\no
&&=(\tilde{P}(\hat{{\bf W}}(r_s,\cdot),\hat{W}_6)-P^-(r_s+\hat{W}_6,\cdot))\bigg(\tilde{\rho}(\hat{{\bf W}}(r_s,\cdot), \hat{W}_6)(\bar{U}(r_s+\hat{W}_6)+\hat{W}_1(r_s,\cdot))\hat{W}_a(r_s,\cdot)\\\no
&&\quad-(\rho^- U_1^- U_a^-)(r_s+\hat{W}_6,\cdot)\bigg)+\tilde{\rho}(\hat{{\bf W}}(r_s,\cdot), \hat{W}_6)\rho^-(r_s+\hat{W}_6,\cdot)\bigg(U_1^-(r_s+\hat{W}_6,\cdot)\hat{W}_{\tau}(r_s,\cdot)\\\no
&&\quad+U_{\tau}^-(r_s+\hat{W}_6,\cdot)\hat{W}_{1}(r_s,\cdot)\bigg)\times \underline{(U_3^-(r_s+\hat{W}_6,\cdot)\hat{W}_{2}(r_s,\cdot)-U_2^-(r_s+\hat{W}_6,\cdot)\hat{W}_{3}(r_s,\cdot))},
\ee
\be\no
&&J_{\tau}(r_s,a,\tau)=-\sin\tau J_2 + \cos\tau J_3\\\no
&&=(\tilde{P}(\hat{{\bf W}}(r_s,\cdot),\hat{W}_6)-P^-(r_s+\hat{W}_6,\cdot))\bigg(\tilde{\rho}(\hat{{\bf W}}(r_s,\cdot), \hat{W}_6)(\bar{U}(r_s+\hat{W}_6)+\hat{W}_1(r_s,\cdot))\hat{W}_{\tau}(r_s,\cdot)\\\no
&&\quad-(\rho^- U_1^- U_{\tau}^-)(r_s+\hat{W}_6,\cdot)\bigg)+\tilde{\rho}(\hat{{\bf W}}(r_s,\cdot), \hat{W}_6)\rho^-(r_s+\hat{W}_6,\cdot)\bigg(U_{a}^-(r_s+\hat{W}_6,\cdot)\hat{W}_{1}(r_s,\cdot)\\\no
&&\quad-U_1^-(r_s+\hat{W}_6,\cdot)\hat{W}_{a}(r_s,\cdot)\bigg)\times \underline{(U_3^-(r_s+\hat{W}_6,\cdot)\hat{W}_{2}(r_s,\cdot)-U_2^-(r_s+\hat{W}_6,\cdot)\hat{W}_{3}(r_s,\cdot))}.
\ee
Since $\hat{W}_a(r_s,1,\tau)=0$ for all $\tau\in \mathbb{T}_{2\pi}$ and $U^-_a(y_1,1,\tau)=0$ for all $(y_1,\tau)\in [r_1,r_2]\times \mathbb{T}_{2\pi}$, then
\be\no
&&(\hat{W}_2,\hat{W}_3)(r_s,1,\tau)=\hat{W}_{\tau}(r_s,1,\tau)(-\sin\tau,\cos\tau),\\\no &&(U_2^-,U_3^-)(r_s+\hat{W}_6(1,\tau),1,\tau)=U_{\tau}^-(r_s+\hat{W}_6(1,\tau),1,\tau)(-\sin\tau,\cos\tau).
\ee
Thus both the above underlined terms vanish at $a=1$. Note that
\be\no
&&\p_a(\tilde{P}(\hat{{\bf W}},\hat{W}_6))(r_s,1,\tau)=\p_a(\tilde{\rho}(\hat{{\bf W}}))(r_s,1,\tau)=0,\ \ \tau\in \mathbb{T}_{2\pi},\\\no
&&(\cos\tau g_2+\sin\tau g_{3})(r_s,1,\tau)=\b[(r_s+W_6)\frac{J_a}{J}-b_0r_s W_a\b](r_s,1,\tau)=0,\ \ \tau\in \mathbb{T}_{2\pi},
\ee
then the first and second equations in \eqref{j211} follow immediately. Applying $\p_a$ on \eqref{j10} and evaluating at $a=1$ yields
\be\no
&&\p_a J(r_s,1,\tau)=2\tilde{\rho}(\hat{{\bf W}}(r_s,\cdot), \hat{W}_6)\hat{W}_2\p_a \hat{W}_2(\tilde{\rho}\hat{W}_3^2+\tilde{P}-(\rho^- (U_3^-)^2+P^-)(r_s+\hat{W}_6,y'))\\\no
&&\quad+ 2\tilde{\rho}(\hat{{\bf W}}(r_s,\cdot), \hat{W}_6)\hat{W}_3\p_a \hat{W}_3(\tilde{\rho}\hat{W}_2^2+\tilde{P}-(\rho^- (U_2^-)^2+P^-)(r_s+\hat{W}_6,y'))\\\no
&&\quad-2(\tilde{\rho}\hat{W}_2\hat{W}_3-(\rho^- U_2^- U_3^-)(r_s+\hat{W}_6,y'))\tilde{\rho}(\hat{W}_3\p_a\hat{W}_2+\hat{W}_2\p_a\hat{W}_3)|_{a=1}\\\no
&&=\tilde{\rho}(\hat{{\bf W}}(r_s,\cdot), \hat{W}_6)(\tilde{P}(\hat{{\bf W}}(r_s,\cdot),\hat{W}_6)-P^-(r_s+\hat{W}_6,\cdot))\p_a(\hat{W}_a^2+\hat{W}_{\tau}^2)\\\no
&&\quad+2\tilde{\rho}(\hat{{\bf W}}(r_s,\cdot), \hat{W}_6)\underline{(U_3^-(r_s+\hat{W}_6,\cdot)\hat{W}_{2}(r_s,\cdot)-U_2^-(r_s+\hat{W}_6,\cdot)\hat{W}_{3}(r_s,\cdot))}\\\no
&&\quad\times \rho^-(r_s+\hat{W}_6,\cdot)(U_2^-(r_s+\hat{W}_6,\cdot)\p_a \hat{W}_3(r_s,\cdot)-U_3^-(r_s+\hat{W}_6,\cdot)\p_a \hat{W}_2(r_s,\cdot)) \bigg|_{a=1}\\\no
&&=0,\ \ \forall \tau\in \mathbb{T}_{2\pi}.
\ee
It remains to prove \eqref{rv11} and \eqref{rv41}. Here we only prove that
\be\label{rv110}
\p_a \{R_{01}(\hat{{\bf W}}(r_s,\cdot), \hat{W}_6)\}(1,\tau)=0,\ \ \forall \tau\in \mathbb{T}_{2\pi}.
\ee
The other cases can be proved similarly. We rewrite $R_{01}$ as
\be\no
&&R_{01}(\hat{{\bf W}}(r_s,\cdot), \hat{W}_6)=(\rho^- U_1^-)(r_s+\hat{W}_6,\cdot)-(\bar{\rho}^-\bar{U}^-)(r_s+\hat{W}_6)\\\no
&&-(\bar{\rho}^+(r_s+\hat{W}_6)-\bar{\rho}^+(r_s))\hat{W}_1(r_s,\cdot)-[\bar{\rho} \bar{U}](r_s+\hat{W}_6)\\\no
&&\quad-(\hat{W}_1(r_s,y')+\bar{U}^+(r_s+\hat{W}_6)-\bar{U}^+(r_s))(\tilde{\rho}(\hat{{\bf W}}(r_s,\cdot),\hat{W}_6)-\bar{\rho}^+(r_s+\hat{W}_6))\\\no
&&\quad+ \sum_{i=a,\tau}(\tilde{\rho}(\hat{{\bf W}}(r_s,\cdot),\hat{W}_6)\hat{W}_{i}(r_s,\cdot)-(\rho^- U_i^-)(r_s+\hat{W}_6,y'))\frac{J_i(\hat{{\bf W}}(r_s,\cdot),\hat{W}_6)}{J(\hat{{\bf W}}(r_s,\cdot),\hat{W}_6)}.
\ee
Then \eqref{rv110} will follow from \eqref{comp2},\eqref{cp100} and \eqref{j211}.

\framebox{{\bf The verification of \eqref{om1} and \eqref{om2}.}} By \eqref{vor402}, we see that
\be\no
&&g_4(\hat{{\bf W}}(r_s,\cdot),\hat{W}_6)=\frac{-\hat{W}_6}{r_s(r_s+\hat{W}_6)}(\frac{1+a^2}{2}(\p_{y_2}\hat{W}_3-\p_{y_3}\hat{W}_2)+y_3\hat{W}_2-y_2\hat{W}_3)\\\no
&&\quad-\frac{1+a^2}{2}\frac{r_2-r_s}{(r_s+\hat{W}_6)(r_2-r_s-\hat{W}_6)}(\p_{y_2}\hat{W}_6\p_{y_1}\hat{W}_3-\p_{y_3}\hat{W}_6\p_{y_1}\hat{W}_2),
\ee
then \eqref{cp100}, \eqref{ad1} and \eqref{ad2} imply that
\be\no
g_4(\hat{{\bf W}}(r_s,\cdot),\hat{W}_6)(r_s,1,\tau)=0,\ \ \forall \tau\in \mathbb{T}_{2\pi}.
\ee
The equation \eqref{om1} immediately follows from the following calculations:
\be\no
&&(1+|y'|^2)\left\{\p_{y_3}\left(\frac{g_2}{1+|y'|^2}\right)-\p_{y_2}\left(\frac{g_3}{1+|y'|^2}\right)\right\}(y_1,a,\tau)\bigg|_{a=1}\\\no
&&=\frac{1}{J}(J_2\p_{y_3}\hat{W}_6-J_3 \p_{y_2}\hat{W}_6)+\frac{r_s+\hat{W}_6}{J}(\p_{y_3} J_2-\p_{y_2} J_3-(y_3 J_2-y_2 J_3))\bigg|_{a=1}\\\no
&&\quad-\frac{r_s+W_6}{J^2}(J_2\p_{y_3} J-J_3 \p_{y_2} J)- b_0 r_s(\p_{y_3} \hat{W}_2-\p_{y_2} \hat{W}_3-(y_3 \hat{W}_2-y_2 \hat{W}_3))\bigg|_{a=1}\\\no
&&=\frac{1}{J}(-J_{\tau}\p_a \hat{W}_6+ J_a \p_{\tau} \hat{W}_6)+\frac{r_s+\hat{W}_6}{J}(-\p_a J_{\tau} + \frac1a \p_{\tau} J_a+(a-\frac1a) J_{\tau})\bigg|_{a=1}\\\no
&&\quad-\frac{r_s+\hat{W}_6}{J^2}(-J_{\tau}\p_a J+\frac1a J_a\p_{\tau} J)-b_0r_s(-\p_a \hat{W}_{\tau} + \frac1a \p_{\tau} \hat{W}_a+(a-\frac1a) \hat{W}_{\tau})\bigg|_{a=1}\\\no
&&=0.
\ee
To verify \eqref{om2}, we need to show that at $a=1$ there holds
\be\no
\bigg(D_2^{\hat{W}_6}\hat{W}_i D_3^{\hat{W}_6}\hat{W}_j- D_3^{\hat{W}_6}\hat{W}_i D_2^{\hat{W}_6}\hat{W}_j\bigg)(y_1,1,\tau)=0,\ \ \forall  i,j=1,4,5,6,a,\tau.
\ee
We compute the case $i=1,j=5$, the other cases are similar. Using \eqref{ad3}, evaluating at $a=1$ one has
\be\no
&&D_2^{\hat{W}_6}\hat{W}_1D_3^{\hat{W}_6}\hat{W}_5- D_3^{\hat{W}_6}\hat{W}_1 D_2^{\hat{W}_6}\hat{W}_5\bigg|_{a=1}=\p_{y_2}\hat{W}_1\p_{y_3}\hat{W}_5-\p_{y_3}\hat{W}_1\p_{y_2} \hat{W}_5\bigg|_{a=1} \\\no
&&\quad+\frac{y_1-r_2}{r_2-r_s-\hat{W}_6}\bigg(\p_{y_1}\hat{W}_5(\p_{y_2}\hat{W}_1\p_{y_3}\hat{W}_6-\p_{y_3}\hat{W}_1\p_{y_2}\hat{W}_6)
\\\no
&&+\p_{y_1}\hat{W}_1(\p_{y_2}\hat{W}_6\p_{y_3}\hat{W}_5-\p_{y_3}\hat{W}_6\p_{y_2}\hat{W}_5)\bigg)\bigg|_{a=1}\\\no
&&=\frac1a (\p_a\hat{W}_1\p_{\tau}\hat{W}_5- \p_a\hat{W}_5\p_{\tau}\hat{W}_1)\bigg|_{a=1}+\p_{y_1} \hat{W}_1(\p_a\hat{W}_6\p_{\tau}\hat{W}_5-\p_a\hat{W}_5\p_{\tau}\hat{W}_6)\bigg|_{a=1}\\\no
&&\quad+\frac{y_1-r_2}{a(r_2-r_s-\hat{W}_6)}\p_{y_1} \hat{W}_5(\p_a\hat{W}_1\p_{\tau}\hat{W}_6-\p_a\hat{W}_6\p_{\tau}\hat{W}_1)\bigg|_{a=1}=0.
\ee

\framebox{{\bf The verification of \eqref{g123}.}} Since
\be\no
&&H_1(\hat{{\bf W}},\hat{W}_6)=\bigg(\frac1{y_1}-\frac1{D_0^{\hat{W}_6}}\bigg)(\frac{1+a^2}{2}(\p_{y_2}\hat{W}_3-\p_{y_3}\hat{W}_2)+y_3\hat{W}_2-y_2\hat{W}_3)\\\no
&&\quad\quad -\frac{1+a^2}{2D_0^{\hat{W}_6}}\frac{r_2-y_1}{r_2-r_s-\hat{W}_6}(\p_{y_2}\hat{W}_6\p_{y_1}\hat{W}_3-\p_{y_3}\hat{W}_6\p_{y_1} \hat{W}_2),
\ee
then it follows from \eqref{om3},\eqref{ad1} and \eqref{ad2} that
\be\no
G_1(\hat{{\bf W}},\hat{W}_6)(y_1,1,\tau)=H_1(\hat{{\bf W}},\hat{W}_6)(y_1,1,\tau)=0,\ \ \forall (y_1,\tau)\in [r_s,r_2]\times \mathbb{T}_{2\pi}.
\ee
Furthermore, there holds
\be\no
&&G_{\tau}=\frac{\tilde{\omega}_1 \hat{W}_{\tau}}{\bar{U}(D_0^{\hat{W}_6})+\hat{W}_1}- \frac{1+a^2}{ D_0^{\hat{W}_6}(\bar{U}(D_0^{\hat{W}_6})+\hat{W}_1)}(\p_a\hat{W}_5+\frac{(y_1-r_2)\p_{y_1}\hat{W}_5\p_a \hat{W}_6}{r_2-r_s-\hat{W}_6})\\\no
&&\quad \quad-\sin\tau H_2+\cos\tau H_3-\frac{\bar{B}-\frac12 \bar{U}^2(y_1)}{\gamma \bar{K} \bar{U}(y_1)} \frac{(1+a^2)}{2y_1}\p_a \{R_4(\hat{{\bf W}},\hat{W}_6)\},\\\no
&&H_{\tau}=-\sin\tau H_2+\cos\tau H_3=-\frac{\hat{W}_6\p_{y_1}\hat{W}_a}{r_2-r_s-\hat{W}_6}- (\frac{1}{D_0^{\hat{W}_6}}-\frac1{y_1}) \hat{W}_a\\\no
&&+\frac{1+a^2}{2}\bigg(\frac{(y_1-r_2)\p_a \hat{W}_6\p_{y_1}\hat{W}_1}{D_0^{\hat{W}_6}(r_2-r_s-\hat{W}_6)}- \frac{(\bar{B}-\frac12 \bar{U}^2(y_1))\p_{a}\hat{W}_4}{\gamma \bar{K} y_1 \bar{U}(y_1)} + (\frac{1}{D_0^{\hat{W}_6}}-\frac1{y_1})\p_a \hat{W}_1\bigg)\\\no
&&\quad+\frac{(1+a^2)\{\hat{W}_5-\bar{U}(D_0^{\hat{W}_6})\hat{W}_1-\frac{1}{2}\sum_{j=1,a,\tau}\hat{W}_j^2\}}{2\gamma D_0^{\hat{W}_6}(\bar{U}(D_0^{\hat{W}_6})+\hat{W}_1)(\bar{K}+\hat{W}_4)}(\p_a\hat{W}_4+\frac{(y_1-r_2)\p_{y_1}\hat{W}_4\p_a \hat{W}_6}{r_2-r_s-\hat{W}_6})\\\no
&&\quad+\frac{(1+a^2)}{2aD_0^{\hat{W}_6}(\bar{U}(D_0^{\hat{W}_6})+\hat{W}_1)}\frac{\bar{B}-\frac12 \bar{U}^2(D_0^{\hat{W}_6})}{\gamma (\bar{K}+\hat{W}_4)}(\p_a\hat{W}_4+\frac{(y_1-r_2)\p_{y_1}\hat{W}_4\p_a \hat{W}_6}{r_2-r_s-\hat{W}_6}).
\ee
and
\be\no
&&G_a=\frac{\tilde{\omega}_1 \hat{W}_{a}}{\bar{U}(D_0^{\hat{W}_6})+\hat{W}_1}+ \frac{1+a^2}{2a D_0^{\hat{W}_6}(\bar{U}(D_0^{\hat{W}_6})+\hat{W}_1)}(\p_{\tau} \hat{W}_5+\frac{(y_1-r_2)\p_{\tau}\hat{W}_6 \p_{y_1}\hat{W}_5}{r_2-r_s-\hat{W}_6})\\\no
&&\q+\cos\tau H_2+\sin\tau H_3+\frac{\bar{B}-\frac12 \bar{U}^2(y_1)}{\gamma \bar{K} \bar{U}(y_1)} \frac{(1+a^2)}{2y_1 a}\p_{\tau}\{R_4(\hat{{\bf W}},\hat{W}_6)\},\\\no
&&H_a=\cos\tau H_2 +\sin\tau H_3\\\no
&&= \frac{\hat{W}_6\p_{y_1}\hat{W}_{\tau}}{r_2-r_s-\hat{W}_6}+(\frac{1}{D_0^{\hat{W}_6}}-\frac1{y_1})\hat{W}_{\tau}-\frac{1+a^2}{2 a D_0^{\hat{W}_6}}\frac{(y_1-r_2)\p_{\tau}\hat{W}_6 \p_{y_1}\hat{W}_1}{r_2-r_s-\hat{W}_6}\\\no
&&\quad-\frac{(1+a^2)(\hat{W}_5-\bar{U}(D_0^{\hat{W}_6})\hat{W}_1-\frac{1}{2}\sum_{j=1,a,\tau}\hat{W}_j^2)} {2\gamma a D_0^{\hat{W}_6}(\bar{U}(D_0^{\hat{W}_6})+\hat{W}_1) (\bar{K}+\hat{W}_4)}(\p_{\tau}\hat{W}_4+\frac{(y_1-r_2)\p_{\tau}\hat{W}_6\p_{y_1}\hat{W}_4}{r_2-r_s-\hat{W}_6})\\\no
&&\quad-\frac{1+a^2}{2a}(\frac{1}{D_0^{\hat{W}_6}}-\frac1{y_1})\p_{\tau}\hat{W}_1+ \frac{\bar{B}-\frac12 \bar{U}^2(y_1)}{\gamma \bar{K} \bar{U}(y_1)} \frac{1}{2y_1}\frac{1+a^2}{a}\p_{\tau} \hat{W}_4\\\no
&&\quad-\frac{(1+a^2)(\bar{B}-\frac12 \bar{U}^2(D_0^{\hat{W}_6}))}{2\gamma a D_0^{\hat{W}_6}(\bar{U}(D_0^{\hat{W}_6})+\hat{W}_1)(\bar{K}+\hat{W}_4)}(\p_{\tau} \hat{W}_4+\frac{(y_1-r_2)\p_{\tau} \hat{W}_6}{r_2-r_s-\hat{W}_6}\p_{y_1}\hat{W}_4).
\ee
Then \eqref{g123} will follow from a direct computation using the compatibility condition \eqref{cp100} and \eqref{ad4}.

\framebox{{\bf The verification of \eqref{g5cp}.}} By definition and \eqref{d350}, one has
\be\no
\p_a G_5(y_1,1,\tau)=\p_a \{G_0(\hat{{\bf W}},\hat{W}_6)\}(y_1,1,\tau),\ \forall (y_1,\tau)\in\Sigma_w.
\ee

In the $(y_1,a,\tau)$ coordinates, one has
\be\no
&&G_0(\hat{{\bf W}},\hat{W}_6)\\\no
&&= \mathbb{F}(\hat{{\bf W}},\hat{W}_6)-\b(d_1(D_0^{\hat{W}_6})D_1^{\hat{W}_6} \hat{W}_1-d_1(y_1)\p_{y_1} \hat{W}_1\b)-(d_2(D_0^{\hat{W}_6})-d_2(y_1)) \hat{W}_1\\\no
&&\quad-2\bigg(\frac{1}{D_0^{\hat{W}_6}}-\frac{1}{y_1}\bigg)\hat{W}_1-\frac{1+a^2}{2D_0^{\hat{W}_6}}\frac{(y_1-r_2)(\p_{y_1}\hat{W}_a\p_a +\frac{1}{a}\p_{y_1}\hat{W}_{\tau}\p_{\tau})\hat{W}_6}{r_2-r_s-\hat{W}_6}\\\label{g01}
&&\quad-\bigg(\frac{1}{D_0^{\hat{W}_6}}-\frac{1}{y_1}\bigg)\bigg\{\frac{1+a^2}{2}(\p_a \hat{W}_a+\frac1a\hat{W}_a+\frac1a\p_{\tau}\hat{W}_{\tau})-a\hat{W}_a\bigg\}
\ee
and
\be\no
&&\bar{c}^2(D_0^{\hat{W}_6})\mathbb{F}(\hat{{\bf W}},\hat{W}_6)=-\bigg((\gamma-1)\hat{W}_5-(\gamma+1)\bar{U}(D_0^{\hat{W}_6})\hat{W}_1\bigg)D_{1}^{\hat{W}_6} \hat{W}_1\\\no
&&\quad-\frac{(\gamma-1)}{D_0^{\hat{W}_6}}\bigg(\hat{W}_5-\bar{U}(D_0^{\hat{W}_6})\hat{W}_1-\frac{1}{2}\sum_{j=1,a,\tau} W_j^2\bigg)\times\bigg(2\hat{W}_1-a\hat{W}_a\\\no
&&\quad+\frac{1+a^2}{2}\bigg\{\p_a \hat{W}_a+\frac{\hat{W}_a} a+\frac1a\p_{\tau}\hat{W}_{\tau}+\frac{y_1-r_2}{r_2-r_s-\hat{W}_6}(\p_{y_1}\hat{W}_a\p_a +\frac{\p_{y_1}\hat{W}_{\tau}}a\p_{\tau})\hat{W}_6\bigg\} \bigg)\\\no
&&\quad+ \frac{\bar{U}'(D_0^{\hat{W}_6})+D_1^{\hat{W}_6}\hat{W}_1}{2}\big(\hat{W}_1^2+(\gamma-1)\sum_{j=1,a,\tau}\hat{W}_j^2)\big)+\frac{(\gamma-1)\bar{U}(D_0^{\hat{W}_6})}{ D_0^{\hat{W}_6}}\sum_{j=1,a,\tau} \hat{W}_j^2\\\no
&&\quad+ (\bar{U}(D_0^{\hat{W}_6})+\hat{W}_1)\sum_{j=a,\tau} \hat{W}_j D_1^{\hat{W}_6} \hat{W}_j + \frac{(1+a^2)(\bar{U}(D_0^{\hat{W}_6})+\hat{W}_1)}{2D_0^{\hat{W}_6} }\\\no
&&\quad\quad\quad\times \bigg\{(\hat{W}_a\p_a +\frac{\hat{W}_{\tau} } a\p_{\tau})\hat{W}_1+\frac{(y_1-r_2)\p_{y_1}\hat{W}_1}{r_2-r_s-\hat{W}_6}(\hat{W}_a\p_a+a^{-1}\hat{W}_{\tau}\p_{\tau}) \hat{W}_6\bigg\}\\\no
&&\q+ \frac{(1+a^2)}{4D_0^{\hat{W}_6}}\b\{(\hat{W}_a\p_a+a^{-1}\hat{W}_{\tau}\p_{\tau})(\hat{W}_a^2+\hat{W}_{\tau}^2)\\\no
&&\quad\quad\quad+\frac{(y_1-r_2)\p_{y_1}(\hat{W}_a^2+\hat{W}_{\tau}^2)}{r_2-r_s-\hat{W}_6}(\hat{W}_a\p_a+a^{-1}\hat{W}_{\tau}\p_{\tau}) \hat{W}_6\bigg\}.
\ee
Utilizing the compatibility condition \eqref{cp100}, a direct calculation yields \eqref{g5cp}. We remark that $(\p_a^2 \hat{W}_a+\p_a \hat{W}_a)(y_1,1,\tau)=0$ is used to verify the last term in \eqref{g01}.

{\bf Acknowledgment.} Weng is supported by National Natural Science Foundation of China (Grants No. 12071359, 12221001). The author would like to express his sincere gratitude to Prof. Zhouping Xin for his insightful discussions, comments and constant supports. 

{\bf Data Availability Statement.} No data, models or code were generated or used during the study.

{\bf Conflict of interest.} The author states that there is no conflict of interests.


\end{document}